\documentclass[11pt]{article}

\usepackage{amsmath, amsthm, amssymb}
\usepackage{enumitem}
\usepackage{pdflscape}
\usepackage{caption}
\usepackage{bm}

\usepackage{ifpdf}
\ifpdf
\usepackage[pdftex]{graphicx}
\else
\usepackage[dvips]{graphicx}
\fi
\usepackage{tikz}
   \usetikzlibrary{arrows,backgrounds}
\usepackage[all]{xy}

\usepackage{multicol}

\usepackage{tocvsec2}

\usepackage{bbm}

\input xy
\xyoption{all}

\usepackage[pdftex,plainpages=false,hypertexnames=false,pdfpagelabels]{hyperref}
\newcommand{\arxiv}[1]{\href{http://arxiv.org/abs/#1}{\tt arXiv:\nolinkurl{#1}}}
\newcommand{\arXiv}[1]{\href{http://arxiv.org/abs/#1}{\tt arXiv:\nolinkurl{#1}}}
\newcommand{\doi}[1]{\href{http://dx.doi.org/#1}{{\tt DOI:#1}}}

\newcommand{\googlebooks}[1]{(preview at \href{http://books.google.com/books?id=#1}{google books})}

\usepackage{xcolor}
\definecolor{dark-red}{rgb}{0.7,0.25,0.25}
\definecolor{dark-blue}{rgb}{0.15,0.15,0.55}
\definecolor{medium-blue}{rgb}{0,0,.8}
\definecolor{shaded-blue}{RGB}{98,140,255}
\definecolor{DarkGreen}{RGB}{0,150,0}
\definecolor{rho}{named}{red}
\definecolor{Salmon}{RGB}{255, 144, 144}
\hypersetup{
   colorlinks, linkcolor={purple},
   citecolor={medium-blue}, urlcolor={medium-blue}
}

\usepackage{longtable}
\usepackage{fullpage}

\theoremstyle{plain}
\newtheorem{thm}{Theorem}[section]
\newtheorem*{thm*}{Theorem}
\newtheorem{thmalpha}{Theorem}

\newtheorem{cor}[thm]{Corollary}
\newtheorem{coralpha}[thmalpha]{Corollary}
\newtheorem*{cor*}{Corollary}

\newtheorem*{conj*}{Conjecture}
\newtheorem{lem}[thm]{Lemma}
\newtheorem{fact}[thm]{Fact}

\newtheorem{prop}[thm]{Proposition}

\newtheorem*{quest*}{Question}
\newtheorem*{claim*}{Claim}

\theoremstyle{definition}
\newtheorem{defn}[thm]{Definition}

\newtheorem{alg}[thm]{Algorithm}

\newtheorem{nota}[thm]{Notation}

\newtheorem{ex}[thm]{Example}

\newtheorem{sub-ex}[thm]{Sub-Example}
\newtheorem{rem}[thm]{Remark}
\newtheorem*{rem*}{Remark}
\newtheorem{remark}[thm]{Remark}

\DeclareMathOperator{\coev}{coev}

\DeclareMathOperator{\End}{End}
\DeclareMathOperator{\ev}{ev}
\DeclareMathOperator{\Hom}{Hom}

\DeclareMathOperator{\spann}{span}

\DeclareMathOperator{\id}{id}

\DeclareMathOperator{\im}{im}

\DeclareMathOperator{\Tr}{Tr}
\DeclareMathOperator{\tr}{tr}

\newcommand{\comment}[1]{}

\newcommand{\be}{\begin{enumerate}[label=(\arabic*)]}
\newcommand{\ee}{\end{enumerate}}

\newcommand{\set}[2]{\left\{#1 \middle| #2\right\}}

\newcommand{\alttens}[1][n]{{\text{alt}\otimes #1}}

\newcommand{\Xalt}{X^{\alttens}}

\def\semicolon{;}
\def\applytolist#1{
    \expandafter\def\csname multi#1\endcsname##1{
        \def\multiack{##1}\ifx\multiack\semicolon
            \def\next{\relax}
        \else
            \csname #1\endcsname{##1}
            \def\next{\csname multi#1\endcsname}
        \fi
        \next}
    \csname multi#1\endcsname}

\def\calc#1{\expandafter\def\csname c#1\endcsname{{\mathcal #1}}}
\applytolist{calc}QWERTYUIOPLKJHGFDSAZXCVBNM;
\def\bbc#1{\expandafter\def\csname bb#1\endcsname{{\mathbb #1}}}
\applytolist{bbc}QWERTYUIOPLKJHGFDSAZXCVBNM;
\def\bfc#1{\expandafter\def\csname bf#1\endcsname{{\mathbf #1}}}
\applytolist{bfc}QWERTYUIOPLKJHGFDSAZXCVBNM;
\def\sfc#1{\expandafter\def\csname s#1\endcsname{{\sf #1}}}
\applytolist{sfc}QWERTYUIOPLKJHGFDSAZXCVBNM;
\def\fc#1{\expandafter\def\csname f#1\endcsname{{\mathfrak #1}}}
\applytolist{fc}QWERTYUIOPLKJHGFDSAZXCVBNM;

\newcommand{\Cstar}{\rm{C^*}}
\newcommand{\Wstar}{\rm{W^*}}

\newcommand{\noshow}[1]{}
\newcommand{\MR}[1]{}

\newcommand{\TLJ}{\cT\cL\cJ}

\usetikzlibrary{shapes}
\usetikzlibrary{cd}
\usetikzlibrary{backgrounds}
\usetikzlibrary{decorations,decorations.pathreplacing,decorations.markings}
\usetikzlibrary{fit,calc,through}
\usetikzlibrary{external}
\usetikzlibrary{arrows}
\tikzset{vertex/.style = {shape=circle,draw,fill=black,inner sep=0pt,minimum size=5pt}}
\tikzset{edge/.style = {->,> = latex', bend right}}
\tikzset{
 super thick/.style={line width=3pt}
}
\tikzset{
    quadruple/.style args={[#1] in [#2] in [#3] in [#4]}{
        #1,preaction={preaction={preaction={draw,#4},draw,#3}, draw,#2}
    }
}
\tikzstyle{shaded}=[fill=gray!25!white]
\tikzstyle{shadedpink}=[left color= white, right color = Salmon]
\tikzstyle{unshaded}=[fill=white]
\tikzstyle{empty box}=[circle, draw, thick, fill=white, opaque, inner sep=2mm]
\tikzstyle{annular}=[scale=.7, inner sep=1mm, baseline]
\tikzstyle{rectangular}=[scale=.75, inner sep=1mm, baseline=-.1cm]
\tikzstyle{mid>}=[decoration={markings, mark=at position 0.5 with {\arrow{>}}}, postaction={decorate}]
\tikzstyle{mid<}=[decoration={markings, mark=at position 0.5 with {\arrow{<}}}, postaction={decorate}]
\tikzstyle{over}=[double, draw=white, super thick, double=]
\tikzstyle{relativecommutantshading}=[fill=blue!20!white]
\tikzstyle{ctwoshading}=[fill=red!15!white]
\tikzstyle{Bshading}=[fill=blue!15!white]

\newcommand{\roundNbox}[6]{
 \draw[rounded corners=5pt, very thick, #1] ($#2+(-#3,-#3)+(-#4,0)$) rectangle ($#2+(#3,#3)+(#5,0)$);
 \coordinate (ZZa) at ($#2+(-#4,0)$);
 \coordinate (ZZb) at ($#2+(#5,0)$);
 \node at ($1/2*(ZZa)+1/2*(ZZb)$) {#6};
}

\newcommand{\ncircle}[5]{
 \draw[very thick, #1] #2 circle (#3);
 \node at #2 {#5};
 \node at ($#2+(#4:.15cm)+(#4:#3cm)$) {\scriptsize{$\star$}};
}

\newcommand{\halfcircle}[5]{
 \draw[very thick, #1] ($ #2 + (0,#3) $) -- ($ #2 + (0,#3) + (#4,0) $) arc (90:-90:#3cm) -- ($ #2 - (0,#3) $) -- ($ #2 + (0,#3) + (0,.0211) $); 
 \node at ($ #2 + .5*(#3,0) +.5*(#4,0) $) {#5};
}

\newcommand{\openhalfcircle}[5]{
 \draw[very thick, fill=white, #1] ($ #2 + (0,#3) $) -- ($ #2 + (#4,0) + (#3,#3) $) arc (90:-90:#3cm) -- ($ #2 - (0,#3) $) -- ($ #2 + (0,#3) $); 
 \draw[super thick, white] ($ #2 + (0,#3) - (0,.0211) $) -- ($ #2 - (0,#3) + (0,.0211) $);
 \node at ($ #2 + .5*(#3,0) +.5*(#4,0) $) {#5};
}

\title{The module embedding theorem via towers of algebras}
\author{Desmond Coles, Peter Huston, David Penneys, and Srivatsa Srinivas}

\usepackage[utf8]{inputenc}
\begin{document}


\maketitle
\begin{abstract}
Jones and Penneys showed that a finite depth subfactor planar algebra embeds in the bipartite graph planar algebra of its principal graph, via a Markov towers of algebras approach.
 We relate several equivalent perspectives on the notion of module over a subfactor planar algebra, and show that a Markov tower is equivalent to a module over the Temperley-Lieb-Jones planar algebra.
 As a corollary, we obtain a classification of semisimple pivotal $\Cstar$ modules over Temperley-Lieb-Jones in terms of pointed graphs with a Frobenius-Perron vertex weighting.
 We then generalize the Markov towers of algebras approach to show that a finite depth subfactor planar algebra embeds in the bipartite graph planar algebra of the fusion graph of any of its cyclic modules.
\end{abstract}

\section{Introduction}

Jones' planar algebras \cite{math.QA/9909027} are a powerful method to construct \cite{MR2979509,1810.06076} and classify \cite{MR3166042,MR3345186,1509.00038} finite index ${\rm II}_1$ subfactors.
Many exotic examples have been constructed via \emph{graph planar algebra embedding}, i.e., by finding \emph{evaluable} planar subalgebras of graph planar algebras.
By \cite{MR2812459}, any finite depth subfactor planar algebra embeds in the graph planar algebra of its principal graph.
This result also extends to infinite depth subfactor planar algebras by \cite{gpa}.

To date, graph planar algebra embedding has been used to construct:
\begin{itemize}
\item
the $E_6$ and $E_8$ subfactor planar algebras \cite{MR1929335},
\item
group planar algebras \cite{MR2511128},
\item
Haagerup-Izumi quadratic subfactors \cite{MR2679382,MR2822034,MR3314808,MR3394622,MR3402358},
\item
quantum group planar algebras \cite{MR3306607}, and
\item
the extended Haagerup fusion categories \cite{MR2979509,1810.06076}.
\end{itemize}
While none of the constructions above rely on the embedding theorem from \cite{MR2812459},
the embedding theorem gives us the motivation to do the hard work of looking for the embedding.
However, the embedding theorem is necessary for Liu's important classification theorem for composites of $A_3$ and $A_4$ subfactor planar algebras \cite{MR3345186}, in which he shows that higher quotients of $A_3*A_4$ do not exist because the possible generator does not embed in the appropriate graph planar algebra.

As noted in \cite{MR2812459}, it was rather surprising that the dual principal graph made no appearance in the embedding theorem.
Adding to this mystery, certain examples above could be constructed by embedding into planar algebras of bipartite graphs which are completely different from the principal and dual principal graphs \cite{MR2679382,MR3402358,1810.06076}.
The answer to why this occurs is the following theorem:

\begin{thm}[\cite{1810.06076}]
Let $\cP_\bullet$ be a finite depth subfactor planar algebra and $\cC$ its unitary $2\times 2$ multifusion category of projections with generator $X \in \cP_{1,+}$, the unshaded-shaded strand.
Endow $\cC$ with the canonical spherical structure inherited from $\cP_\bullet$.
There is a bijective correspondence between:
\begin{itemize}
\item
planar $\dag$-algebra embeddings $\cP_\bullet \to \cG\cP\cA(\Gamma)_\bullet$, where $\Gamma$ is a finite connected bipartite graph, and 
\item
indecomposable finitely semisimple pivotal left $\cC$-module $\Cstar$ categories $\cM$ whose fusion graph with respect to $X$ is $\Gamma$.
\end{itemize}
\end{thm}

The proof in \cite{1810.06076} is mostly in the language of tensor and module categories.
In this article, we provide an independent proof in the original towers of algebras approach to subfactor theory \cite{MR0696688,MR936086,MR999799,MR1278111} and the graph planar algebra embedding theorem \cite{MR2812459}.

Our starting point is the well-known correspondence between:
\begin{enumerate}[label={\rm(\arabic*)}]
\item
unitary $2\times 2$ multitensor categories $\cC$ with orthogonal decomposition into simples $1_\cC = 1_0 \oplus 1_1$ and generator $X = 1_0 \otimes X \otimes 1_1$ with its canonical spherical/balanced unitary dual functor
(see \cite{MR2091457,1808.00323}),\footnote{
 That $X$ is a generator means that any proper full subcategory of $\cC$ containing $X$ which is closed under tensor product, direct sum, taking dual, and taking subobjects is equivalent to $\cC$; see Definition \ref{def:multitensor}. 
} 
and
\item
Jones' subfactor planar algebras $\cP_\bullet$ \cite{math.QA/9909027}.
\end{enumerate}
In \S\ref{sec:Modules}, we build on this correspondence by defining analogous notions of \emph{right modules} for these algebraic objects.
We briefly describe these objects here, and we refer the reader to \S\ref{sec:3TypesOfModules} for more details.

A \emph{pivotal} module category for $\cC$ is a finitely semisimple $\Cstar$ category $\cM$ which is an indecomposable right $\cC$-module category equipped with a faithful positive trace $\Tr^\cM_m$ on each endomorphism $\Cstar$ algebra $\cM(m\to m)$ which is compatible with the right $\cC$-action \cite{MR3019263,1810.06076}.
That is, for all $m\in \cM$, $c\in \cC$, and $f\in \End_\cM(m \vartriangleleft c)$, 
$$
\Tr^\cM_{m\vartriangleleft c}(f)
=
\Tr^\cM_{m}(
(\id_m \vartriangleleft\coev_c^\dag) \circ (f\vartriangleleft\id_{\overline{c}}) \circ (\id_m \vartriangleleft \coev_c)
),
$$
where $(\overline{c}, \ev_c, \coev_c)$ is the canonical \emph{balanced} dual of $c\in \cC$ \cite{MR1444286,MR2091457,MR3342166,1808.00323}.

A \emph{(connected) right planar module} $\cM_\bullet$ for a subfactor planar algebra $\cP_\bullet$ is a sequence of finite dimensional von Neumann algebras $(\cM_k)_{k\geq 0}$ with $\dim(\cM_0) = 1$,\footnote{
The adjective \emph{connected} refers to the condition that $\dim(\cM_0) = 1$.
}
together with an action of the 
shaded planar module operad, 
which is a variation of Voronov's Swiss cheese operad \cite{MR1718089}.
We refer the reader to Definition \ref{def:PlanarModule} for the details, but we include a representative tangle below which acts amongst the algebras $\cM_k$ and the box spaces $\cP_{n,\pm}$:\,\footnote{ 
We use the convention that all $\cM_k$ appear before $\cP_{n,\pm}$ in the tensor product; this is not problematic, as the tensor category of finite dimensional complex vector spaces is symmetric.
}
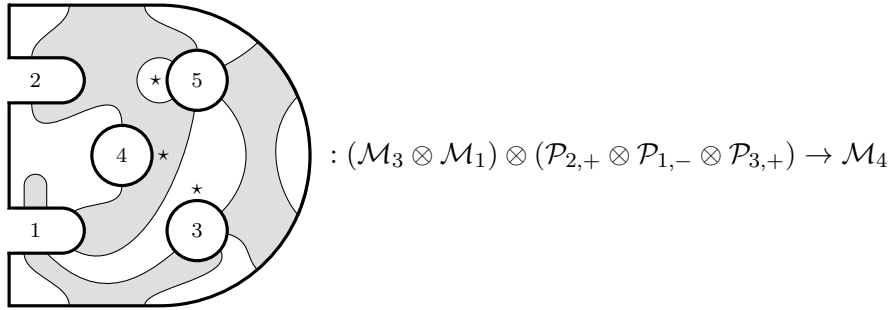
$$
\begin{tikzpicture}[baseline = -.1cm]
 \filldraw[shaded] (.3,1.3) .. controls ++(90:.3cm) and ++(270:.3cm) .. (.8,2) -- 
  (1.8,2) .. controls ++(270:.3cm) and ++(90:.3cm) .. (2.5,1.4) --
  (2.5,1) .. controls ++(270:1cm) and ++(-30:1cm) .. (.85,-1.25) --
  (.85,-.75) .. controls ++(30:.3cm) and ++(270:.3cm) .. (1.5,-.4) -- 
  (1.5,.4) .. controls ++(90:.8cm) and ++(270:.8cm) .. (.3,.7);
 \filldraw[unshaded] (2,1) circle (.3cm);
 \filldraw[shaded] (.2,-.7) -- (.2,-.4) arc (180:0:.15cm) -- (.5,-.7);
 \filldraw[shaded]  (.2,-1.3) .. controls ++(270:.3cm) and ++(90:.3cm) .. (.8,-2) -- (1.8,-2)
  (1.8,-2) .. controls ++(90:.3cm) and ++(-30:1.2cm) .. (2.5,-1)
  .. controls ++(-135:1cm) and ++(-45:1cm) .. (.5,-1.3) -- (.2,-1.3);
 \filldraw[shaded] (2.5,1) .. controls ++(30:.3cm) and ++(-135:.3cm) .. (3.3,1.5)
  arc (49:-49:2cm)  .. controls ++(135:.3cm) and ++(0:.8cm) ..  (2.5,-1)
  .. controls ++(30:1cm) and ++(-30:1cm) .. (2.5,1);
 \filldraw[unshaded] (3.85,.8) .. controls ++(-135:.8cm) and ++(135:.3cm) .. (3.85,-.8) arc (-23:23:2cm);
 \halfcircle{}{(0,0)}{2}{2}{}
 \openhalfcircle{}{(0,1)}{.3}{.4}{\scriptsize{$2$}}
 \openhalfcircle{}{(0,-1)}{.3}{.4}{\scriptsize{$1$}}
 \ncircle{unshaded}{(2.5,1)}{.4}{180}{\scriptsize{$5$}}
 \ncircle{unshaded}{(1.5,0)}{.4}{0}{\scriptsize{$4$}}
 \ncircle{unshaded}{(2.5,-1)}{.4}{90}{\scriptsize{$3$}}
\end{tikzpicture}
\,:
(\cM_3 \otimes \cM_1) \otimes (\cP_{2,+}\otimes \cP_{1,-} \otimes \cP_{3,+}) \to \cM_{4}
$$
Here, one can glue shaded planar module tangles into the module input semidisks, and one can glue shaded planar tangles into the circular input disks.
In addition, the tower of algebras $(\cM_k)_{k \geq 0}$ must satisfy that multiplication in the von Neumann algebra $\cM_k$ is given by the tangle
$$
\begin{tikzpicture}[baseline = -.1cm]
 \draw (.25,-1) -- (.25,1);
 \halfcircle{}{(0,0)}{1}{.3}{}
 \openhalfcircle{}{(0,.4)}{.25}{.25}{\scriptsize{$2$}}
 \openhalfcircle{}{(0,-.4)}{.25}{.25}{\scriptsize{$1$}}
 \node at (.4,0) {\scriptsize{$k$}};
 \node at (.4,.8) {\scriptsize{$k$}};
 \node at (.4,-.8) {\scriptsize{$k$}};
\end{tikzpicture}
:
\cM_k \otimes \cM_k \to \cM_k,
$$
and the $*$-structure on $\cM_k$ is compatible with the reflection of tangles about a horizontal axis.
Notice this canonically identifies $\cM_0 = \bbC$ as a von Neumann algebra.
Under this identification, we require that the linear functionals
$$
\tr_k:=
d^{-k}\cdot\,
\begin{tikzpicture}[baseline = -.1cm]
 \draw (.25,.25) arc (180:0:.2cm) -- (.65,-.25) arc (0:-180:.2cm);
 \halfcircle{}{(0,0)}{.8}{.3}{}
 \openhalfcircle{}{(0,0)}{.25}{0}{}
 \node at (.8,0) {\scriptsize{$k$}};
\end{tikzpicture}
:
\cM_k \to \cM_0 = \bbC
$$
are faithful positive normalized traces, where $d$ is the loop parameter of $\cP_\bullet$.

The following theorem generalizes the correspondence between
unitary $2\times 2$ multitensor categories $\cC$ with $1_\cC =1_0\oplus 1_1$ and generator $X = 1_0\otimes X \otimes 1_1$ 
and
subfactor planar algebras $\cP_\bullet$.

\begin{thmalpha}
\label{thm:ModuleEquivalence}
Let $\cP_\bullet$ be a subfactor planar algebra corresponding to $(\cC, X)$ as above.
There is an equivalence between:
\begin{enumerate}[label={\rm(\arabic*)}]
\item
pivotal right $\cC$-module $\Cstar$ categories $(\cM,\Tr^\cM)$ with choice of simple basepoint $m = m\vartriangleleft 1_0$, and
\item
connected right planar modules $\cM_\bullet$ for $\cP_\bullet$.
\end{enumerate}
\end{thmalpha}

One passes from (2) to (1) in Theorem \ref{thm:ModuleEquivalence} by taking the category of projections, similar to the correspondence between $\cP_\bullet$ and $(\cC, X)$ \cite{MR2559686,MR3405915}.
One passes from (1) to (2) using the diagrammatic calculus for module categories, similar to how one gets a subfactor planar algebra from $(\cC, X)$ via the diagrammatic calculus for pivotal categories \cite{MR2811311,1808.00323}.

From a pivotal semisimple right $\cC$-module $\Cstar$ category $(\cM, \Tr^\cM)$ together with a choice of simple basepoint $m\in \cM$ with $m= m\vartriangleleft 1_0$, we build a 
tower of finite dimensional von Neumann algebras by setting
$$
M_n := \End_\cM(m \vartriangleleft  \underbrace{X\otimes \overline{X} \otimes \cdots \otimes X^?}_{n \text{ tensorands}})
$$
where for our generator $X\in \cC$, we set $X^? = X$ if $n$ is even and $X^? = \overline{X}$ if $n$ is odd.
The trace $\Tr^\cM$ endows each von Neumann algebra $M_n$ with a faithful tracial state $\tr_n := \Tr^\cM(\id_n)^{-1} \Tr^\cM$ together with canonical Jones projections $e_n \in M_{n+1}$ for all $n\geq 1$.
Based on the parity of $n$, the $e_n$ are defined for $k\geq 0$ by
\begin{align*}
e_{2k+1}
&=
\begin{tikzpicture}[baseline]
 \fill[ctwoshading] (-.3,-.5) rectangle (-.7,0.5);
 \draw[thick,red] (-0.3,-.5) -- (-0.3,0.5);
 \draw (0,-.5) -- (0,.5);
 \filldraw[shaded] (.3,.5) arc (-180:0:.3cm);
 \filldraw[shaded] (.3,-.5) arc (180:0:.3cm);
 \node at (.2,0){\tiny{$2k$}};
\end{tikzpicture} 
\hspace{.3cm}
:=
d^{-1}
\big(
\id_m \vartriangleleft \id_{(X\otimes \overline{X})^{\otimes k}} \otimes (\coev_X \circ \coev_X^\dag)
\big)
\in M_{2k+2}
\\
e_{2k+2}
&=
\begin{tikzpicture}[baseline]
 \fill[ctwoshading] (-.3,-.5) rectangle (-.7,0.5);
 \fill[shaded] (0,-.5) rectangle (1.2,.5);
 \draw[thick,red] (-0.3,-.5) -- (-0.3,0.5);
 \draw (0,-.5) -- (0,.5);
 \filldraw[unshaded] (.3,.5) arc (-180:0:.3cm);
 \filldraw[unshaded] (.3,-.5) arc (180:0:.3cm);
 \node at (.4,0){\tiny{$2k+1$}};
\end{tikzpicture} 
:=
d^{-1}
\big(
\id_m \vartriangleleft \id_X \otimes \id_{(\overline{X}\otimes X)^{\otimes k}} \otimes (\ev_X^\dag \circ \ev_X)
\big)
\in M_{2k+3}.
\end{align*}
Here, $(\overline{X}, \ev_X, \coev_X)$ is the balanced dual of $X$, $m$ is graphically represented by a red strand, and the left hand side of $m$ is shaded red to denote the absence of a left $\cC$-action.

We call $M_\bullet = (M_n, \tr_n, e_{n+1})_{n\geq 0}$ 
a \emph{Markov tower} as it satisfies the following axioms:
\begin{enumerate}[label={\rm(M\arabic*)}]
\item
The projections $(e_n)$ satisfy the Temperley-Lieb-Jones relations with \emph{modulus} $d>0$ (our convention for $d$ is $e_ie_{i\pm 1}e_i = d^{-2} e_i$.)
\item
For all $x\in M_n$, $e_n xe_n = E_n(x) e_n$, where $E_n: M_n \to M_{n-1}$ is the canonical trace-preserving conditional expectation.
\item
For all $n\geq 1$, $E_{n+1}(e_n) = d^{-2}$.
\item
For all $n\geq 1$, we have the Pimsner-Popa pull down property \cite{MR860811}: $M_{n+1} e_n = M_n e_n$, which is equivalent to $M_n e_n M_n$ being a 2-sided ideal in $M_{n+1}$.
\end{enumerate}
One should view a Markov tower as an analog of Popa's $\lambda$-lattices \cite{MR1334479} where we only have one tower of algebras rather than a tower/lattice of commuting squares.
Indeed, one should compare \ref{eq:MarkovJonesProjections} and \ref{eq:MarkovImplement} with (1.3.2) and \ref{eq:MarkovIndex} and \ref{eq:MarkovPullDown} with (1.3.3') from \cite{MR1334479} respectively.
We expect that the notion of Markov tower with some compatibility axioms is the correct notion of a right module for Popa's $\lambda$-lattices (see Remark \ref{rem:LambdaLatticeModule}).
We leave this exploration to a future article as it would take us too far afield.

Markov towers satisfy many nice properties exhibited by standard invariants of finite index ${\rm II}_1$ subfactors from \cite[Ch.~4]{MR999799}; we mention a few here, and we refer the reader to \S\ref{sec:MarkovTowers} for more details.
The traces satisfy the \emph{Markov property} $\tr_{n+2}(x e_n) = d^{-2}\tr_{n+1}(x)$ for every $x\in M_{n+1}$, and the Markov tower has a \emph{principal graph} consisting of the non-reflected part of the Bratteli diagram at each step.
The tower is called \emph{finite depth} if the principal graph is finite.

From a Markov tower, we can form a semisimple $\Cstar$ \emph{projection category} $\cM$, whose simple objects are in canonical bijection with the vertices of the principal graph.
Moreover, the traces and Jones projections canonically equip $\cM$ with the structure of a pivotal right $\cT\cL\cJ(d)$-module $\Cstar$ category.
Now any pointed bipartite graph $(\Gamma,v)$ with a \emph{quantum dimension function} on vertices $\dim: V(\Gamma) \to \bbR_{>0}$ satisfying
$$
d\cdot \dim(v) = \sum_{w\sim v} \dim(w)
$$
gives us a Markov tower of modulus $d$, where we write $w\sim v$ to mean $w$ is connected to $v$, and the sum is taken with multiplicity.
We thus get the following corollary, which should be compared with the non-pivotal case in \cite{MR3420332}.

\begin{coralpha}
\label{cor:TLJPivotalModuleClassification}
Equivalence classes of pivotal $\cT\cL\cJ(d)$-module $\Cstar$ categories with simple basepoint
are in bijection with 
pointed connected bipartite graphs $(\Gamma, v)$ with a quantum dimension function.
\end{coralpha}

We now specialize to the hypotheses of the module embedding theorem, i.e., $\cQ_\bullet$ is a finite depth subfactor planar algebra, $(\cC, X)$ is its corresponding spherical unitary multifusion category of projections with generator $X= 1_0 \otimes X \otimes 1_1$ the unshaded-shaded strand, and $(\cM, \Tr^\cM, m)$ is a pivotal right $\cC$-module $\Cstar$ category with simple basepoint $m = m\vartriangleleft 1_0$.
In this case, the Markov tower $M_\bullet$ constructed above has finite depth, and its principal graph $\Gamma$ is the \emph{fusion graph} of $(\cM,m)$ with respect to $X\in \cC$.
This means there is an $r>0$ such that the inclusion $M_{2r} \subset (M_{2r+1}, \tr_{2r+1})$ is \emph{strongly Markov}, meaning that there is a finite \emph{Pimsner-Popa basis} $\{b\}$ for $M_{2r+1}$ over $M_{2r}$ satisfying $\sum_b be_{2r} b^* = 1_{M_{2r+2}}$, and the \emph{Watatani index} $\sum_b bb^*$ \cite{MR996807} is a scalar (see \cite[1.1.4(c)]{MR1278111}).

By \cite[\S2.3]{MR2812459}, the inclusion $A_0:=M_{2r} \subset (M_{2r+1}, \tr_{2r+1}) =: (A_1, \tr_1)$ has a canonical associated planar $\dag$-algebra $\cP_\bullet$, which is built from the tower of higher relative commutants.
Moreover, by \cite[Thm.~3.8]{MR2812459}, the planar algebra $\cP_\bullet$ is \emph{non-canonically} isomorphic to the bipartite graph planar algebra $\cG_\bullet$ of the Bratteli diagram of the inclusion $A_0 \subset A_1$, which is also the fusion graph $\Gamma$. 
(This isomorphism depends on the loop algebra representation for $A_0 \subset A_1$ from \cite[\S3.1]{MR2812459}, which amounts to choosing compatible bases for the algebras.)

\begin{thmalpha}[Module embedding]
\label{thm:ModulePAEmbedding}
The unital $\dag$-algebra maps $\Phi_{n,\pm}:= \id_m \vartriangleleft \id_{(X\otimes \overline{X})^{\otimes r}}\vartriangleleft - : \cQ_{n,\pm} \to \cP_{n,\pm}$
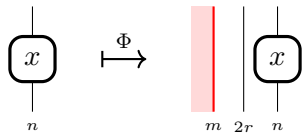
$$
\begin{tikzpicture}[baseline]
 \draw (0,-.7) -- (0,.7);
 \roundNbox{unshaded}{(0,0)}{.3}{0}{0}{$x$}
 \node at (0,-.9) {\tiny{$n$}};
\end{tikzpicture}
\quad
\begin{tikzpicture}[baseline]
 \clip (0.5,0.9) -- (-0.5,0.9) -- (-0.5,-0.9) -- (0.5,-0.9);
 \draw [|->,thick] (-0.3,0)--(0.3,0);
 \node at (0,0.25) {\scriptsize{$\Phi$}};
\end{tikzpicture}
\quad
\begin{tikzpicture}[baseline]
 \fill[ctwoshading] (-.4,-.7) -- (-0.4,.7) -- (-.7,0.7) -- (-.7,-0.7) -- (-.4,-0.7);
 \draw[thick,red] (-0.4,-.7) -- (-0.4,0.7);
 \draw (0.45,-.7) -- (0.45,.7);
 \draw (0,-.7) -- (0,.7);
 \roundNbox{unshaded}{(0.45,0)}{.3}{0}{0}{$x$}
 \node at (0.45,-.9) {\tiny{$n$}};
 \node at (0,-0.9){\tiny{$2r$}};
 \node at (-0.4,-.9) {\tiny{$m$}};
\end{tikzpicture} 
$$
give a planar $\dag$-algebra embedding $\cQ_\bullet \hookrightarrow \cP_\bullet$.
\end{thmalpha}

Choosing $\cM = \cC_{00}\oplus \cC_{10}$ and $m = 1_{0}$ corresponding to the unshaded empty diagram exactly recovers the embedding into the graph planar algebra of the principal graph of $\cQ_\bullet$ from \cite{MR2812459}.
Similarly, we get an embedding into the graph planar algebra of the dual principal graph by choosing $\cM= \cC_{10} \oplus \cC_{11}$ and an arbitrary simple object $m \in \cC_{10}$.

Notice we made three choices in our proof of the Module Embedding Theorem \ref{thm:ModulePAEmbedding}; we picked a simple object $m\in \cM$ with $m=m\vartriangleleft 1_0$, an $r\geq 0$ such that $M_{2r}\subset (M_{2r+1},\tr_{2r+1})$ is strongly Markov, and a planar $\dag$-algebra isomorphism $\cQ_\bullet \cong \cG\cP\cA(\Gamma)_\bullet$.
In our final Section \ref{sec:InvarianceOfEmbedding}, we explain that different choices still produce an equivalent planar $\dag$-algebra embedding $\cQ_\bullet \to \cG\cP\cA(\Gamma)_\bullet$.
Indeed, we show that the two corresponding strongly Markov inclusions are related by a shift and a compression by a projection with central support 1, and these processes yield planar $\dag$-isomorphisms on the associated canonical relative commutant planar algebras.

\paragraph{Acknowledgements.}

This article is the undergraduate research project of Desmond Coles and Srivatsa Srinivas, which was supervised by Peter Huston and David Penneys during Summer 2018, supported by Penneys' NSF DMS CAREER grant 1654159.
The authors would like to thank Corey Jones for many helpful conversations.
Additionally, David Penneys would like to thank Emily Peters and Noah Snyder for helpful conversations.

\section{Modules for subfactor planar algebras} 
\label{sec:Modules}

The standard invariant of a finite index ${\rm II}_1$ subfactor has many axiomatizations, including Popa's $\lambda$-lattices \cite{MR1334479} and Jones' subfactor planar algebras \cite{math.QA/9909027}.
Here, we use the language of subfactor planar algebras.
We discuss the well-known correspondence between subfactor planar algebras and their projection unitary multitensor categories.
We then introduce the notion of a planar module for a subfactor planar algebra, and we show it corresponds to a module category for the projection unitary multitensor category.

\subsection{Unitary multitensor categories and subfactor planar algebras}  
\label{sec:CategoriesPlanarAlgberasLattices}

In this section, we rapidly recall the definitions of a subfactor planar algebra \cite{math.QA/9909027} and its unitary $2\times 2$ multitensor category of projections \cite{MR2811311,1808.00323}.

\begin{defn}
The \emph{shaded planar operad} consists of shaded planar tangles with the operation of composition.
Shaded planar tangles have $r\geq 0$ input disks each with $2k_i$ boundary points, and an output disk with $2k_0$ boundary points.
Internal to the output disk are non-intersecting strings, which either attach 2 distinct boundary points, or are closed loops.
There is also a checkerboard shading, and a distinguished interval marked $\star$ for each input disk and the output disk.
If the $\star$ for the $i$-th disk is on an interval which meets an unshaded region, that disk has \emph{type} $(k_i, +)$, and if it meets a shaded region, the disk has type $(k_i,-)$.
A tangle with $r$ input disks has \emph{type}
 $((k_0,\pm_0);(k_1,\pm_1),\dots, (k_r, \pm_r))$ if the output disk has type $(k_0, \pm_0)$ and the $i$-th input disk has type $(k_i, \pm_i)$.
There is a natural definition of the composite tangle $T \circ_i S$ when the output disk of a tangle $T$ has the same type as the $i$-th input disk of a tangle $T$.
We give a representative example below, and we refer the reader to \cite{MR2679382,MR2972458} for a more precise definition.
$$
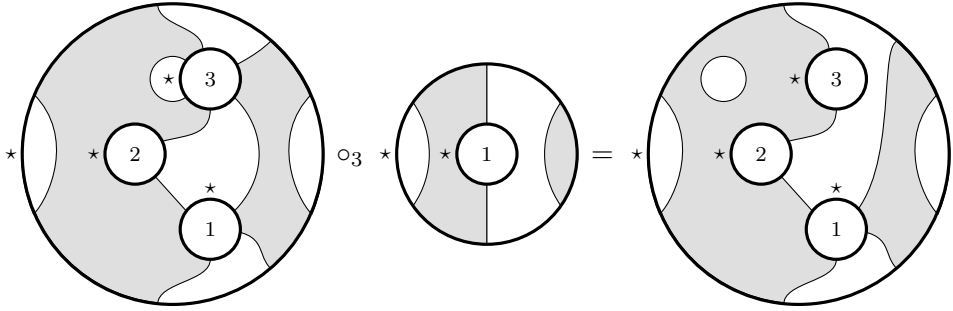
\begin{tikzpicture}[baseline = -.1cm]
 \filldraw[shaded]  
  (1.8,-2) .. controls ++(90:.3cm) and ++(-90:.3cm) .. (2.5,-1.4) -- (2.4,-1) -- 
  (1.5,0) .. controls ++(40:.5cm) and ++(270:.5cm) .. (2.5,.6) -- 
  (2.4,1.4) .. controls ++(90:.3cm) and ++(270:.3cm) .. (1.8,2) arc (96:264:2cm) ;
 \filldraw[unshaded] (2,1) circle (.3cm);
 \filldraw[shaded] (2.5,1) .. controls ++(30:.3cm) and ++(-135:.3cm) .. (3.3,1.5)
  arc (49:-49:2cm)  .. controls ++(135:.3cm) and ++(0:.8cm) ..  (2.5,-1)
  .. controls ++(30:1cm) and ++(-30:1cm) .. (2.5,1);
 \filldraw[unshaded] (3.85,.8) .. controls ++(-135:.8cm) and ++(135:.3cm) .. (3.85,-.8) arc (-23:23:2cm);
 \filldraw[unshaded] (.15,.8) .. controls ++(-45:.8cm) and ++(45:.3cm) .. (.15,-.8) arc (203:157:2cm);
 \ncircle{}{(2,0)}{2}{180}{}
 \ncircle{unshaded}{(2.5,1)}{.4}{180}{\scriptsize{$3$}}
 \ncircle{unshaded}{(1.5,0)}{.4}{180}{\scriptsize{$2$}}
 \ncircle{unshaded}{(2.5,-1)}{.4}{90}{\scriptsize{$1$}}
\end{tikzpicture}
\circ_{3}
\begin{tikzpicture}[baseline = -.1cm]
 \filldraw[shaded] (0,-1.2) -- (0,1.2) arc (90:270:1.2cm);
 \filldraw[unshaded] (145:1.2cm) arc (145:215:1.2cm) arc (-35:35:1.2cm);
 \filldraw[unshaded] (0,-1.2) -- (0,1.2) arc (90:-90:1.2cm);
 \filldraw[shaded] (35:1.2cm) arc (35:-35:1.2cm) arc (215:145:1.2cm);
 \ncircle{}{(0,0)}{1.2}{180}{}
 \ncircle{unshaded}{(0,0)}{.4}{180}{\scriptsize{$1$}}
\end{tikzpicture}
=
\begin{tikzpicture}[baseline = -.1cm]
 \filldraw[shaded]  
  (1.8,-2) .. controls ++(90:.3cm) and ++(-90:.3cm) .. (2.5,-1.4) -- (2.4,-1) -- 
  (1.5,0) .. controls ++(40:.5cm) and ++(270:.5cm) .. (2.5,.6) -- 
  (2.4,1.4) .. controls ++(90:.3cm) and ++(270:.3cm) .. (1.8,2) arc (96:264:2cm) ;
 \filldraw[unshaded] (1,1) circle (.3cm);
 \filldraw[shaded] (3.3,1.5) arc (49:-49:2cm)  .. controls ++(135:.3cm) and ++(0:.8cm) ..  (2.5,-1)
  .. controls ++(30:1cm) and ++(210:.3cm) .. (3.3,1.5);
 \filldraw[unshaded] (3.85,.8) .. controls ++(-135:.8cm) and ++(135:.3cm) .. (3.85,-.8) arc (-23:23:2cm);
 \filldraw[unshaded] (.15,.8) .. controls ++(-45:.8cm) and ++(45:.3cm) .. (.15,-.8) arc (203:157:2cm);
 \ncircle{}{(2,0)}{2}{180}{}
 \ncircle{unshaded}{(2.5,1)}{.4}{180}{\scriptsize{$3$}}
 \ncircle{unshaded}{(1.5,0)}{.4}{180}{\scriptsize{$2$}}
 \ncircle{unshaded}{(2.5,-1)}{.4}{90}{\scriptsize{$1$}}
\end{tikzpicture}
$$
The shaded planar operad also has a $\dag$-structure, with the tangle $T^\dag$ obtained by reflecting $T$ about a diameter. 

A \emph{shaded planar algebra} $\cP_\bullet$ consists of a family $\cP_{n,\pm}$ of $\bbC$-vector spaces together with an action of the shaded planar operad.
 That is, each shaded planar tangle $T$ with input disks of type $(k_i, \pm_i)$ for $1\leq i\leq r$ and output disk of type $(k_0,\pm_0)$ defines a multilinear map $Z(T) : \prod_{i=1}^r \cP_{k_i, \pm_i} \to \cP_{k_0, \pm_0}$, and tangle composition corresponds to composition of multilinear maps. Each $\cP_{n,\pm}$ should also have a dagger structure so that for every tangle $T$ and tuple $\eta_1\ldots r$ of inputs, $Z(T^\dag)(\eta_1^\dag,\eta_2^\dag\ldots\eta_r^\dag)=Z(T)(\eta_1,\eta_2,\ldots\eta_n)^\dag$. 

\end{defn}

\begin{nota}
We will try to shade our diagrams as much as possible for a shaded planar algebra.
However, sometimes shading our diagrams requires us to split into many cases.
In order to avoid this, we sometimes suppress the shading when it can be inferred from the indices.
We also tend to suppress the external boundary disk of a shaded planar tangle; when we do so, the $\star$ is always on the left.
For explicit examples, compare \ref{PA:Positive} and \ref{PA:Spherical} in the following definition.
\end{nota}

\begin{defn}
A shaded planar algebra is called a \emph{subfactor planar algebra} if moreover
\begin{enumerate}[label={\rm(PA\arabic*)}]
\item
(finite dimensional) $\dim(\cP_{n,\pm}) <\infty$ for all $n\geq 0$.
\item
(evaluable/connected) $\dim(\cP_{0,\pm}) = 1$.
\item
\label{PA:Positive}
(positive)
$\langle x, y\rangle_{n,\pm} := 
\begin{tikzpicture}[baseline=-.1cm]
 \draw (0,0) -- (.8,0);
 \roundNbox{unshaded}{(0,0)}{.3}{0}{0}{$x$}
 \roundNbox{unshaded}{(1,0)}{.3}{0}{0}{$y^\dag$}
 \node at (.5,.2){\tiny{$n$}};
 \node at (-.45,0){\scriptsize{$\star$}};
 \node at (1.45,0){\scriptsize{$\star$}};
\end{tikzpicture}
$
defines a positive-definite inner product on each $\cP_{n,\pm}$.
\item
\label{PA:Spherical}
(spherical)
for all $x\in \cP_{1,+}$, 
$
\begin{tikzpicture}[baseline=-.1cm]
 \filldraw[shaded] (0,.3) arc (180:0:.3cm) -- (.6,-.3) arc (0:-180:.3cm);
 \roundNbox{unshaded}{(0,0)}{.3}{0}{0}{$x$}
 \node at (-.45,0){\scriptsize{$\star$}};
\end{tikzpicture}
=
\begin{tikzpicture}[baseline=-.1cm]
 \fill[shaded, rounded corners = 5pt] (-.8,-.8) rectangle (.5,.8);
 \filldraw[unshaded] (0,.3) arc (0:180:.3cm) -- (-.6,-.3) arc (-180:0:.3cm);
 \roundNbox{unshaded}{(0,0)}{.3}{0}{0}{$x$}
 \node at (-.45,0){\scriptsize{$\star$}};
\end{tikzpicture}
$.
\end{enumerate}
In a subfactor planar algebra, closed contractible loops can be traded for a multiplicative scalar $d>0$; we call this the \textit{loop parameter}.
By Jones' index rigidity theorem \cite{MR0696688}, $d\in \set{2\cos(\pi/k)}{k\geq 3}\cup [2,\infty)$.

Given a subfactor planar algebra $\cP_\bullet$, we get two towers of finite dimensional von Neumann algebras $\cP_\pm= (\cP_{n,\pm})_{n\geq 0}$ 
with Jones projections for $k\geq 0$ given by
\begin{equation}
\label{eq:JonesProjectionsForPlanarAlgebra}
e_{2k+1,+}
:=
\begin{tikzpicture}[baseline]
 \draw (0,-.5) -- (0,.5);
 \filldraw[shaded] (.3,.5) arc (-180:0:.3cm);
 \filldraw[shaded] (.3,-.5) arc (180:0:.3cm);
 \node at (.2,0){\tiny{$2k$}};
\end{tikzpicture} 
\qquad
e_{2k+2,+}
:=
\begin{tikzpicture}[baseline]
 \fill[shaded] (0,-.5) rectangle (1.2,.5);
 \draw (0,-.5) -- (0,.5);
 \filldraw[unshaded] (.3,.5) arc (-180:0:.3cm);
 \filldraw[unshaded] (.3,-.5) arc (180:0:.3cm);
 \node at (.4,0){\tiny{$2k+1$}};
\end{tikzpicture} 
\qquad
e_{2k+1,-}:=
\begin{tikzpicture}[baseline]
 \fill[shaded] (-.2,-.5) rectangle (1.2,.5);
 \draw (0,-.5) -- (0,.5);
 \filldraw[unshaded] (.3,.5) arc (-180:0:.3cm);
 \filldraw[unshaded] (.3,-.5) arc (180:0:.3cm);
 \node at (.2,0){\tiny{$2k$}};
\end{tikzpicture} 
\qquad
e_{2k+2,-}:=
\begin{tikzpicture}[baseline]
 \fill[shaded] (-.2,-.5) rectangle (0,.5);
 \draw (0,-.5) -- (0,.5);
 \filldraw[shaded] (.3,.5) arc (-180:0:.3cm);
 \filldraw[shaded] (.3,-.5) arc (180:0:.3cm);
 \node at (.4,0){\tiny{$2k+1$}};
\end{tikzpicture}
\end{equation}
and traces for $n\geq 0$ given by
\begin{equation}
\label{eq:TracesForPlanarAlgebra}
\tr_{n,\pm}:= d^{-n}\,
\begin{tikzpicture}[baseline=-.1cm]
 \draw (0,.3) arc (180:0:.3cm) -- (.6,-.3) arc (0:-180:.3cm);
 \roundNbox{unshaded}{(0,0)}{.3}{0}{0}{}
 \node at (.8,0){\tiny{$n$}};
\end{tikzpicture} \,.
\end{equation}
We will see in \S\ref{sec:MarkovTowersAndElementaryProperties} below that $(\cP_{n,\pm}, \tr_{n,\pm}, e_{n+1,\pm})_{n\geq 0}$ has the structure of a \emph{Markov tower}, which comes with a \emph{principal graph}.
The principal graph of $\cP_+$ is finite if and only if the principal graph of $\cP_-$ is finite; in this case, $\cP_\bullet$ is said to have \emph{finite depth}.
\end{defn}

\begin{defn}\label{def:multitensor}
A \emph{unitary $2\times 2$ multitensor category} $\cC$ is an indecomposable rigid $\Cstar$ tensor category which is Karoubi complete such that $1_\cC$ has an orthogonal decomposition into simple objects as $1_\cC = 1_0 \oplus 1_1$.
We write $\cC_{ij} := 1_i \otimes \cC \otimes 1_j$ for $i,j \in \{0,1\}$.
By \cite{MR1444286}, such a $\cC$ is automatically semisimple.
When $\cC$ is finitely semisimple, it is called a unitary $2\times 2$ \emph{multifusion} category \cite{MR3242743}.

We say $X\in \cC_{01}$ \emph{generates} $\cC$ if every object of $\cC$ is isomorphic to a direct summand of an alternating tensor power of $X$ and $\overline{X}$
$$
X^{ \text{alt}\otimes n}
:=
\underbrace{X\otimes \overline{X} \otimes \cdots \otimes X^?}_{n\text{ tensorands}}
\qquad
\qquad
\overline{X}^{ \text{alt}\otimes n}
:=
\underbrace{\overline{X} \otimes X\otimes \cdots \otimes \overline{X}^?}_{n\text{ tensorands}}
$$
where $X^? = X$ if $n$ is odd and $\overline{X}$ when $n$ is even, and $\overline{X}^? = \overline{X}$ when $n$ is odd and $X$ when $n$ is even.
Here, $(\overline{X}, \ev_X, \coev_X)$ is the canonical \emph{balanced} dual of $X$ \cite{MR3342166,1805.09234,1808.00323} which satisfies the zig-zag axioms and the balancing equation
$$
\psi(\ev_X \circ (\id_{\overline{X}} \otimes f) \circ \ev_X^\dag)
=
\psi(\coev_X^\dag \circ (f\otimes \id_{\overline{X}}) \circ \coev_X)
\qquad
\forall f\in \cC(X\to X)
$$
where $\psi : \cC(1_\cC\to 1_\cC) \to \bbC$ is the linear functional such that $\psi(\id_{1_0}) = \psi(\id_{1_1}) = 1$.
\end{defn}

The following theorem is well-known to experts.

\begin{thm}
\label{thm:SubfactorTensorCategoryEquivalence}
There is an equivalence of categories\,\footnote{We suppress the subtlety about the right hand side of this equivalence being a contractible 2-category. 
We refer the reader to \cite{1607.06041,1808.00323} for more details.}
\[
\left\{\, 
\parbox{4.8cm}{\rm Subfactor planar algebras $\cP_\bullet$}\,\left\}
\,\,\,\,\cong\,\,
\left\{\,\parbox{7.5cm}{\rm Pairs $(\cC, X)$ with $\cC$ a unitary $2\times 2$ multitensor category together with a generator $X\in \cC_{01}$}\,\right\}.
\right.\right.
\]
\end{thm}

Starting with a subfactor planar algebra $\cP_\bullet$, one may form its unitary $2\times 2$ multitensor category of projections $\cC$ \cite{MR2559686,MR3405915,1808.00323}, which comes with a canonical generator corresponding to the unshaded-shaded strand in $\cP_{1,+}$, and the canonical spherical unitary dual functor \cite{MR3342166,1805.09234,1808.00323}.
This unitary $2\times 2$ multitensor category can also be thought of as a unitary 2-category called the \emph{paragroup}; we refer the reader to \cite{MR3157990} for more details.

Starting with a pair $(\cC, X)$, we get a subfactor planar algebra by defining
$$
\cP_{n,+} := \End_\cC(X^{ \text{alt} \otimes n})
\qquad\qquad
\cP_{n,-} := \End_\cC(\overline{X}^{ \text{alt} \otimes n}),
$$
and we define the action of the shaded planar operad via the diagrammatic calculus for pivotal tensor categories.
We refer the reader to \cite{MR2811311,1808.00323} for more details.

\subsection{Modules for unitary multitensor categories and subfactor planar algebras}
\label{ssec:moduledefs}
\label{sec:3TypesOfModules}

We now define the various notions of module for 
\begin{itemize}
\item
a unitary $2\times 2$ multitensor category $\cC$ with its canonical unitary spherical structure and a generator $X\in \cC_{01}$,
and
\item
a subfactor planar algebra $\cQ_\bullet$.
\end{itemize}

\begin{defn}
 \label{def:PivotalModule}
Let $\cC$ be a unitary $2\times 2$ multitensor category.
A \emph{pivotal right $\cC$-module $\Cstar$ category} is a pair $(\cM,\Tr^\cM)$ where
$\cM$ is a semisimple right $\cC$-module $\Cstar$ category, and $\Tr^\cM$ is a family of positive traces $\Tr^\cM_n:\cM(n \to n)\to \bbC$ on each endomorphism space for $n\in \cM$ satisfying the following axioms:
\begin{enumerate}[label={\rm(Tr\arabic*)}]
\item
\label{PivotalModule:tracial}
$\Tr^\cM_m(g \circ f) = \Tr^\cM_n(f\circ g)$ for all $f\in \cM(m\to n)$ and $g\in \cM(n \to m)$.
\item
\label{PivotalModule:positive}
$\Tr^\cM_m(f^\dag \circ f) \geq 0$ for all $f\in \cM(m \to n)$, and $\Tr^\cM_m(f^\dag \circ f) =0$ if and only if $f = 0$.
\item
\label{PivotalModule:compatible}
For all $m\in \cM$ and $c\in \cC$,
$
\Tr^\cM_{m\vartriangleleft c}(f)
=
\Tr^\cM_{m}(
(\id_m \vartriangleleft \coev_c^\dag) \circ (f\vartriangleleft \id_{\overline{c}}) \circ (\id_m \vartriangleleft \coev_c)
)
$
\end{enumerate}
Notice that $\cM = \cM_0 \oplus \cM_1$ where $\cM_0 = \cM \vartriangleleft 1_0$ and $\cM_1 = \cM\vartriangleleft 1_1$.

A pivotal right module category is called \emph{pointed} if it is indecomposable, we have a chosen simple object $m\in \cM$, and $\Tr^\cM$ is normalized so that $\Tr^\cM_m(\id_m) = 1_\bbC$.
Generally, we choose $m\in \cM_0$, but this choice is not essential.

When $\cC$ is generated by a single $X\in \cC_{01}$ and $(\cM,\Tr^\cM,m)$ is a pointed pivotal right module category with $m\in \cM_0$, we define the \emph{cyclic pivotal right module category} $\cM_{m,X}$ to be the (non Karoubi complete!) full subcategory of $\cM$ whose objects are of the form 
$
m \vartriangleleft X^{ \text{alt}\otimes k}
$ 
for $k\geq 0$, which is a pointed pivotal right module category over $\cC_X$, the (non Karoubi complete!) full subcategory of $\cC$ whose objects are of the form $X^{ \text{alt}\otimes k}$ and $\overline{X}^{ \text{alt}\otimes k}$ for $k\geq 0$.
\end{defn}

We next give an appropriate definition of planar modules over a planar algebra as algebras over another operad.\footnote{
Our definition of a planar module over a planar algebra differs significantly from the annular planar modules introduced in \cite{MR1929335}. 
Our planar modules will correspond to the above notion of module over a multitensor category, 
whereas annular planar modules 
are more closely related with representations of the tube algebra/affine annular category, which correspond to objects in the Drinfeld center of the multitensor category \cite{MR3447719}.
}

\begin{defn}
\label{def:PlanarModule}
 The \textit{shaded planar module operad} is a variant of the shaded planar operad, akin to a shaded, stranded version of the Swiss-cheese operad introduced in \cite{MR1718089}. 
 In this operad, the starred region of the boundary of the output disk of a tangle is replaced by a vertical line on the left side of a tangle, and the adjacent region inside the tangle must be unshaded. 
 In addition to the usual input disks, tangles may also have input semidisks, whose boundaries intersect the left wall. 
Similar to the definition of type for an input disk, a semidisk (input or output) has \emph{type} $k_i$ if it has $2k_i$ boundary points which meet $2k_i$ strings.
A tangle with $r$ input semidisks and $s$ input disks has \emph{type}
 $(k_0;k_1,\dots, k_r; (\ell_1,\pm_1),\dots, (\ell_s, \pm_s))$ if the output disk has type $k_0$, the $i$-th input semidisk has type $k_i$, and the $j$-th input disk has type $(\ell_j,\pm_j)$. 
 The operadic composition comes from plugging tangles into semidisks when the types are compatible.
 A representative tangle appears below. 
\end{defn}
$$
\begin{tikzpicture}[baseline = -.1cm]
 \filldraw[shaded] (.3,1.3) .. controls ++(90:.3cm) and ++(270:.3cm) .. (.8,2) -- 
  (1.8,2) .. controls ++(270:.3cm) and ++(90:.3cm) .. (2.5,1.4) --
  (2.5,1) .. controls ++(270:1cm) and ++(-30:1cm) .. (.85,-1.25) --
  (.85,-.75) .. controls ++(30:.3cm) and ++(270:.3cm) .. (1.5,-.4) -- 
  (1.5,.4) .. controls ++(90:.8cm) and ++(270:.8cm) .. (.3,.7);
 \filldraw[unshaded] (2,1) circle (.3cm);
 \filldraw[shaded] (.2,-.7) -- (.2,-.4) arc (180:0:.15cm) -- (.5,-.7);
 \filldraw[shaded]  (.2,-1.3) .. controls ++(270:.3cm) and ++(90:.3cm) .. (.8,-2) -- (1.8,-2)
  (1.8,-2) .. controls ++(90:.3cm) and ++(-30:1.2cm) .. (2.5,-1)
  .. controls ++(-135:1cm) and ++(-45:1cm) .. (.5,-1.3) -- (.2,-1.3);
 \filldraw[shaded] (2.5,1) .. controls ++(30:.3cm) and ++(-135:.3cm) .. (3.3,1.5)
  arc (49:-49:2cm)  .. controls ++(135:.3cm) and ++(0:.8cm) ..  (2.5,-1)
  .. controls ++(30:1cm) and ++(-30:1cm) .. (2.5,1);
 \filldraw[unshaded] (3.85,.8) .. controls ++(-135:.8cm) and ++(135:.3cm) .. (3.85,-.8) arc (-23:23:2cm);
 \halfcircle{}{(0,0)}{2}{2}{}
 \openhalfcircle{}{(0,1)}{.3}{.4}{\scriptsize{$2$}}
 \openhalfcircle{}{(0,-1)}{.3}{.4}{\scriptsize{$1$}}
 \ncircle{unshaded}{(2.5,1)}{.4}{180}{\scriptsize{$5$}}
 \ncircle{unshaded}{(1.5,0)}{.4}{0}{\scriptsize{$4$}}
 \ncircle{unshaded}{(2.5,-1)}{.4}{90}{\scriptsize{$3$}}
\end{tikzpicture}
\,:
(\cM_3 \otimes \cM_1) \otimes (\cP_{2,+}\otimes \cP_{1,-} \otimes \cP_{3,+}) \to \cM_{4}
$$
Tangles of the shaded planar module operad can also be composed with shaded planar tangles, by plugging a shaded planar tangle into an input disk. 
One should think of the box spaces for semidisks as being endomorphisms of objects in a module category, while the involvement of full disks allows a planar algebra to act on the module. 
The input disks and semidisks are also numbered, with the numbering determining the order of the tensor factors in the domain of the action map as depicted above. Like the shaded planar operad, the shaded planar module operad is a symmetric operad, and vector spaces form a symmetric monoidal category, so we often suppress the numbering. 

\begin{defn}
\label{def:RightPlanarModule}
A \emph{right planar module} $\cM_\bullet$ for the subfactor planar algebra $\cP_\bullet$ consists of a sequence of finite dimensional $\bbC$-vector spaces $(\cM_k)_{k\geq 0}$
and a conjugate-linear map $\dag: \cM_k \to \cM_k$ for all $k\geq 0$,
together with an action of the shaded planar module operad on the box spaces $\cM_\bullet$ and $\cP_\bullet$ compatible with the composition of tangles and the shaded planar algebra structure on $\cP_\bullet$, and the $\dag$ operation. 
In other words, the box spaces $\cM_\bullet$ and $\cP_\bullet$ together must have the structure of an algebra over the shaded planar module operad, which must extend the original shaded planar operad algebra structure on $\cP_\bullet$. 

Notice that each of the $\cM_k$ has a $\dag$-algebra structure with multiplication given by the tangle
$$
\begin{tikzpicture}[baseline = -.1cm]
 \draw (.25,-1) -- (.25,1);
 \halfcircle{}{(0,0)}{1}{.3}{}
 \openhalfcircle{}{(0,.4)}{.25}{.25}{\scriptsize{$2$}}
 \openhalfcircle{}{(0,-.4)}{.25}{.25}{\scriptsize{$1$}}
 \node at (.4,0) {\scriptsize{$k$}};
 \node at (.4,.8) {\scriptsize{$k$}};
 \node at (.4,-.8) {\scriptsize{$k$}};
\end{tikzpicture}
:
\cM_k \otimes \cM_k \to \cM_k.
$$
We require that each $\dag$-algebra $\cM_k$ is a finite dimensional $\Cstar/\Wstar$ algebra.
Moreover, we require that for each $k$, the following map $\cM_k \to \cM_0$ is positive, faithful, and tracial:
\begin{equation}
\label{eq:NonnormalizedPlanarModuleTrace}
\Tr_k:=
\begin{tikzpicture}[baseline = -.1cm]
 \draw (.25,.25) arc (180:0:.2cm) -- (.65,-.25) arc (0:-180:.2cm);
 \halfcircle{}{(0,0)}{.8}{.3}{}
 \openhalfcircle{}{(0,0)}{.25}{0}{}
 \node at (.8,0) {\scriptsize{$k$}};
\end{tikzpicture}
:
\cM_k \to \cM_0
\end{equation}

We call $\cM_\bullet$ \emph{connected} if $\dim(\cM_0) = 1$.
In this case, we can canonically identify $\cM_0 = \bbC$ as a $\Cstar$-algebra, and each $\Tr_k$ is a scalar-valued.
We define $\tr_k := d^{-k} \Tr_k$, where $d$ is the loop parameter of $\cP_\bullet$.
Notice that the $\tr_k$ are faithful tracial states.
We will see that, under the correspondence of Theorem \ref{thm:ModuleEquivalence}, connected right planar modules correspond to cyclic pivotal right module $\Cstar$-categories.
\end{defn}

\begin{ex}
 Given a subfactor planar algebra $\cP_\bullet$, $\cP_+ :=( \cP_{k,+})_{k\geq 0}$ is a cyclic right planar module for $\cP_\bullet$, while $\cP_- :=( \cP_{k,-})_{k\geq 0}$ is a right planar module for the dual planar algebra of $\cP_{\bullet}$ obtained by reversing the shading. 
\end{ex}

\begin{ex}
Suppose $\cG_\bullet =\cG\cP\cA(\Gamma)_\bullet$ is the graph planar algebra of the bipartite graph $\Gamma$.
Then for any $+$/even vertex $v$ of $\Gamma$, we get a cyclic right planar module $\cM_\bullet = \cM(\Gamma,v)_\bullet$ by defining $\cM(\Gamma, v)_{k}:= p_v \cG_{k,\pm}$ and action of the planar module operad by
$$
Z_{\cM_\bullet}(\cT)
\left(
\underbrace{m_1\otimes \cdots \otimes m_r}_{\in \cM_\bullet} 
\otimes 
\underbrace{x_1\otimes \cdots\otimes x_s}_{\in \cP_\bullet}
\right)
:=
Z_{\cG_\bullet}(\widetilde{\cT})(p_v\otimes m_1\otimes \cdots \otimes m_r \otimes \Phi(x_1)\otimes \cdots\otimes \Phi(x_s))
$$
where $\widetilde{\cT}$ is obtained from $\cT$ by 
first turning each half-open input semidisk into a closed input disk in the interior of the output disk with its $\star$ on the left, 
rounding out the $90^\circ$ angles on the left boundary into a smooth curve,
putting the external $\star$ on the left hand side,
and inserting one $(0,+)$-type input disk in the left-most region of the new tangle which is numbered first.
We illustrate this procedure on the tangle above:
$$
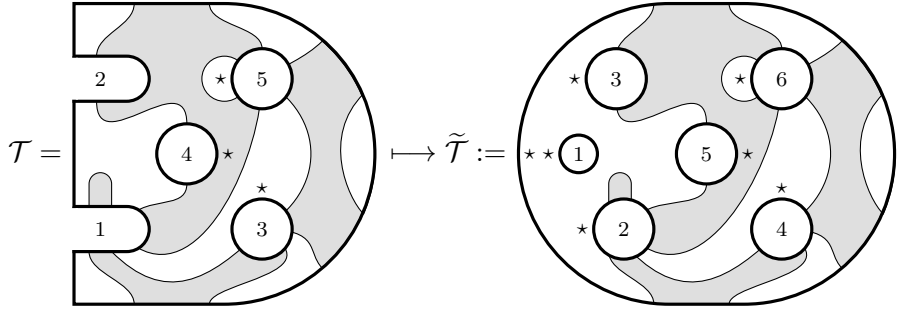
\cT=
\begin{tikzpicture}[baseline = -.1cm]
 \filldraw[shaded] (.3,1.3) .. controls ++(90:.3cm) and ++(270:.3cm) .. (.8,2) -- 
  (1.8,2) .. controls ++(270:.3cm) and ++(90:.3cm) .. (2.5,1.4) --
  (2.5,1) .. controls ++(270:1cm) and ++(-30:1cm) .. (.85,-1.25) --
  (.85,-.75) .. controls ++(30:.3cm) and ++(270:.3cm) .. (1.5,-.4) -- 
  (1.5,.4) .. controls ++(90:.8cm) and ++(270:.8cm) .. (.3,.7);
 \filldraw[unshaded] (2,1) circle (.3cm);
 \filldraw[shaded] (.2,-.7) -- (.2,-.4) arc (180:0:.15cm) -- (.5,-.7);
 \filldraw[shaded]  (.2,-1.3) .. controls ++(270:.3cm) and ++(90:.3cm) .. (.8,-2) -- (1.8,-2)
  (1.8,-2) .. controls ++(90:.3cm) and ++(-30:1.2cm) .. (2.5,-1)
  .. controls ++(-135:1cm) and ++(-45:1cm) .. (.5,-1.3) -- (.2,-1.3);
 \filldraw[shaded] (2.5,1) .. controls ++(30:.3cm) and ++(-135:.3cm) .. (3.3,1.5)
  arc (49:-49:2cm)  .. controls ++(135:.3cm) and ++(0:.8cm) ..  (2.5,-1)
  .. controls ++(30:1cm) and ++(-30:1cm) .. (2.5,1);
 \filldraw[unshaded] (3.85,.8) .. controls ++(-135:.8cm) and ++(135:.3cm) .. (3.85,-.8) arc (-23:23:2cm);
 \halfcircle{}{(0,0)}{2}{2}{}
 \openhalfcircle{}{(0,1)}{.3}{.4}{\scriptsize{$2$}}
 \openhalfcircle{}{(0,-1)}{.3}{.4}{\scriptsize{$1$}}
 \ncircle{unshaded}{(2.5,1)}{.4}{180}{\scriptsize{$5$}}
 \ncircle{unshaded}{(1.5,0)}{.4}{0}{\scriptsize{$4$}}
 \ncircle{unshaded}{(2.5,-1)}{.4}{90}{\scriptsize{$3$}}
\end{tikzpicture}
\longmapsto
\widetilde{\cT}:=
\begin{tikzpicture}[baseline = -.1cm]
 \filldraw[shaded] (.3,1.3) .. controls ++(90:.3cm) and ++(270:.3cm) .. (.8,2) -- 
  (1.8,2) .. controls ++(270:.3cm) and ++(90:.3cm) .. (2.5,1.4) --
  (2.5,1) .. controls ++(270:1cm) and ++(-30:1cm) .. (.7,-1.25) --
  (.7,-.75) .. controls ++(30:.3cm) and ++(270:.3cm) .. (1.5,-.4) -- 
  (1.5,.4) .. controls ++(90:.8cm) and ++(270:.8cm) .. (.3,.7);
 \filldraw[unshaded] (2,1) circle (.3cm);
 \filldraw[shaded] (.2,-.7) -- (.2,-.4) arc (180:0:.15cm) -- (.5,-.7);
 \filldraw[shaded]  (.2,-1.3) .. controls ++(270:.3cm) and ++(90:.3cm) .. (.8,-2) -- (1.8,-2)
  (1.8,-2) .. controls ++(90:.3cm) and ++(-30:1.2cm) .. (2.5,-1)
  .. controls ++(-135:1cm) and ++(-45:1cm) .. (.5,-1.3) -- (.2,-1.3);
 \filldraw[shaded] (2.5,1) .. controls ++(30:.3cm) and ++(-135:.3cm) .. (3.3,1.5)
  arc (49:-49:2cm)  .. controls ++(135:.3cm) and ++(0:.8cm) ..  (2.5,-1)
  .. controls ++(30:1cm) and ++(-30:1cm) .. (2.5,1);
 \filldraw[unshaded] (3.85,.8) .. controls ++(-135:.8cm) and ++(135:.3cm) .. (3.85,-.8) arc (-23:23:2cm);
 \draw[very thick] (-1,0) arc (180:90:2cm) -- (2,2) arc (90:-90:2cm) -- (1,-2) arc (-90:-180:2cm);
 \ncircle{unshaded}{(-.2,0)}{.25}{180}{\scriptsize{$1$}}
 \ncircle{unshaded}{(.3,1)}{.4}{180}{\scriptsize{$3$}}
 \ncircle{unshaded}{(.4,-1)}{.4}{180}{\scriptsize{$2$}}
 \ncircle{unshaded}{(2.5,1)}{.4}{180}{\scriptsize{$6$}}
 \ncircle{unshaded}{(1.5,0)}{.4}{0}{\scriptsize{$5$}}
 \ncircle{unshaded}{(2.5,-1)}{.4}{90}{\scriptsize{$4$}}
 \node at (-.85,0) {\scriptsize{$\star$}};
\end{tikzpicture}
$$
\end{ex}

\subsection{Equivalence of modules}

In this section, $\cP_\bullet$ will denote a subfactor planar algebra, and $(\cC, X)$ will denote its unitary $2\times 2$ multitensor category of projections, where $X = 1_0 \otimes X \otimes 1_1$ is the generating unshaded-shaded strand.
We now sketch the proof of the following theorem.

\begin{thm*}[Theorem \ref{thm:ModuleEquivalence}]
There is a canonical bijection between equivalence classes of
\begin{enumerate}[label={\rm(\arabic*)}]
\item
indecomposable pivotal right $\cC$-module $\Cstar$ categories $(\cM,\Tr^\cM)$ with simple basepoint $m = m\vartriangleleft 1_0$, and
\item
connected right planar modules $\cM_\bullet$ for $\cP_\bullet$.
\end{enumerate}
\end{thm*}

As an application, we get a classification of pivotal module $\Cstar$ categories for the 2-shaded Temperley-Lieb-Jones category with parameter $d$ in Corollary \ref{cor:TLJPivotalModuleClassification}, whose proof appears in \S\ref{sec:TLJmodules}.

\begin{defn}
Suppose $(\cM, \Tr^\cM, m)$ is an indecomposable pivotal right $\cC$-module $\Cstar$ category, and $m\in \cM$ is a distinguished simple object with $m= m \vartriangleleft 1_0$.
 We build a connected right planar $\cP_\bullet$-module $\cM_\bullet$ by defining $\cM_k := \End_\cC( m \vartriangleleft X^{\text{alt}\otimes k})$ for $k\geq 0$, and we define the action of the shaded planar module operad via the diagrammatic calculus for $\cM$. The process is similar to that in \cite[Def.~3.2]{MR2812459}.
 
We first define a standard form for tangles of the shaded planar module operad, such that every tangle is isotopic to one in standard form. We say a tangle is in \textit{standard form} if 
\begin{enumerate}[label={\rm(SF\arabic*)}]
 \item Each disk and semidisk, including the output semidisk, is rectangular in shape, with an equal number of strings emerging from the top and bottom,
 \item in the case of a disk, the starred boundary interval includes the left side, and 
 \item a horizontal line through the tangle passes through a disk, semidisk, or extremum of a strand at most once. 
\end{enumerate}

Given a shaded planar module tangle $T$ with type 
$(k_0;k_1,\dots, k_r; (\ell_1,\pm_1),\dots, (\ell_s, \pm_s))$ 
together with appropriate inputs $(f_1,\dots, f_r, x_1,\dots, x_s)$ with $m_i \in \cM_{k_i}$ and $x_j \in \cP_{\ell_j, \pm_j}$,
we begin with the identity morphism of $m \vartriangleleft X^{\text{alt}k_0}$ in $\cM_{k_0}$,
and we move an imaginary horizontal line upwards along the tangle.
Each time the horizontal line passes a disk, semidisk, or local extrema (which can happen only one at a time!), we compose with a morphism from $\cM$.
In more detail, when the horizontal line passes:
\begin{itemize}
\item
the $i$-th input semidisk with vertical stands to the right, we compose with $f_j \vartriangleleft \id$, where $\id$ is the appropriate identity morphism corresponding to the strands to the right of the input semidisk
\item
the $j$-th input disk with vertical strands to the left and right, we compose with $\id_{m}\vartriangleleft \id_l \otimes x \otimes \id_r$ where $\id_l, \id_r$ correspond to the  appropriate identity morphisms corresponding to the strands to the left and right of the input disk
\item
a local extrema with vertical strands to the left and right, we compose with $\id_{m}\vartriangleleft \id_l \otimes v \otimes \id_r$ where $v$ stands for the following (co)evaluation or its dagger depending on the shading:
$$
\begin{tikzpicture}[baseline=-.1cm]
 \fill[shaded] (-.2,-.2) rectangle (.8,.3);
 \filldraw[unshaded] (0,-.2) arc (180:0:.3cm);
\end{tikzpicture}
\leadsto
\ev_X
\qquad\qquad
\begin{tikzpicture}[baseline=-.1cm, yscale=-1]
 \filldraw[shaded] (0,-.2) arc (180:0:.3cm);
\end{tikzpicture}
\leadsto
\coev_X
\qquad\qquad
\begin{tikzpicture}[baseline=-.1cm, yscale=-1]
 \fill[shaded] (-.2,-.2) rectangle (.8,.3);
 \filldraw[unshaded] (0,-.2) arc (180:0:.3cm);
\end{tikzpicture}
\leadsto
\ev_X^\dag
\qquad\qquad
\begin{tikzpicture}[baseline=-.1cm]
 \filldraw[shaded] (0,-.2) arc (180:0:.3cm);
\end{tikzpicture}
\leadsto
\coev_X^\dag,
$$
and $\id_l, \id_r$ are the appropriate identity morphisms as above.
\end{itemize}
The output is the composite morphism in $\cM_{k_0} = \End_\cM(m \vartriangleleft X^{\text{alt}k_0})$.
One then checks that the resulting composite morphism is independent of the choice of standard form for the shaded planar module tangle $T$.

\begin{ex}
Here is an explicit example of a tangle in standard form, together with the corresponding multi-linear map obtained by composing the associated morphisms in $\cM$ from bottom to top:
\begin{align*}
\begin{tikzpicture}[baseline = -.1cm]
 \filldraw[shaded] (.3,-1.5) arc (-180:0:.25cm) -- (.8,0) arc (180:0:.25cm) -- (1.5,0) .. controls ++(90:.4cm) and ++(270:.4cm) .. (.7,.9) -- (.3,.9) -- (.3,-.9);
 \filldraw[shaded] (.3,1.5) -- (.3,2.2) -- (1.5,2.2) arc (90:-90:2.2cm) -- (.3,-2.2) .. controls ++(90:.3cm) and ++(270:1.3cm) .. (1.3,-.6) -- (1.5,-.6) arc (-180:0:.25cm) -- (2,0) .. controls ++(90:1.5cm) and ++(90:1cm) .. (.7,1.5);
 \filldraw[unshaded] (1.5,-2.2) .. controls ++(90:.5cm) and ++(270:1cm) .. (3,0) .. controls ++(90:1cm) and ++(270:.5cm) .. (1.5,2.2) arc (90:-90:2.2cm);
 \halfcircle{}{(0,0)}{2.2}{1.5}{}
 \openhalfcircle{}{(0,1.2)}{.3}{.4}{\scriptsize{$2$}}
 \openhalfcircle{}{(0,-1.2)}{.3}{0}{\scriptsize{$1$}}
 \ncircle{unshaded}{(1.4,-.3)}{.3}{180}{\scriptsize{$3$}}
 \ncircle{unshaded}{(2.8,.575)}{.3}{180}{\scriptsize{$4$}}
  \node at (-2.5,-2.1) {\scriptsize{$\id_m\vartriangleleft \coev_X \otimes \id_{X\otimes \overline{X}}$}};
  \node at (-2.5,-1.5) {\scriptsize{$- \vartriangleleft \id_{\overline{X}\otimes X\otimes \overline{X}}$}};
  \node at (-2.5,-.9) {\scriptsize{$\id_m\vartriangleleft \id_{X\otimes \overline{X}\otimes X} \otimes \ev_X^\dag \otimes \id_{\overline{X}}$}};
  \node at (-2.5,-.3) {\scriptsize{$\id_m\vartriangleleft \id_{X\otimes \overline{X}}\otimes - \otimes \id_{X\otimes \overline{X}}$}};
  \node at (-2.5,.3) {\scriptsize{$\id_m\vartriangleleft \id_{X}\otimes \coev_X^\dag \otimes \id_{\overline{X}\otimes X\otimes \overline{X}}$}};
  \node at (-2.5,.9) {\scriptsize{$\id_m\vartriangleleft \id_{X\otimes \overline{X}\otimes X}\otimes - $}};
  \node at (-2.5,1.5) {\scriptsize{$-\vartriangleleft \id_{X\otimes \overline{X}}$}};
  \node at (-2.5,2.1) {\scriptsize{$\id_m\vartriangleleft \id_{X}\otimes \coev_X^\dag \otimes \id_{\overline{X}}$}};
\end{tikzpicture}
:
\cM_1\otimes \cM_2 \otimes \cP_{2,+} \otimes \cP_{1,-}
&\longrightarrow 
\cM_{2}
\end{align*}
\end{ex}
\end{defn}

\begin{defn}
Given a connected right planar $\cP_\bullet$-module $\cM_\bullet$, we let $\cM$ be its \emph{category of projections}.
The objects of $\cM$ are the orthogonal projections in $\cM_k$ for $k\geq 0$.
The $\Hom$-space $\cM(p \to q)$ for $p\in \cM_j$ and $q\in \cM_k$ is only nonzero if $j \equiv k \mod 2$; in this case, we define
$$
\cM(p \to q)
:=
\set{x\in \cM_{(j+k)/2}}{
\,\,
\begin{tikzpicture}[baseline = -.1cm]
 \draw (.25,-.6) -- (.25,.6);
 \halfcircle{}{(0,0)}{.6}{.4}{}
 \halfcircle{unshaded}{(0,0)}{.25}{.25}{$x$}
 \node at (.4,.4) {\scriptsize{$k$}};
 \node at (.4,-.4) {\scriptsize{$j$}};
\end{tikzpicture}
=
\begin{tikzpicture}[baseline = -.1cm]
 \draw (.4,-1.5) -- (.4,1.5);
 \halfcircle{}{(0,0)}{1.5}{.4}{}
 \halfcircle{unshaded}{(0,.85)}{.25}{.5}{$q$}
 \halfcircle{unshaded}{(0,0)}{.25}{.5}{$x$}
 \halfcircle{unshaded}{(0,-.85)}{.25}{.5}{$p$}
 \node at (.6,1.3) {\scriptsize{$k$}};
 \node at (.6,.4) {\scriptsize{$k$}};
 \node at (.6,-.4) {\scriptsize{$j$}};
 \node at (.6,-1.3) {\scriptsize{$j$}};
\end{tikzpicture}
}.
$$
Composition is given by a suitable version of the usual multiplication tangle, and the $\dag$-structure is given by $\dag$ in $\cM_\bullet$.
Given a projection $p \in \cM_{k}$ and projections $q\in \cP_{n,+}$ and $r\in \cP_{n,-}$, we define depending on parity
$$
p \vartriangleleft q
\underset{\,\,k=2j}{:=}\,\,
\begin{tikzpicture}[baseline = -.1cm]
 \draw (.25,-.6) -- (.25,.6);
 \draw (1.2,-.6) -- (1.2,.6);
 \halfcircle{}{(0,0)}{.6}{1.3}{}
 \halfcircle{unshaded}{(0,0)}{.25}{.25}{$p$}
 \ncircle{unshaded}{(1.2,0)}{.25}{180}{$q$}
 \node at (.5,.4) {\scriptsize{$2j$}};
 \node at (.5,-.4) {\scriptsize{$2j$}};
 \node at (1.4,.4) {\scriptsize{$n$}};
 \node at (1.4,-.4) {\scriptsize{$n$}};
\end{tikzpicture}
\qquad
\qquad
p \vartriangleleft r
\underset{\,\,k=2j+1}{:=}\,\,
\begin{tikzpicture}[baseline = -.1cm]
 \fill[shaded] (.25,-.6) rectangle (1.2,.6);
 \draw (.25,-.6) -- (.25,.6);
 \draw (1.2,-.6) -- (1.2,.6);
 \halfcircle{}{(0,0)}{.6}{1.3}{}
 \halfcircle{unshaded}{(0,0)}{.25}{.25}{$p$}
 \ncircle{unshaded}{(1.2,0)}{.25}{180}{$q$}
 \node at (.7,.4) {\scriptsize{$2j+1$}};
 \node at (.7,-.4) {\scriptsize{$2j+1$}};
 \node at (1.4,.4) {\scriptsize{$n$}};
 \node at (1.4,-.4) {\scriptsize{$n$}};
\end{tikzpicture}\,.
$$
For morphisms $f\in \cM(p \to q)$ 
with $p\in \cM_j$ and $q\in \cM_k$
and $g\in \cC(r \to s)$ 
with $r \in \cP_{m,\pm}$ and $s\in \cP_{n,\pm}$
such that $p\vartriangleleft r$ and $q\vartriangleleft s$ are well-defined,
we define
$$
f\vartriangleleft g
:=
\begin{tikzpicture}[baseline = -.1cm]
 \draw (.25,-.6) -- (.25,.6);
 \draw (1.2,-.6) -- (1.2,.6);
 \halfcircle{}{(0,0)}{.6}{1.3}{}
 \halfcircle{unshaded}{(0,0)}{.25}{.25}{$f$}
 \ncircle{unshaded}{(1.2,0)}{.25}{180}{$g$}
 \node at (.4,.4) {\scriptsize{$k$}};
 \node at (.4,-.4) {\scriptsize{$j$}};
 \node at (1.4,.4) {\scriptsize{$n$}};
 \node at (1.4,-.4) {\scriptsize{$m$}};
\end{tikzpicture}\,.
$$
where shading depends on the parity of $j$ and $k$.
We leave the straightforward verification that $\cM$ is a right $\cC$-module $\Cstar$ category to the reader.
Finally, we replace $\cM$ with its unitary Karoubi completion, which formally adds orthogonal direct sums and then takes the orthogonal projection completion.

The distinguished simple basepoint of $\cM$ is given by the identity projection $1_{\cM_0} \in \cM_0$.
The trace $\Tr^\cM_p : \cM(p \to p) \to \bbC$ for $p\in \cM_k$ is given by the \emph{non-normalized} trace $\Tr_k$ from \eqref{eq:NonnormalizedPlanarModuleTrace} restricted to $p\cM_k p = \cM(p \to p)$. 
By definition, we have assumed $\Tr_k$ to be tracial and positive on endomorphisms. That the trace is also compatible with the action of $\cC$ can be seen by composing the trace tangle with the action tangles. 
\end{defn}

\section{Markov towers and their projection categories} 
\label{sec:MarkovTowers}

So far, we have presented two versions of the concept of a module over a subfactor planar algebra. The algebraic data of each shares a common structure: that of a \emph{Markov tower} of finite dimensional tracial von Neumann algebras. Studying elementary properties of Markov towers will therefore allow us to state many important results about planar modules in single, common language.  
The definition of a Markov tower can obtained from the definition of Popa's $\lambda$-\emph{sequences of commuting squares} from \cite{MR1334479} by forgetting one of the towers, 
analogous to the way one defines a module for an algebraic object by replacing one argument of the algebraic operation with an element from the module. In short, Markov towers are the towers-of-algebras analog of a module category. 
In \S\ref{sec:TLJmodules} below, we will see that Markov towers are exactly a $\lambda$-lattice approach to pivotal Temperley-Lieb-Jones module categories; this motivates the view of subfactor planar modules as simply Markov towers with an additional structure. 

\subsection{Markov towers and their elementary properties}
\label{sec:MarkovTowersAndElementaryProperties}
\begin{defn}
A \emph{Markov tower} $M_\bullet = (M_n, \tr_n, e_{n+1})_{n\geq 0}$ consists of a sequence $(M_n, \tr_n)_{n\geq 0}$ of finite dimensional von Neumann algebras, such that $M_n$ is unitally included in $M_{n+1}$, each $M_n$ has a faithful normal tracial states such that $\tr_{n+1}|_{M_n} = \tr_n$ for all $n\geq 0$, and there is a sequence of \emph{Jones projections} $e_n \in M_{n+1}$ for all $n\geq 1$, such that:
\begin{enumerate}[label={\rm(M\arabic*)}]
\item
\label{eq:MarkovJonesProjections}
The projections $(e_n)$ satisfy the Temperley-Lieb-Jones relations:
\begin{enumerate}[label={\rm(TLJ\arabic*)}]
\item
$e_i^2 = e_i = e_i^*$ for all $i$,
\item
$e_i e_j = e_j e_i$ for $|i-j|>1$, and
\item
there is a fixed constant $d>0$ called the \emph{modulus} such that $e_{i} e_{i\pm 1} e_i = d^{-2} e_i$ for all $i$.
\end{enumerate}
\item
\label{eq:MarkovImplement}
For all $x\in M_n$, $e_n x e_n = E_n(x)e_n$, where $E_n: M_n \to M_{n-1}$ is the canonical faithful trace-preserving conditional expectation.
\item
\label{eq:MarkovIndex}
For all $n\geq 1$, $E_{n+1}(e_n) = d^{-2}$.
\item
\label{eq:MarkovPullDown}
(pull down)
For all $n\geq 1$, $M_{n+1}e_n = M_n e_n$.
\end{enumerate}
We call a Markov tower \emph{connected} if $\dim(M_0) = 1$.
\end{defn}

\begin{rem}
One should think of the preceding definition as obtained from Popa's definition of \emph{$\lambda$-sequence} \cite{MR1334479} and removing one of the two sequences of algebras, together with the commuting square condition.
Compare the existence of Jones projections, \ref{eq:MarkovJonesProjections}, and \ref{eq:MarkovImplement} with (1.3.2), and \ref{eq:MarkovIndex} and \ref{eq:MarkovPullDown} with (1.3.3') from \cite{MR1334479} respectively.
\end{rem}

\begin{rem}\label{pulldowniff}
Observe that $M_n e_n M_n$ is a 2-sided ideal in $M_{n+1}$ for all $n\geq 1$ if and only if the pull down condition holds. Indeed, if the pull down condition holds, then $M_{n+1} M_n e_n M_n \subseteq M_{n+1} e_n M_n = M_n e_n M_n$; the same argument holds on the right by first taking adjoints. Conversely, if $M_n e_n M_n$ is a 2-sided ideal, then $M_{n+1} e_n = (M_{n+1} e_n)e_n \subseteq (M_n e_n M_n) e_n = M_n e_n$.
\end{rem}

\begin{prop}\label{prop:ElementaryMarkov} A Markov tower satisfies the following elementary properties for $n\geq 1$.
\begin{enumerate}[label={\rm(EP\arabic*)}]
\item
\label{EP:Injective}
The map $M_{n}\ni y\mapsto ye_n \in M_{n+1}$ is injective.

\item
\label{EP:UniquePullDown}
For all $x\in M_{n+1}$, $d^{2}E_{n+1}(x e_n)$ is the unique element $y\in M_n$ such that $x e_n = ye_n$ \cite[Lem.~1.2]{MR860811}.

\item
\label{EP:MarkovTraces}
The traces $\tr_{n+1}$ satisfy the following \emph{Markov property} with respect to $M_n$ and $e_n$: for all $x\in M_n$, $\tr_{n+1}(xe_n) = d^{-2} \tr_n(x)$.

\item
\label{EP:CompressM_{n+1}}
$e_n M_{n+1}e_n = M_{n-1}e_n$.

\item
\label{EP:2SidedIdeal}
$X_{n+1}:=M_n e_n M_n$ is a 2-sided ideal of $M_{n+1}$, and thus $M_{n+1}$ splits as a direct sum of von Neumann algebras $X_{n+1}\oplus Y_{n+1}$.
(In \cite[Thm.~4.1.4 and Thm.~4.6.3]{MR999799}, $Y_{n+1}$ is the so-called `new stuff'.)
By convention, we define $Y_0 = M_0$ and $Y_1 = M_1$, so that $X_0 = (0)$ and $X_1 = (0)$.

\item
\label{EP:BasicContruction}
The map $ae_n b\mapsto ap_n b$ gives a $*$-isomorphism from $X_{n+1}=M_n e_n M_n$ to $\langle M_n , p_n\rangle=M_np_nM_n$, the Jones basic construction of $M_{n-1} \subseteq M_n$ acting on $L^2(M_n,\tr_n)$.

\item
\label{EP:OtherMarkovDef}
Under the isomorphism $X_{n+1} \cong M_n p_n M_n$, the canonical non-normalized trace $\Tr_{n+1}$ on the Jones basic construction algebra $M_np_nM_n$ satisfying $\Tr_{n+1}(ap_nb) = \tr_n(ab)$ for $a,b\in M_n$ equals $d^2 \tr_{n+1}|_{X_{n+1}}$.

\item
\label{EP:NewStuff}
If $y\in Y_{n+1}$ and $x\in X_{n}$, then $yx = 0$ in $M_{n+1}$.
Hence $E_{n+1}(Y_{n+1}) \subseteq Y_{n}$.
(``The new stuff comes only from the old new stuff" \cite{MR999799}.)

\item
\label{EP:FiniteDepth}
If $Y_n =(0)$, then $Y_{k} = (0)$ for all $k\geq n$.

\end{enumerate}
\end{prop}

\begin{proof} 
\mbox{}
\begin{enumerate}[label={\rm(EP\arabic*)}]
\item
By \ref{eq:MarkovIndex}, $d^2E_{n+1}(ye_n) = y $, so the proposed map has a left inverse.

\item
This follows directly from \ref{eq:MarkovPullDown} and \ref{EP:Injective}.

\item
By \ref{eq:MarkovIndex}, for $x\in M_n$, we have $\tr_{n+1}(xe_n) = \tr_n(E_{n+1}(xe_n)) = \tr_n(x E_{n+1}(e_n)) = d^{-2} \tr_n(x)$.

\item
By \ref{eq:MarkovPullDown}, $e_n M_{n+1} e_n= e_nM_n e_n$.
By \ref{eq:MarkovImplement}, $e_n M_n e_n = M_{n-1}e_n$.

\item
That $M_ne_nM_n$ is a 2-sided ideal is equivalent to \ref{eq:MarkovPullDown} as in Remark \ref{pulldowniff}.

\item
It suffices to show the map is injective, which also shows it is well-defined. 
Suppose $\sum a_i p_n b_i = 0$.
Then for all $a,b\in M_n$, we have $0=p_na\left(\sum a_i p_n b_i\right) bp_n = \sum E_{n}(aa_i)E_n(b_ib)p_n$, and therefore $\sum E_{n}(aa_i)E_n(b_ib) = 0$ as $M_n \ni x\mapsto xp_n \in \langle M_n, p_n\rangle$ is injective by \ref{EP:Injective} applied to the Jones tower for $M_{n-1} \subset (M_n,\tr_n)$, which is a Markov tower.
Hence 
$$
0 = \sum E_{n}(aa_i)E_n(b_ib)e_n = e_na\left(\sum a_i e_n b_i\right) be_n
$$ 
for all $a,b\in M_n$, and thus $\sum a_i e_n b_i = 0$, so the map is injective. 

\item
For $a,b\in M_n$, by \ref{EP:MarkovTraces}, 
$\Tr_{n+1}(ap_n b) = \tr_n(ab) = \tr_n(ba) = d^2\tr_{n+1}(bae_n) = d^2 \tr_{n+1}(ae_n b)$.

\item
Since $X_0 = (0)$ and $X_1 = (0)$ by definition, we may assume $n\geq 2$.
As in the proof of \cite[Thm.~4.6.3.vi]{MR999799}, we may assume $y$ is a central projection in $M_{n+1}$ such that $y e_{n} = 0$.
Then for all $ae_{n-1} b \in X_n$, by \ref{eq:MarkovJonesProjections}, $y ae_{n-1} b = d^2 yae_{n-1} e_{n}e_{n-1} b = d^2 ae_{n-1} ye_{n} e_{n-1} b = 0$.
The final claim follows from $z_{n}E_{n+1}(y) = E_{n+1}(z_n y)= 0$ where $z_n$ is the central support of $e_{n-1}$ in $M_n$.

\item
This follows immediately from \ref{EP:NewStuff}.
\qedhere

\end{enumerate}
\end{proof}

\begin{rem}
\label{rem:InfiniteDimensionalMarkovTower}
 The foregoing observations all hold in the case where the $M_n$ are arbitrary tracial von Neumann algebras. In this paper, we restrict our attention to the finite dimensional case because of the following, in which we obtain a principal graph for a Markov tower. To generalize to the infinite case, a measure-theoretic replacement for the principal graph would need to be introduced. 
\end{rem}

Notice that by \ref{EP:BasicContruction}, the Bratteli diagram for the inclusion $M_{n}\subset M_{n+1}$ consists of the reflection of the Bratteli diagram for the inclusion $M_{n-1} \subset M_n$, together with possibly some new edges and vertices corresponding to simple summands of $Y_{n+1}$. 
By \ref{EP:NewStuff}, the new vertices at level $n+1$ only connect to the vertices that were new at level $n$. This leads to the following definition.

\begin{defn}
The \emph{principal graph} of the Markov tower $M_\bullet$ consists of the \emph{new} vertices at every level $n$ of the Bratteli diagram, together with all the edges connecting them.
A Markov tower is said to have \emph{finite depth} if the principal graph is finite.
\end{defn}

It follows that a Markov tower has finite depth if and only if there is $n\in \bbN$ such that $Y_n = (0)$, as in \ref{EP:FiniteDepth}. 
Let $M_\bullet$ be a Markov tower with finite depth, and take the minimal integer $n\in \bbN$ such that $Y_n=(0)$. 
Notice that for $k<n$, the Bratteli diagram of $M_k\subseteq M_{k+1}$ contains the reflection of the Bratteli diagram of $M_{k-1}\subseteq M_{k}$, along with additional vertices and edges which are part of the principal graph. 
At the base, all of the Bratteli diagram for $M_0\subseteq M_1$ is part of the principal graph. 
We can therefore `unravel' the Bratteli diagram for $M_{n}\subseteq M_{n+1}$ to obtain the principal graph for the Markov tower $M_\bullet$.

\begin{fact}\label{Fact:BratteliPrincipal}
 If a Markov tower $M_\bullet$ has finite depth and $n\in \bbN$ is such that $Y_n=(0)$, then for $k\geq n$, there is a canonical graph isomorphism between the principal graph of $M_\bullet$ and the Bratteli diagram for $M_{k}\subseteq M_{k+1}$.
\end{fact}

\begin{defn}
The principal graph $\Gamma$ of a Markov tower $M_\bullet$ has a \emph{quantum dimension function} $\dim : V(\Gamma) \to \bbR_{>0}$ given as follows.
Let $v \in V(\Gamma)$, and let $p\in M_k$ be a minimal projection with $k$ minimal corresponding to the vertex $v$.
We define $\dim(v) :=d^k \tr_k(p)$, and we note this dimension is independent of the choice of $p \in M_k$ representing $v$.
Moreover, the quantum dimension function $\dim$ satisfies the Frobenius-Perron property
\begin{equation}
\label{eq:QuantumDimension}
d\cdot \dim(v) 
=
\sum_{w\sim v}
\dim(w)
\end{equation}
where we write $w\sim v$ to mean $w$ is connected to $v$, and the above sum is taken with multiplicity.
\end{defn}

\subsection{Examples of Markov towers}

We discuss various examples of Markov towers in great detail.

\begin{ex}
The Temperley-Lieb-Jones algebras of modulus $d\geq 2$ with the usual Jones projections and Markov traces form a Markov tower with principal graph $A_{\infty}$.
\end{ex}

\begin{ex}
\label{ex:MarkovTowerFromRightModule}
Suppose $(\cM,m, \Tr^\cM)$ is a cyclic pivotal $\cT\cL\cJ(d)$-module $\Cstar$ category.
Let $X\in \cT\cL\cJ(d)$ be the generating object corresponding to the unshaded-shaded strand.
As described in the introduction, we get a Markov tower by defining  
$M_n:=\End_{\cC}(m\vartriangleleft \Xalt)$, $\tr_n := \Tr^\cM_{m\vartriangleleft \Xalt}(\id_{m\vartriangleleft \Xalt})^{-1} \Tr^\cM_{m\vartriangleleft \Xalt}$, and Jones projections depending on parity by 
\begin{align*}
e_{2k+1}
&=
\begin{tikzpicture}[baseline]
 \fill[ctwoshading] (-.3,-.5) rectangle (-.7,0.5);
 \draw[thick,red] (-0.3,-.5) -- (-0.3,0.5);
 \draw (0,-.5) -- (0,.5);
 \filldraw[shaded] (.3,.5) arc (-180:0:.3cm);
 \filldraw[shaded] (.3,-.5) arc (180:0:.3cm);
 \node at (.2,0){\tiny{$2k$}};
\end{tikzpicture} 
\hspace{.3cm}
:=
d^{-1}
\big(
\id_m \vartriangleleft \id_{(X\otimes \overline{X})^{\otimes k}} \otimes (\coev_X \circ \coev_X^\dag)
\big)
\in M_{2k+2}
\\
e_{2k+2}
&=
\begin{tikzpicture}[baseline]
 \fill[ctwoshading] (-.3,-.5) rectangle (-.7,0.5);
 \fill[shaded] (0,-.5) rectangle (1.2,.5);
 \draw[thick,red] (-0.3,-.5) -- (-0.3,0.5);
 \draw (0,-.5) -- (0,.5);
 \filldraw[unshaded] (.3,.5) arc (-180:0:.3cm);
 \filldraw[unshaded] (.3,-.5) arc (180:0:.3cm);
 \node at (.4,0){\tiny{$2k+1$}};
\end{tikzpicture} 
:=
d^{-1}
\big(
\id_m \vartriangleleft \id_X \otimes \id_{(\overline{X}\otimes X)^{\otimes k}} \otimes (\ev_X^\dag \circ \ev_X)
\big)
\in M_{2k+3}.
\end{align*}
The principal graph of $M_\bullet$ is precisely the fusion graph of $\cM$ with respect to $X$.
\end{ex}

\begin{ex}
\label{example:MarkovTowerFromPlanarModule}
We obtain the equivalent connected right planar module for the subfactor planar algebra $\cT\cL\cJ(d)_\bullet$ to Example \ref{ex:MarkovTowerFromRightModule} under Theorem \ref{thm:ModuleEquivalence} as follows.
We define $M_k : = \cM_k$ with its $\dag$-algebra structure and faithful tracial state $\tr_k$ from Definition \ref{def:RightPlanarModule}.
Jones projections are defined depending on parity by
$$
e_{2k+1}
:=
\begin{tikzpicture}[baseline]
 \draw (0,-.5) -- (0,.5);
 \filldraw[shaded] (.3,.5) arc (-180:0:.3cm);
 \filldraw[shaded] (.3,-.5) arc (180:0:.3cm);
 \node at (.2,0){\tiny{$2k$}};
 \halfcircle{}{(-.2,0)}{.5}{1.1}{}
\end{tikzpicture} 
\qquad
e_{2k+2,+}
:=
\begin{tikzpicture}[baseline]
 \halfcircle{shaded}{(-.2,0)}{.5}{1.1}{}
 \fill[unshaded] (-.18,-.48) rectangle (-.005,.48);
 \draw (0,-.5) -- (0,.5);
 \filldraw[unshaded] (.3,.48) arc (-180:0:.3cm);
 \filldraw[unshaded] (.3,-.48) arc (180:0:.3cm);
 \node at (.4,0){\tiny{$2k+1$}};
\end{tikzpicture} 
\,.
$$
\end{ex}

\begin{lem}
 \label{lem:FiniteDepthAlgebraHasFiniteDepthModule}
Suppose $\cP_\bullet$ is a finite depth subfactor planar algebra and $\cM_\bullet$ is a right planar module for $\cP_\bullet$.
Then the associated Markov tower $M_\bullet$ has finite depth, with $\operatorname{depth}(M_\bullet) \leq \operatorname{depth}(\cP_\bullet)$.
\end{lem}
\begin{proof}
Let $r$ be minimal such that $\cP_{r+1,+} = \cP_{r,+} e_{r,+}\cP_{r,+}$, and let $\{b\}$ be a Pimsner-Popa basis for $\cP_{r+1,+}$ over $\cP_{r,+}$ so that $\sum_b b e_{r,+} b^* = 1_{\cP_{r+1,+}}$.
Since 
 \[1_{M_{r+1}}=
 \begin{tikzpicture}[baseline]
  \halfcircle{}{(0,0)}{.5}{.5}{}
  \draw (.35,-.5) -- (.35,.5);
  \node at (.7,0) {\tiny{$r+1$}};
 \end{tikzpicture}
 =
 \begin{tikzpicture}[baseline]
  \halfcircle{}{(0,0)}{.5}{1}{}
  \ncircle{}{(.5,-.1)}{.25}{180}{}
  \draw (.5,-.5) -- (.5,.5);
  \node at (.8,.27) {\tiny{$r+1$}};
 \end{tikzpicture}
 =1_{\cP_{r+1,+}}\text{,}\]
 we have that
$\{b\}$ is a Pimsner-Popa basis for $M_{r+1}$ over $M_r$.
Hence $M_\bullet$ has finite depth by \ref{EP:FiniteDepth}.
The last claim follows immediately.
\end{proof}

\begin{defn}
\label{def:StronglyMarkovInclusion}
Recall from \cite{MR1278111} that an inclusion of finite von Neumann algebras $A_0\subset A_1$ with a faithful normal tracial state $\tr_1$ on $A_1$ is called a \emph{Markov inclusion} 
if
the canonical faithful normal semifinite trace on the Jones basic construction $A_2 = JA_0'J = \langle A_1, e_1\rangle \subset B(L^2(A_1, \tr_1))$ given by the extension of $xe_1y \mapsto \tr_1(xy)$ is finite and $\Tr_2(1)^{-1}\Tr_2|_{A_1}=\tr_1$.

Following \cite{MR2812459}, we call such an inclusion \emph{strongly Markov} if moreover there is a \emph{Pimsner-Popa basis} for $A_1$ over $A_0$, which is a finite subset $\{b\}\subset A_1$ such that $1_{A_2} = \sum_b b e_1b^*$.
This is equivalent to $x = \sum_b b E_{A_0}(b^*x)$ for all $x\in A_1$, and also to $A_2 = A_1e_1A_1$ by \cite[Prop.~3(b)]{MR561983}
(see also \cite{MR996807}).

Given a strongly Markov inclusion $A_0\subset (A_1, \tr_1)$, 
its \emph{Watatani index} \cite{MR996807} is the scalar $[A_1:A_0]:=\Tr_2(1) = \sum_b bb^* $.
We refer the reader to \cite[1.1.4(c)]{MR1278111} for other equivalent properties for the Watatani index in the presence of a Pimsner-Popa basis. 
We may iterate the Jones basic construction to get a tower of von Neumann algebras $(A_n ,\tr_n, e_{n+1})_{n\geq 0}$ with faithful tracial states such that each inclusion $A_{n}\subset (A_{n+1}, \tr_{n+1})$ is strongly Markov with index $[A_{n+1}: A_n] = [A_1:A_0]$ \cite{MR2812459}.
By \cite[Prop.~2.7.3]{MR996807}, the relative commutants $A_i' \cap A_j$ for $i\leq j$ are always finite dimensional von Neumann algebras.
\end{defn}

\begin{nota}
From this point on, we reserve the notation $A_\bullet=(A_n,\tr_n, e_{n+1})_{n\geq 0}$ for the Jones tower of a strongly Markov inclusion of tracial von Neumann algebras and $M_\bullet = (M_n, \tr_n, e_{n+1})_{n\geq 0}$ for a Markov tower.
\end{nota}

\begin{ex}
 Given a strongly Markov inclusion of tracial von Neumann algebras $A_0\subset (A_1, \tr_1)$, its Jones tower $(A_n, \tr_n, e_{n+1})_{n\geq 0}$ is a (possibly infinite dimensional) Markov tower, as in Remark \ref{rem:InfiniteDimensionalMarkovTower}.

Taking the relative commutant with $A_0$, we get a Markov tower of \emph{finite dimensional} von Neumann algebras $(A_0'\cap A_n , \tr_n|_{A_0'\cap A_n} , e_{n+1})_{n\geq 0}$.
Similarly, $(A_1'\cap A_{n+1} , \tr_{n+1}|_{A_1'\cap A_{n+1}} , e_{n+2})_{n\geq 0}$ is a Markov tower of finite dimensional von Neumann algebras.
\end{ex}

We now classify all connected Markov towers in terms of pointed bipartite graphs and quantum dimension functions.

\begin{ex}
Suppose $(\Gamma, v)$ is a locally finite pointed bipartite graph with countably many vertices, and $\dim : V(\Gamma) \to \bbR_{>0}$ is a quantum dimension function satisfying \eqref{eq:QuantumDimension}.
We construct a connected Markov tower $M_\bullet$ by defining $M_0 = \bbC$, and inductively constructing each $M_k$ as dictated by the principal graph $\Gamma$ starting at $v$ in the usual way \cite{MR999799,MR1473221}.
We define the trace vector for $M_k$ by normalizing the vector obtained from $\dim$ applied to the minimal projections appearing at level $k$.

It is straightforward to check that $M_\bullet$ has principal graph $\Gamma$ with basepoint $v$ corresponding to $1_{M_0}$.
Moreover, by construction, the quantum dimension function of $M_\bullet$ is exactly $\dim$.
\end{ex}

Indeed, the above example can be easily generalized to the following result.

\begin{prop}
\label{prop:ClassificationOfMarkovTowers}
Connected Markov towers $M_\bullet$ are classified up to $*$-isomorphism by pointed bipartite graphs $(\Gamma,v)$ with a quantum dimension function $\dim: V(\Gamma) \to \bbR_{>0}$ satisfying \eqref{eq:QuantumDimension}.
\end{prop}

\subsection{Operations on Markov towers to produce new Markov towers}
\label{sec:OperationsOnMarkovTowers}

In this section, we describe various operations on a Markov tower $M_\bullet = (M_n, \tr_n, e_{n+1})_{n\geq 0}$ which yield new Markov towers.
We begin with shifting and compressing the tower.
We then study the multistep tower.
For each of these operations, we discuss how the principal graph changes. 

We omit the proof of the following straightforward proposition.
\begin{prop}[Shifting a Markov tower]
\label{prop:ShiftMarkovTower}
Suppose $M_\bullet=(M_n, \tr_n, e_{n+1})_{n\geq 0}$ is a Markov tower.
For any $k\geq 1$, $M_{\bullet+k}:=(M_{n+k}, \tr_{n+k}, e_{n+k+1})_{n\geq 0}$ is also a Markov tower.
\end{prop}

\begin{remark}
 \label{Rem:ShiftTowerEffects}
Notice that shifting a Markov tower simply truncates the Bratteli diagram, and by Fact \ref{Fact:BratteliPrincipal}, the principal graph is unchanged.
\end{remark}

Given a Markov tower $M_\bullet$, we obtain another Markov tower by compression by a non-zero projection $p\in P(M_0)$.
First, for all $n\geq 0$, we define a faithful trace $\tr_n^p$ on $pM_n p$ by
\begin{equation}
\label{eq:CompressedTrace}
\tr^p_n(x) := \tr_n(p)^{-1}\tr_n(pxp).
\end{equation}
It is straightforward to verify that the unique trace-preserving conditional expectation is given by 
\begin{equation}
\label{eq:CompressedConditionalExpectation}
E^p_n : pM_np \to pM_{n-1}p
\qquad
\qquad
E^p_n(pxp) := E_n(pxp) = pE_n(x)p
\end{equation}
Notice that since $[e_n,p] = 0$ for all $n\in \bbN$, for all $pxp \in pM_n p$ we have 
\begin{equation}
\label{eq:CompressionImplementsConditionalExpectation}
e_np (pxp) e_np = p e_nxe_np = pE_n(x)e_np = E_n^p(pxp)e_np,
\end{equation}
so the conditional expectation is implemented by $e_np$.

\begin{prop}
Suppose $M_\bullet = (M_n,\tr_n, e_{n+1})_{n\geq0}$ is a Markov tower of finite dimensional von Neumann algebras and $p\in P(M_0)$ is a nonzero projection.
Then $pMp_\bullet:= (pM_np, \tr_n^p, pe_{n+1})_{n\geq0}$ is a Markov tower, where $\tr_n^p$ is defined as in \eqref{eq:CompressedTrace}.
\end{prop}
\begin{proof}
First, it is easy to see that the projections $(pe_n)_{n\geq 1}$ satisfy the Temperley-Lieb-Jones relations \ref{eq:MarkovJonesProjections}, since $[e_n,p]=0$ for all $n\geq 0$.
That $pe_n$ implements the trace-preserving conditional expectation $pM_np \to pM_{n-1}p$ as in \ref{eq:MarkovImplement} was shown above in \eqref{eq:CompressionImplementsConditionalExpectation}.
Using \eqref{eq:CompressedConditionalExpectation}, this immediately implies that $E_{n+1}^p(pe_n) = pE_{n+1}(e_n) = d^{-2}p = d^{-2} 1_{M_n}$, so \ref{eq:MarkovIndex} holds.
Finally, for all $n\geq 1$, $pM_{n+1}p (pe_n) = pM_{n+1}e_np = pM_ne_np = pM_npe_n$, so we have \ref{eq:MarkovPullDown}.
\end{proof}

\begin{remark}
 \label{Rem:CompressionTowerEffects}
We can determine the Bratteli diagram and principal graph for $pMp_\bullet$, as follows. 
If $p$ has central support $1$, then the Bratteli diagram is unchanged. 
In general, any vertices on the bottom row corresponding to simple summands of $M_0$ where $p$ does not have support disappear, as well as those edges no longer supported from below. 
By proceeding up the tower and, at each level, removing those vertices and edges no longer supported from below, we obtain the Bratteli diagram for $pMp_\bullet$.
\end{remark}

\begin{nota}
\label{nota:TLJK Diagrams}
We will make heavy use of the string diagrammatic representation of Temperley-Lieb-Jones diagrams.
Ordinarily, for subfactors and planar algebras, Kauffman diagrams \cite{MR899057} are drawn with strings going from \emph{bottom to top}. 
We put the number $k$ above or next to a strand to denote a bundle of $k$ parallel strands, and the label is omitted for single strands.
For example, the generators $E_i = de_i$ are represented by
$$
E_i
=
\begin{tikzpicture}[baseline = .3cm]
 \draw[thick, rounded corners = 5pt] (0,0) rectangle (2,.8);
 \draw (.2,0) -- (.2,.8);
 \draw (1,0) -- (1,.8);
 \draw (.4,.8) arc (-180:0:.2cm);
 \draw (.4,0) arc (180:0:.2cm);
 \node at (.35,.4) {\scriptsize{$i$}};
 \node at (1.5,.4) {\scriptsize{$n{-}i{-}2$}};
\end{tikzpicture}\,.
$$
Of particular importance will be the \emph{cabled/multi-step} Jones projections from \cite{MR965748} which were of importance in \cite{MR1424954,MR2812459}:
\begin{align*}
f^{j+k}_j
&:=
d^{k(k-1)}(e_{j+k}e_{j+k-1}\cdots e_{j+1})(e_{j+k+1}e_{j+k}\cdots e_{j+2})\cdots(e_{j+2k-1}e_{j+2k-2}\cdots e_{j+k})
\\
F^{j+k}_j
&:=
\begin{tikzpicture}[baseline = .3cm]
 \draw[thick, rounded corners = 5pt] (0,0) rectangle (1,.8);
 \draw (.2,0) -- (.2,.8);
 \draw (.4,.8) arc (-180:0:.2cm);
 \draw (.4,0) arc (180:0:.2cm);
 \node at (.4,1) {\scriptsize{$k$}};
 \node at (.4,-.2) {\scriptsize{$k$}};
 \node at (.2,1) {\scriptsize{$j$}};
\end{tikzpicture}
=
d^k f^{j+k}_j.
\end{align*}
We record the following relation for later use:
\begin{equation}
\label{eq:MultistepRelation}
f^{j+k}_j
=
d^{k(k-1)}(e_{j+k}e_{j+k+1}\cdots e_{j+2k-2}e_{j+2k-1})
\cdot\,
\begin{tikzpicture}[baseline = .3cm]
 \draw[thick, rounded corners = 5pt] (0,0) rectangle (2,.8);
 \draw (.2,0) -- (.2,.8);
 \draw (1.8,0) -- (1.8,.8);
 \draw (.6,.8) arc (-180:0:.25cm);
 \draw (.6,0)  .. controls ++(90:.35cm) and ++(270:.35cm) .. (1.5,.8);
 \draw (1,0) arc (180:0:.25cm);
 \node at (.6,1) {\scriptsize{$k{-}1$}};
 \node at (1,-.2) {\scriptsize{$k{-}1$}};
 \node at (.2,-.2) {\scriptsize{$j$}};
\end{tikzpicture}
\end{equation}

Now suppose we fix $j\geq 0$ and $k\geq 1$.
For $n\in \bbN$, define the $k$-cabled Jones projections
$
g_n := f^{j+nk}_{j+(n-1)k}
$.
It is straightforward to verify using Kauffman's diagrammatic calculus for Temperley-Lieb-Jones algebras that the projections $(g_n)_{n\in \bbN}$ satisfy the Temperley-Lieb-Jones relations \ref{eq:MarkovJonesProjections} with $d^{-2}$ replaced with $d^{-2k}$.
\end{nota}

We now show that taking every $k$-th algebra in a Markov tower gives us another Markov tower.

\begin{prop}
\label{prop:MultistepJonesProjections}
Suppose $M_\bullet=(M_n, \tr_n, e_{n+1})$ is a Markov tower, and let $j\geq 0$ and $k\geq 1$.
Define $g_n \in M_{j+(n+1)k}$ as in Notation \ref{nota:TLJK Diagrams}.
Then $M_{j+k\bullet}:=(M_{j+nk}, \tr_{j+nk}, g_{n+1})_{n\geq 0}$ is a Markov tower.
\end{prop}
\begin{proof}
We saw Condition \ref{eq:MarkovJonesProjections} holds from the diagrammatic calculus, and Conditions \ref{eq:MarkovImplement} and \ref{eq:MarkovIndex} are straightforward induction arguments.

We prove \ref{eq:MarkovPullDown} by strong induction on $k$.
The base case $k=1$ is exactly \ref{eq:MarkovPullDown} for the original Markov tower.
Now suppose that \ref{eq:MarkovPullDown} holds for any multi-step towers with increment less than $k$.
Consider the multi-step tower of algebras $(M_{j+nk})_{n\geq 0}$, which has increment $k$.
By Proposition \ref{prop:ShiftMarkovTower}, we may assume $j=0$.
Using \eqref{eq:MultistepRelation} and \ref{eq:MarkovPullDown} for the original Markov tower, we have 
\begin{align*}
M_{(n+1)k} g_n 
&= 
M_{(n+1)k} f^{(n-1)k+k}_{(n-1)k} 
= 
M_{(n+1)k}
(e_{nk}e_{nk+1}\cdots e_{(n+1)k-2}e_{(n+1)k-1})
\cdot\,
\begin{tikzpicture}[baseline = .3cm]
 \draw[thick, rounded corners = 5pt] (0,0) rectangle (2,.8);
 \draw (.2,0) -- (.2,.8);
 \draw (1.8,0) -- (1.8,.8);
 \draw (.6,.8) arc (-180:0:.25cm);
 \draw (.6,0)  .. controls ++(90:.35cm) and ++(270:.35cm) .. (1.5,.8);
 \draw (1,0) arc (180:0:.25cm);
 \node at (.6,1) {\scriptsize{$k{-}1$}};
 \node at (1,-.2) {\scriptsize{$k{-}1$}};
 \node at (0,-.2) {\scriptsize{$(n{-}1)k$}};
\end{tikzpicture}
\\&=
M_{(n+1)k-1} 
e_{(n+1)k-1}
\cdot \,
\begin{tikzpicture}[baseline = .3cm]
 \draw[thick, rounded corners = 5pt] (0,0) rectangle (2,.8);
 \draw (.2,0) -- (.2,.8);
 \draw (1.8,0) -- (1.8,.8);
 \draw (.6,.8) arc (-180:0:.25cm);
 \draw (.6,0)  .. controls ++(90:.35cm) and ++(270:.35cm) .. (1.5,.8);
 \draw (1,0) arc (180:0:.25cm);
 \node at (.6,1) {\scriptsize{$k{-}1$}};
 \node at (1,-.2) {\scriptsize{$k{-}1$}};
 \node at (0,-.2) {\scriptsize{$(n{-}1)k$}};
\end{tikzpicture}
=
M_{(n+1)k-1} 
\begin{tikzpicture}[baseline = .3cm]
 \draw[thick, rounded corners = 5pt] (0,0) rectangle (2,.8);
 \draw (.2,0) -- (.2,.8);
 \draw (.8,.8) arc (-180:0:.25cm);
 \draw (1.5,.8) arc (-180:0:.15cm);
 \draw (1,0) arc (180:0:.3cm);
 \node at (.8,1) {\scriptsize{$k{-}1$}};
 \node at (1.6,-.2) {\scriptsize{$k$}};
 \node at (.2,-.2) {\scriptsize{$(n{-}1)k$}};
\end{tikzpicture}
\,.
\end{align*}
Since we may perform isotopy in the Temperley-Lieb-Jones subalgebra of $M_{(n+1)k}$, we may decompose the diagram on the right hand side as follows:
$$
\begin{tikzpicture}[baseline = .3cm]
 \draw[thick, rounded corners = 5pt] (0,0) rectangle (2,.8);
 \draw (.2,0) -- (.2,.8);
 \draw (.8,.8) arc (-180:0:.25cm);
 \draw (1.5,.8) arc (-180:0:.15cm);
 \draw (1,0) arc (180:0:.3cm);
 \node at (.8,1) {\scriptsize{$k{-}1$}};
 \node at (1.6,-.2) {\scriptsize{$k$}};
 \node at (.4,-.2) {\scriptsize{$(n{-}1)k$}};
\end{tikzpicture}
=
d^{k-2}
\begin{tikzpicture}[baseline = 1.4cm]
 \draw[thick, rounded corners = 5pt] (0,0) rectangle (2.2,3);
 \draw (.2,0) -- (.2,3);
 \draw (.6,3) arc (-180:0:.2cm);
 \draw (.4,3) .. controls ++(270:.35cm) and ++(90:.35cm) .. (1,2) arc (-180:0:.3cm) -- (1.6,3);
 \draw (1.8,3) -- (1.8,2) arc (0:-180:.5cm) arc (0:180:.2cm) -- (.4,1) arc (-180:0:.2cm) arc (180:0:.4cm) arc (-180:0:.2cm) -- (2,3);
 \draw (1.2,1) circle (.2cm);
 \draw (.8,0) arc (180:0:.4cm);
 \draw[dashed] (0,1) -- (2.2,1);
 \draw[dashed] (0,2) -- (2.2,2);
 \node at (1.2,.7) {\scriptsize{$k{-}2$}};
 \node at (1.2,2.5) {\scriptsize{$k{-}2$}};
 \node at (1.75,.2) {\scriptsize{$k$}};
 \node at (.2,-.2) {\scriptsize{$(n{-}1)k$}};
 \node at (4.2,1.5) {$\leftarrow d\cdot f^{(n-1)k+1+(k-1)}_{(n-1)k+1}$};
\end{tikzpicture}
$$
By the induction hypothesis, we have
$$
M_{(n+1)k-1}
\begin{tikzpicture}[baseline = .3cm]
 \draw[thick, rounded corners = 5pt] (0,0) rectangle (2,.8);
 \draw (.2,0) -- (.2,.8);
 \draw (.8,.8) arc (-180:0:.25cm);
 \draw (1.5,.8) arc (-180:0:.15cm);
 \draw (1,0) arc (180:0:.3cm);
 \node at (.8,1) {\scriptsize{$k{-}1$}};
 \node at (1.6,-.2) {\scriptsize{$k$}};
 \node at (.4,-.2) {\scriptsize{$(n{-}1)k$}};
\end{tikzpicture}
=
M_{nk+(k-1)}
\begin{tikzpicture}[baseline = 1.4cm]
 \draw[thick, rounded corners = 5pt] (0,0) rectangle (2.2,3);
 \draw (.2,0) -- (.2,3);
 \draw (.6,3) arc (-180:0:.2cm);
 \draw (.4,3) .. controls ++(270:.35cm) and ++(90:.35cm) .. (1,2) arc (-180:0:.3cm) -- (1.6,3);
 \draw (1.8,3) -- (1.8,2) arc (0:-180:.5cm) arc (0:180:.2cm) -- (.4,1) arc (-180:0:.2cm) arc (180:0:.4cm) arc (-180:0:.2cm) -- (2,3);
 \draw (1.2,1) circle (.2cm);
 \draw (.8,0) arc (180:0:.4cm);
 \draw[dashed] (0,1) -- (2.2,1);
 \draw[dashed] (0,2) -- (2.2,2);
 \node at (1.2,.7) {\scriptsize{$k{-}2$}};
 \node at (1.2,2.5) {\scriptsize{$k{-}2$}};
 \node at (1.75,.2) {\scriptsize{$k$}};
 \node at (.2,-.2) {\scriptsize{$(n{-}1)k$}};
\end{tikzpicture}
=
M_{nk}\,
\begin{tikzpicture}[baseline = .9cm]
 \draw[thick, rounded corners = 5pt] (0,0) rectangle (2.2,2);
 \draw (.2,0) -- (.2,2);
 \draw (1,2) arc (-180:0:.3cm);
 \draw (1.8,2) arc (0:-180:.5cm); 
 \draw (.4,2) -- (.4,1) arc (-180:0:.2cm) arc (180:0:.4cm) arc (-180:0:.2cm) -- (2,2);
 \draw (1.2,1) circle (.2cm);
 \draw (.8,0) arc (180:0:.4cm);
 \draw[dashed] (0,1) -- (2.2,1);
 \node at (1.2,.7) {\scriptsize{$k{-}2$}};
 \node at (1.2,2.2) {\scriptsize{$k{-}2$}};
 \node at (1.75,.2) {\scriptsize{$k$}};
 \node at (.2,-.2) {\scriptsize{$(n{-}1)k$}};
\end{tikzpicture}
=
M_{nk} g_n.
$$
This completes the proof.
\end{proof}

\begin{remark}
 \label{Rem:MultistepTowerEffects}
If $M_\bullet$ is a Markov tower, then we know from Proposition \ref{prop:MultistepJonesProjections} that 
$M_{k\bullet}$ is also a Markov tower. 
The Bratteli diagram for $M_{k\bullet}$ can be read off the original Bratteli diagram quite easily: the vertices of the level of the Bratteli diagram corresponding to $M_{kn}$ are the same in both towers, while the number of edges between two vertices in the new diagram is the number of upward paths between those vertices in the old diagram.
Note that, since a vertex of the multistep principal graph may belong to the `old stuff' in the original Bratteli diagram, the number of edges between adjacent vertices in the multistep principal graph is not simply the number of paths in the original principal graph.
In the case where $k$ is odd, taking the $k$-step basic construction therefore collapses the vertices of each $k$ levels of the principal graph into one level; when $k$ is even, we lose the odd part of the principal graph entirely, but aside from this, the situation is the same.
\end{remark}

\subsection{The projection category of a Markov tower}

We now define the category of projections of a Markov tower.

\begin{defn}
\label{def:MarkovProjections}
Let $M_\bullet = (M_n, \tr_n, e_{n+1})_{n\geq 0}$ be a Markov tower.
We define the category $\cM$ to be the unitary Karoubi completion (formally adding orthogonal direct sums, and then taking the orthogonal projection completion) of the $\Cstar$ category $\cM_0$ with finite dimensional $\Hom$-spaces defined as follows.
\begin{itemize}
\item
The objects of $\cM_0$ are the symbols $[n]$ for $n\geq 0$.
\item
Given $n,k \geq 0$, we define 
$\cM_0([n] \to [n+2k]):= M_{n+k}$
and
$\cM_0([n+2k] \to [n]) := M_{n+k}$.
\item
The identity morphism in $\cM_0([n] \to [n])$ is $1_{M_n}$.
\item
For $x\in \cM_0([n]\to [n+2k])$ or $x\in \cM_0([n+2k] \to [n])$, we define $x^\dag := x^* \in M_{n+k}$.
\item
We define composition in three cases. In each, we make use of Temperley-Lieb diagrams. The Jones projections of a Markov tower generate an image of the Markov tower of Temperley-Lieb algebras, so each such diagram represents a well-defined element of $M_\bullet$. 
\begin{enumerate}[label={\rm(C\arabic*)}]
\item
\label{compose:upup}
If $x\in \cM_0([n] \to [n+2i])$ and $y\in \cM_0([n+2i] \to [n+2i+2j])$, we define 
$$
y\circ x
:=
d^iE^{n+2i+j}_{n+i+j}
\left(
\raisebox{.3cm}{$
\underbrace{
y\cdot x \cdot
\begin{tikzpicture}[baseline = .3cm]
 \draw[thick, rounded corners = 5pt] (0,0) rectangle (1.5,.8);
 \draw (.2,0) -- (.2,.8);
 \draw (.5,.8) arc (-180:0:.2cm);
 \draw (.5,0) .. controls ++(90:.3cm) and ++(270:.3cm) .. (1.2,.8);
 \draw (.8,0) arc (180:0:.2cm);
 \node at (.2,1) {\scriptsize{$n$}};
 \node at (.5,1) {\scriptsize{$i$}};
 \node at (1.2,1) {\scriptsize{$j$}};
 \node at (.8,-.2) {\scriptsize{$i$}};
\end{tikzpicture}
}_{\in M_{n+2i+j}}
$}
\right)
\in M_{n+i+j}
=
\cM_0([n] \to [n+2i+2j]).
$$
We define the composite $x^\dag \circ y^\dag := (y\circ x)^\dag$, which defines composition 
$$
\cM_0([n+2i+2j] \to [n+2i]) \otimes \cM_0([n+2i] \to [n]) \to \cM_0([n+2i+2j] \to [n]).
$$
To show $x^\dag \circ y^\dag$ is well-defined, we check that when $i=j=0$,
\begin{equation}
\label{eq:n to n to n}
x^\dag\circ y^\dag= x^* y^* = (y x)^* = (y\circ x)^\dag.
\end{equation}

\item
\label{compose:updown}
If $x\in \cM_0([n] \to [n+2i+2j])$ and $y \in \cM_0([n+2i+2j] \to [n+2i])$
we define
$$
y\circ x
:=
d^iE^{n+2i + j}_{n+i}
\left(
\raisebox{.3cm}{$
\underbrace{
y\cdot x\cdot
\begin{tikzpicture}[baseline = .3cm]
 \draw[thick, rounded corners = 5pt] (0,0) rectangle (1.5,.8);
 \draw (.2,0) -- (.2,.8);
 \draw (.8,.8) arc (-180:0:.2cm);
 \draw (1.2,0) .. controls ++(90:.3cm) and ++(270:.3cm) .. (.5,.8);
 \draw (.5,0) arc (180:0:.2cm);
 \node at (.2,1) {\scriptsize{$n$}};
 \node at (.5,1) {\scriptsize{$j$}};
 \node at (.8,1) {\scriptsize{$i$}};
 \node at (.5,-.2) {\scriptsize{$i$}};
\end{tikzpicture}
}_{\in M_{n+2i +j}}
$}
\right)
\in
M_{n+i}
=
\cM_0([n] \to [n+2i]).
$$
As above, we define the composite $x^\dag \circ y^\dag := (y\circ x)^\dag$, which defines composition
$$
\cM_0([m] \to [m+2k]) \otimes \cM_0([n+2j] \to [n]) \to \cM_0([m] \to [n]).
$$
To show that $x^\dag \circ y^\dag$ is well-defined, we check that when $i=0$, 
\begin{equation}
\label{eq:n to n+2k to n}
x^\dag \circ y^\dag 
=
E^{n+j}_{n}(x^* y^*)
=
E^{n+j}_{n}((yx)^*) 
=
E^{n+j}_{n}(y x)^*
=
(y\circ x)^\dag.
\end{equation}

\item
\label{compose:downup}
If $x\in \cM_0([n+2i] \to [n])$ and $y\in \cM_0([n] \to [n+2i+2j])$, we define
$$
y\circ x 
:=
y\cdot
d^{-i}\,
\begin{tikzpicture}[baseline = .3cm]
 \draw[thick, rounded corners = 5pt] (0,0) rectangle (1.5,.8);
 \draw (.2,0) -- (.2,.8);
 \draw (.8,.8) arc (-180:0:.2cm);
 \draw (1.2,0) .. controls ++(90:.3cm) and ++(270:.3cm) .. (.5,.8);
 \draw (.5,0) arc (180:0:.2cm);
 \node at (.2,1) {\scriptsize{$n$}};
 \node at (.5,1) {\scriptsize{$j$}};
 \node at (.8,1) {\scriptsize{$i$}};
 \node at (.5,-.2) {\scriptsize{$i$}};
\end{tikzpicture}
\cdot x
\in
M_{n+2i+j}
=
\cM_0([n+2i] \to [n+2i+2j]).
$$
As above, we define the composite $x^\dag \circ y^\dag := (y\circ x)^\dag$, which defines composition
$$
\cM_0([n+2i+2j] \to [n]) \otimes \cM_0([n] \to [n+2i]) \to \cM_0([n+2i+2j] \to [n+2i]).
$$
To show that $x^\dag \circ y^\dag$ is well-defined, we check that when $j=0$, 
\begin{equation}
\label{eq:n+2j to n to n+2j}
x^\dag \circ y^\dag 
=
x^* \cdot
d^{-i}\,
\begin{tikzpicture}[baseline = .3cm]
 \draw[thick, rounded corners = 5pt] (0,0) rectangle (1.2,.8);
 \draw (.2,0) -- (.2,.8);
 \draw (.5,.8) arc (-180:0:.2cm);
 \draw (.5,0) arc (180:0:.2cm);
 \node at (.2,1) {\scriptsize{$n$}};
 \node at (.5,1) {\scriptsize{$i$}};
 \node at (.5,-.2) {\scriptsize{$i$}};
\end{tikzpicture}
\cdot y^*
=
\left(
y \cdot
d^{-i}\,
\begin{tikzpicture}[baseline = .3cm]
 \draw[thick, rounded corners = 5pt] (0,0) rectangle (1.2,.8);
 \draw (.2,0) -- (.2,.8);
 \draw (.5,.8) arc (-180:0:.2cm);
 \draw (.5,0) arc (180:0:.2cm);
 \node at (.2,1) {\scriptsize{$n$}};
 \node at (.5,1) {\scriptsize{$i$}};
 \node at (.5,-.2) {\scriptsize{$i$}};
\end{tikzpicture}
\cdot x
\right)^*
=
(y\circ x)^\dag.
\end{equation}
\end{enumerate}
\end{itemize}
Showing that composition is associative directly from the definitions above is a highly non-trivial exercise using the axioms \ref{eq:MarkovJonesProjections} -- \ref{eq:MarkovPullDown} of a Markov tower.
A better way to prove associativity is to prove that each $4\times 4$ (possibly non-associative) \emph{linking algebra} \cite{MR808930} 
$$
\cL:=
\resizebox{.95\hsize}{!}{$
\left(\begin{array}{llll}
\cM_0([n] \to[n])
&
\cM_0([n{+}2i] \to[n])
&
\cM_0([n{+}2(i{+}j)] \to[n])
&
\cM_0([n{+}2(i{+}j{+}k)] \to[n])
\\
\cM_0([n] \to[n{+}2i])
&
\cM_0([n{+}2i] \to[n{+}2i])
&
\cM_0([n{+}2(i{+}j)] \to[n{+}2i])
&
\cM_0([n{+}2(i{+}j{+}k)] \to[n{+}2i])
\\
\cM_0([n] \to[n{+}2(i{+}j)])
&
\cM_0([n{+}2i] \to[n{+}2(i{+}j)])
&
\cM_0([n{+}2(i{+}j)] \to[n{+}2(i{+}j)])
&
\cM_0([n{+}2(i{+}j{+}k)] \to[n{+}2(i{+}j)])
\\
\cM_0([n] \to[n{+}2(i{+}j{+}k)])
&
\cM_0([n{+}2i] \to[n{+}2(i{+}j{+}k)])
&
\cM_0([n{+}2(i{+}j)] \to[n{+}2(i{+}j{+}k)])
&
\cM_0([n{+}2(i{+}j{+}k)] \to[n{+}2(i{+}j{+}k)])
\end{array}\right)
$}
$$ 
is $\dag/*$-isomorphic to a von Neumann algebra, which is necessarily associative!
This technique also offers the advantage that it simultaneously proves $\cM_0$ is $\Cstar$.\footnote{
Just as being a $\Cstar$ algebra is a property of a complex $*$-algebra, being a $\Cstar$ category is a property of a $\bbC$-linear dagger category.
}

Notice we have an equality of sets
\begin{equation}
\label{eq:LinkingAlgebra}
\cL=
\left(
\begin{array}{llll}
M_n & M_{n{+}i} & M_{n{+}i{+}j} & M_{n{+}i{+}j{+}k}
\\
M_{n{+}i} & M_{n{+}2i} & M_{n{+}2i{+}j} & M_{n{+}2i{+}j{+}k}
\\
M_{n{+}i{+}j} & M_{n{+}2i{+}j} & M_{n{+}2i{+}2j} & M_{n{+}2i{+}2j{+}k}
\\
M_{n{+}i{+}j{+}k} & M_{n{+}2i{+}j{+}k} & M_{n{+}2i{+}2j{+}k} & M_{n{+}2i{+}2j{+}2k}
\end{array}
\right).
\end{equation}
We define the following map entry-wise; that is, for an element $x \in \cL$, we plug $x_{ab}$ into the input disk in the $ab$-th entry of the map $\pi :\cL \to p\operatorname{Mat}_4(M_{n+2i+2j+2k})p$ given by
\begin{equation}
\label{eq:4x4Map}
\resizebox{.92\hsize}{!}{
$
\begin{pmatrix}
d^{{-}i{-}j{-}k}\,
\begin{tikzpicture}[baseline=-.1cm]
 \draw (.2,.6) arc (-180:0:.2cm);
 \draw (.2,-.6) arc (180:0:.2cm);
 \draw (.8,.6) arc (-180:0:.2cm);
 \draw (.8,-.6) arc (180:0:.2cm);
 \draw (1.4,.6) arc (-180:0:.2cm);
 \draw (1.4,-.6) arc (180:0:.2cm);
 \draw (0,-.6) -- (0,.6);
 \roundNbox{unshaded}{(0,0)}{.2}{0}{0}{}
 \node at (0,-.8) {\scriptsize{$n$}};
 \node at (0,.8) {\scriptsize{$n$}};
 \node at (.2,.8) {\scriptsize{$i$}};
 \node at (.6,.8) {\scriptsize{$i$}};
 \node at (.8,.8) {\scriptsize{$j$}};
 \node at (1.2,.8) {\scriptsize{$j$}};
 \node at (1.4,.8) {\scriptsize{$k$}};
 \node at (1.8,.8) {\scriptsize{$k$}};
 \node at (.2,-.8) {\scriptsize{$i$}};
 \node at (.6,-.8) {\scriptsize{$i$}};
 \node at (.8,-.8) {\scriptsize{$j$}};
 \node at (1.2,-.8) {\scriptsize{$j$}};
 \node at (1.4,-.8) {\scriptsize{$k$}};
 \node at (1.8,-.8) {\scriptsize{$k$}};
\end{tikzpicture}
& 
d^{{-}i{-}j{-}k}\,
\begin{tikzpicture}[baseline=-.1cm]
 \draw (.2,.6) arc (-180:0:.2cm);
 \draw (.2,-.6) -- (.2,.1) arc (180:0:.2cm) -- (.6,-.6);
 \draw (.8,.6) arc (-180:0:.2cm);
 \draw (.8,-.6) arc (180:0:.2cm);
 \draw (1.4,.6) arc (-180:0:.2cm);
 \draw (1.4,-.6) arc (180:0:.2cm);
 \draw (0,-.6) -- (0,.6);
 \roundNbox{unshaded}{(0,-.1)}{.2}{0}{.2}{}
 \node at (0,-.8) {\scriptsize{$n$}};
 \node at (0,.8) {\scriptsize{$n$}};
 \node at (.2,.8) {\scriptsize{$i$}};
 \node at (.6,.8) {\scriptsize{$i$}};
 \node at (.8,.8) {\scriptsize{$j$}};
 \node at (1.2,.8) {\scriptsize{$j$}};
 \node at (1.4,.8) {\scriptsize{$k$}};
 \node at (1.8,.8) {\scriptsize{$k$}};
 \node at (.2,-.8) {\scriptsize{$i$}};
 \node at (.6,-.8) {\scriptsize{$i$}};
 \node at (.8,-.8) {\scriptsize{$j$}};
 \node at (1.2,-.8) {\scriptsize{$j$}};
 \node at (1.4,-.8) {\scriptsize{$k$}};
 \node at (1.8,-.8) {\scriptsize{$k$}};
\end{tikzpicture}
&
d^{{-}i{-}j{-}k}\,
\begin{tikzpicture}[baseline=-.1cm]
 \draw (.2,-.6) -- (.2,0) .. controls ++(90:.45cm) and ++(90:.45cm) .. (1.2,0) -- (1.2,-.6);
 \draw (.4,-.6) -- (.4,0) .. controls ++(90:.2cm) and ++(90:.2cm) ..  (1,0) -- (1,-.6);
 \draw (.2,.6) arc (-180:0:.2cm);
 \draw (.8,.6) arc (-180:0:.2cm);
 \draw (1.4,.6) arc (-180:0:.2cm);
 \draw (1.4,-.6) arc (180:0:.2cm);
 \draw (0,-.6) -- (0,.6);
 \roundNbox{unshaded}{(0,-.2)}{.2}{0}{.4}{}
 \node at (0,-.8) {\scriptsize{$n$}};
 \node at (0,.8) {\scriptsize{$n$}};
 \node at (.2,.8) {\scriptsize{$i$}};
 \node at (.6,.8) {\scriptsize{$i$}};
 \node at (.8,.8) {\scriptsize{$j$}};
 \node at (1.2,.8) {\scriptsize{$j$}};
 \node at (1.4,.8) {\scriptsize{$k$}};
 \node at (1.8,.8) {\scriptsize{$k$}};
 \node at (.2,-.8) {\scriptsize{$i$}};
 \node at (.4,-.8) {\scriptsize{$j$}};
 \node at (1,-.8) {\scriptsize{$j$}};
 \node at (1.2,-.8) {\scriptsize{$i$}};
 \node at (1.4,-.8) {\scriptsize{$k$}};
 \node at (1.8,-.8) {\scriptsize{$k$}};
\end{tikzpicture}
&
d^{{-}i{-}j{-}k}\,
\begin{tikzpicture}[baseline=-.1cm]
 \draw (.2,-.6) -- (.2,0) .. controls ++(90:.45cm) and ++(90:.45cm) .. (1.8,0) -- (1.8,-.6);
 \draw (.4,-.6) -- (.4,0) .. controls ++(90:.3cm) and ++(90:.3cm) ..  (1.4,0) -- (1.4,-.6);
 \draw (.6,-.6) -- (.6,0) .. controls ++(90:.15cm) and ++(90:.15cm) ..  (1,0) -- (1,-.6);
 \draw (.2,.6) arc (-180:0:.2cm);
 \draw (.8,.6) arc (-180:0:.2cm);
 \draw (1.4,.6) arc (-180:0:.2cm);
 \draw (0,-.6) -- (0,.6);
 \roundNbox{unshaded}{(0,-.2)}{.2}{0}{.6}{}
 \node at (0,-.8) {\scriptsize{$n$}};
 \node at (0,.8) {\scriptsize{$n$}};
 \node at (.2,.8) {\scriptsize{$i$}};
 \node at (.6,.8) {\scriptsize{$i$}};
 \node at (.8,.8) {\scriptsize{$j$}};
 \node at (1.2,.8) {\scriptsize{$j$}};
 \node at (1.4,.8) {\scriptsize{$k$}};
 \node at (1.8,.8) {\scriptsize{$k$}};
 \node at (.2,-.8) {\scriptsize{$i$}};
 \node at (.4,-.8) {\scriptsize{$j$}};
 \node at (.6,-.8) {\scriptsize{$k$}};
 \node at (1,-.8) {\scriptsize{$k$}};
 \node at (1.4,-.8) {\scriptsize{$j$}};
 \node at (1.8,-.8) {\scriptsize{$i$}};
\end{tikzpicture}
\\
d^{{-}i{-}j{-}k}\,
\begin{tikzpicture}[baseline=-.1cm, yscale = -1]
 \draw (.2,.6) arc (-180:0:.2cm);
 \draw (.8,.6) arc (-180:0:.2cm);
 \draw (.2,-.6) -- (.2,.1) arc (180:0:.2cm) -- (.6,-.6);
 \draw (.8,-.6) arc (180:0:.2cm);
 \draw (1.4,.6) arc (-180:0:.2cm);
 \draw (1.4,-.6) arc (180:0:.2cm);
 \draw (0,-.6) -- (0,.6);
 \roundNbox{unshaded}{(0,-.1)}{.2}{0}{.2}{}
 \node at (0,-.8) {\scriptsize{$n$}};
 \node at (0,.8) {\scriptsize{$n$}};
 \node at (.2,.8) {\scriptsize{$i$}};
 \node at (.6,.8) {\scriptsize{$i$}};
 \node at (.8,.8) {\scriptsize{$j$}};
 \node at (1.2,.8) {\scriptsize{$j$}};
 \node at (1.4,.8) {\scriptsize{$k$}};
 \node at (1.8,.8) {\scriptsize{$k$}};
 \node at (.2,-.8) {\scriptsize{$i$}};
 \node at (.6,-.8) {\scriptsize{$i$}};
 \node at (.8,-.8) {\scriptsize{$j$}};
 \node at (1.2,-.8) {\scriptsize{$j$}};
 \node at (1.4,-.8) {\scriptsize{$k$}};
 \node at (1.8,-.8) {\scriptsize{$k$}};
\end{tikzpicture}
&
d^{{-}j{-}k}\,
\begin{tikzpicture}[baseline=-.1cm]
 \draw (.6,.6) arc (-180:0:.2cm);
 \draw (.6,-.6) arc (180:0:.2cm);
 \draw (1.4,.6) arc (-180:0:.2cm);
 \draw (1.4,-.6) arc (180:0:.2cm);
 \draw (0,-.6) -- (0,.6);
 \roundNbox{unshaded}{(0,0)}{.2}{0}{0}{}
 \node at (0,-.8) {\scriptsize{$n{+}2i$}};
 \node at (0,.8) {\scriptsize{$n{+}2i$}};
 \node at (.6,.8) {\scriptsize{$j$}};
 \node at (1,.8) {\scriptsize{$j$}};
 \node at (1.4,.8) {\scriptsize{$k$}};
 \node at (1.8,.8) {\scriptsize{$k$}};
 \node at (.6,-.8) {\scriptsize{$j$}};
 \node at (1,-.8) {\scriptsize{$j$}};
 \node at (1.4,-.8) {\scriptsize{$k$}};
 \node at (1.8,-.8) {\scriptsize{$k$}};
\end{tikzpicture}
&
d^{{-}j{-}k}\,
\begin{tikzpicture}[baseline=-.1cm]
 \draw (.6,.6) arc (-180:0:.2cm);
 \draw (.6,-.6) -- (.6,.1) arc (180:0:.2cm) -- (1,-.6);
 \draw (1.4,.6) arc (-180:0:.2cm);
 \draw (1.4,-.6) arc (180:0:.2cm);
 \draw (0,-.6) -- (0,.6);
 \roundNbox{unshaded}{(0,-.1)}{.2}{0}{.6}{}
 \node at (0,-.8) {\scriptsize{$n{+}2i$}};
 \node at (0,.8) {\scriptsize{$n{+}2i$}};
 \node at (.6,.8) {\scriptsize{$j$}};
 \node at (.6,-.8) {\scriptsize{$j$}};
 \node at (1.4,.8) {\scriptsize{$k$}};
 \node at (1.8,.8) {\scriptsize{$k$}};
 \node at (1,.8) {\scriptsize{$j$}};
 \node at (1,-.8) {\scriptsize{$j$}};
 \node at (1.4,-.8) {\scriptsize{$k$}};
 \node at (1.8,-.8) {\scriptsize{$k$}};
\end{tikzpicture}
&
d^{{-}j{-}k}\,
\begin{tikzpicture}[baseline=-.1cm]
 \draw (.6,-.6) -- (.6,0) .. controls ++(90:.45cm) and ++(90:.45cm) ..  (1.6,0) -- (1.6,-.6);
 \draw (.8,-.6) -- (.8,0) .. controls ++(90:.3cm) and ++(90:.3cm) ..  (1.2,0) -- (1.2,-.6);
 \draw (.6,.6) arc (-180:0:.2cm);
 \draw (1.2,.6) arc (-180:0:.2cm);
 \draw (0,-.6) -- (0,.6);
 \roundNbox{unshaded}{(0,-.2)}{.2}{0}{.8}{}
 \node at (0,-.8) {\scriptsize{$n{+}2i$}};
 \node at (0,.8) {\scriptsize{$n{+}2i$}};
 \node at (.6,.8) {\scriptsize{$j$}};
 \node at (1,.8) {\scriptsize{$j$}};
 \node at (1.2,.8) {\scriptsize{$k$}};
 \node at (1.6,.8) {\scriptsize{$k$}};
 \node at (.6,-.8) {\scriptsize{$j$}};
 \node at (.8,-.8) {\scriptsize{$k$}};
 \node at (1.2,-.8) {\scriptsize{$k$}};
 \node at (1.6,-.8) {\scriptsize{$j$}};
\end{tikzpicture}
\\
d^{{-}i{-}j{-}k}\,
\begin{tikzpicture}[baseline=-.1cm, yscale = -1]
 \draw (.2,-.6) -- (.2,0) .. controls ++(90:.45cm) and ++(90:.45cm) .. (1.2,0) -- (1.2,-.6);
 \draw (.4,-.6) -- (.4,0) .. controls ++(90:.2cm) and ++(90:.2cm) ..  (1,0) -- (1,-.6);
 \draw (.2,.6) arc (-180:0:.2cm);
 \draw (.8,.6) arc (-180:0:.2cm);
 \draw (1.4,.6) arc (-180:0:.2cm);
 \draw (1.4,-.6) arc (180:0:.2cm);
 \draw (0,-.6) -- (0,.6);
 \roundNbox{unshaded}{(0,-.2)}{.2}{0}{.4}{}
 \node at (0,-.8) {\scriptsize{$n$}};
 \node at (0,.8) {\scriptsize{$n$}};
 \node at (.2,.8) {\scriptsize{$i$}};
 \node at (.6,.8) {\scriptsize{$i$}};
 \node at (.8,.8) {\scriptsize{$j$}};
 \node at (1.2,.8) {\scriptsize{$j$}};
 \node at (1.4,.8) {\scriptsize{$k$}};
 \node at (1.8,.8) {\scriptsize{$k$}};
 \node at (.2,-.8) {\scriptsize{$i$}};
 \node at (.4,-.8) {\scriptsize{$j$}};
 \node at (1,-.8) {\scriptsize{$j$}};
 \node at (1.2,-.8) {\scriptsize{$i$}};
 \node at (1.4,-.8) {\scriptsize{$k$}};
 \node at (1.8,-.8) {\scriptsize{$k$}};
\end{tikzpicture}
&
d^{-j-k}
\,
\begin{tikzpicture}[baseline=-.1cm, yscale =-1]
 \draw (.6,.6) arc (-180:0:.2cm);
 \draw (.6,-.6) -- (.6,.1) arc (180:0:.2cm) -- (1,-.6);
 \draw (1.4,.6) arc (-180:0:.2cm);
 \draw (1.4,-.6) arc (180:0:.2cm);
 \draw (0,-.6) -- (0,.6);
 \roundNbox{unshaded}{(0,-.1)}{.2}{0}{.6}{}
 \node at (0,-.8) {\scriptsize{$n{+}2i$}};
 \node at (0,.8) {\scriptsize{$n{+}2i$}};
 \node at (.6,.8) {\scriptsize{$j$}};
 \node at (.6,-.8) {\scriptsize{$j$}};
 \node at (1.4,.8) {\scriptsize{$k$}};
 \node at (1.8,.8) {\scriptsize{$k$}};
 \node at (1,.8) {\scriptsize{$j$}};
 \node at (1,-.8) {\scriptsize{$j$}};
 \node at (1.4,-.8) {\scriptsize{$k$}};
 \node at (1.8,-.8) {\scriptsize{$k$}};
\end{tikzpicture}
&
d^{-k}
\begin{tikzpicture}[baseline=-.1cm]
 \draw (0,-.6) -- (0,.6);
 \draw (.8,.6) arc (-180:0:.2cm);
 \draw (.8,-.6) arc (180:0:.2cm);
 \roundNbox{unshaded}{(0,0)}{.2}{0}{0}{}
 \node at (0,-.8) {\scriptsize{$n{+}2i{+}2j$}};
 \node at (0,.8) {\scriptsize{$n{+}2i{+}2j$}};
 \node at (.8,.8) {\scriptsize{$k$}};
 \node at (1.2,.8) {\scriptsize{$k$}};
 \node at (.8,-.8) {\scriptsize{$k$}};
 \node at (1.2,-.8) {\scriptsize{$k$}};
\end{tikzpicture}
&
d^{-k}\,
\begin{tikzpicture}[baseline=-.1cm]
 \draw (0,-.6) -- (0,.6);
 \draw (1,.6) arc (-180:0:.2cm);
 \draw (1,-.6) -- (1,.1) arc (180:0:.2cm) -- (1.4,-.6);
 \roundNbox{unshaded}{(0,-.1)}{.2}{0}{1}{}
 \node at (0,-.8) {\scriptsize{$n{+}2i{+}2j$}};
 \node at (0,.8) {\scriptsize{$n{+}2i{+}2j$}};
 \node at (1,.8) {\scriptsize{$k$}};
 \node at (1.4,.8) {\scriptsize{$k$}};
 \node at (1,-.8) {\scriptsize{$k$}};
 \node at (1.4,-.8) {\scriptsize{$k$}};
\end{tikzpicture}
\\
d^{{-}i{-}j{-}k}\,
\begin{tikzpicture}[baseline=-.1cm, yscale = -1]
 \draw (.2,-.6) -- (.2,0) .. controls ++(90:.45cm) and ++(90:.45cm) .. (1.8,0) -- (1.8,-.6);
 \draw (.4,-.6) -- (.4,0) .. controls ++(90:.3cm) and ++(90:.3cm) ..  (1.4,0) -- (1.4,-.6);
 \draw (.6,-.6) -- (.6,0) .. controls ++(90:.15cm) and ++(90:.15cm) ..  (1,0) -- (1,-.6);
 \draw (.2,.6) arc (-180:0:.2cm);
 \draw (.8,.6) arc (-180:0:.2cm);
 \draw (1.4,.6) arc (-180:0:.2cm);
 \draw (0,-.6) -- (0,.6);
 \roundNbox{unshaded}{(0,-.2)}{.2}{0}{.6}{}
 \node at (0,-.8) {\scriptsize{$n$}};
 \node at (0,.8) {\scriptsize{$n$}};
 \node at (.2,.8) {\scriptsize{$i$}};
 \node at (.6,.8) {\scriptsize{$i$}};
 \node at (.8,.8) {\scriptsize{$j$}};
 \node at (1.2,.8) {\scriptsize{$j$}};
 \node at (1.4,.8) {\scriptsize{$k$}};
 \node at (1.8,.8) {\scriptsize{$k$}};
 \node at (.2,-.8) {\scriptsize{$i$}};
 \node at (.4,-.8) {\scriptsize{$j$}};
 \node at (.6,-.8) {\scriptsize{$k$}};
 \node at (1,-.8) {\scriptsize{$k$}};
 \node at (1.4,-.8) {\scriptsize{$j$}};
 \node at (1.8,-.8) {\scriptsize{$i$}};
\end{tikzpicture}
&
d^{-j-k}
\begin{tikzpicture}[baseline=-.1cm, yscale = -1]
 \draw (.6,-.6) -- (.6,0) .. controls ++(90:.45cm) and ++(90:.45cm) ..  (1.6,0) -- (1.6,-.6);
 \draw (.8,-.6) -- (.8,0) .. controls ++(90:.3cm) and ++(90:.3cm) ..  (1.2,0) -- (1.2,-.6);
 \draw (.6,.6) arc (-180:0:.2cm);
 \draw (1.2,.6) arc (-180:0:.2cm);
 \draw (0,-.6) -- (0,.6);
 \roundNbox{unshaded}{(0,-.2)}{.2}{0}{.8}{}
 \node at (0,-.8) {\scriptsize{$n{+}2i$}};
 \node at (0,.8) {\scriptsize{$n{+}2i$}};
 \node at (.6,.8) {\scriptsize{$j$}};
 \node at (1,.8) {\scriptsize{$j$}};
 \node at (1.2,.8) {\scriptsize{$k$}};
 \node at (1.6,.8) {\scriptsize{$k$}};
 \node at (.6,-.8) {\scriptsize{$j$}};
 \node at (.8,-.8) {\scriptsize{$k$}};
 \node at (1.2,-.8) {\scriptsize{$k$}};
 \node at (1.6,-.8) {\scriptsize{$j$}};
\end{tikzpicture}
&
d^{-k}
\begin{tikzpicture}[baseline=-.1cm, yscale = -1]
 \draw (0,-.6) -- (0,.6);
 \draw (1,.6) arc (-180:0:.2cm);
 \draw (1,-.6) -- (1,.1) arc (180:0:.2cm) -- (1.4,-.6);
 \roundNbox{unshaded}{(0,-.1)}{.2}{0}{1}{}
 \node at (0,-.8) {\scriptsize{$n{+}2i{+}2j$}};
 \node at (0,.8) {\scriptsize{$n{+}2i{+}2j$}};
 \node at (1,.8) {\scriptsize{$k$}};
 \node at (1.4,.8) {\scriptsize{$k$}};
 \node at (1,-.8) {\scriptsize{$k$}};
 \node at (1.4,-.8) {\scriptsize{$k$}};
\end{tikzpicture}
&
\begin{tikzpicture}[baseline=-.1cm, yscale =-1]
 \draw (0,-.6) -- (0,.6);
 \roundNbox{unshaded}{(0,0)}{.2}{0}{0}{}
 \node at (0,-.8) {\scriptsize{$n{+}2i{+}2j{+}2k$}};
 \node at (0,.8) {\scriptsize{$n{+}2i{+}2j{+}2k$}};
\end{tikzpicture}
\end{pmatrix}
$}
\end{equation}
where $p\in \operatorname{Mat}_4(M_{n+2i+2j+2k})$ is the following projection:
$$
p:=
\operatorname{diag}\left(
d^{{-}i{-}j{-}k}\,
\begin{tikzpicture}[baseline = -.1cm]
 \draw[thick, rounded corners = 5pt] (-.2,-.3) rectangle (2,.3);
 \draw (.2,.3) arc (-180:0:.2cm);
 \draw (.2,-.3) arc (180:0:.2cm);
 \draw (.8,.3) arc (-180:0:.2cm);
 \draw (.8,-.3) arc (180:0:.2cm);
 \draw (1.4,.3) arc (-180:0:.2cm);
 \draw (1.4,-.3) arc (180:0:.2cm);
 \draw (0,-.3) -- (0,.3);
 \node at (0,-.5) {\scriptsize{$n$}};
 \node at (0,.5) {\scriptsize{$n$}};
 \node at (.2,.5) {\scriptsize{$i$}};
 \node at (.6,.5) {\scriptsize{$i$}};
 \node at (.8,.5) {\scriptsize{$j$}};
 \node at (1.2,.5) {\scriptsize{$j$}};
 \node at (1.4,.5) {\scriptsize{$k$}};
 \node at (1.8,.5) {\scriptsize{$k$}};
 \node at (.2,-.5) {\scriptsize{$i$}};
 \node at (.6,-.5) {\scriptsize{$i$}};
 \node at (.8,-.5) {\scriptsize{$j$}};
 \node at (1.2,-.5) {\scriptsize{$j$}};
 \node at (1.4,-.5) {\scriptsize{$k$}};
 \node at (1.8,-.5) {\scriptsize{$k$}};
\end{tikzpicture}
\,,
d^{-j-k}
\begin{tikzpicture}[baseline = -.1cm]
 \draw[thick, rounded corners = 5pt] (-.2,-.3) rectangle (1.8,.3);
 \draw (.6,.3) arc (-180:0:.2cm);
 \draw (.6,-.3) arc (180:0:.2cm);
 \draw (1.2,.3) arc (-180:0:.2cm);
 \draw (1.2,-.3) arc (180:0:.2cm);
 \draw (0,-.3) -- (0,.3);
 \node at (0,-.5) {\scriptsize{$n{+}2i$}};
 \node at (0,.5) {\scriptsize{$n{+}2i$}};
 \node at (.6,.5) {\scriptsize{$j$}};
 \node at (1,.5) {\scriptsize{$j$}};
 \node at (1.2,.5) {\scriptsize{$k$}};
 \node at (1.6,.5) {\scriptsize{$k$}};
 \node at (.6,-.5) {\scriptsize{$j$}};
 \node at (1,-.5) {\scriptsize{$j$}};
 \node at (1.2,-.5) {\scriptsize{$k$}};
 \node at (1.6,-.5) {\scriptsize{$k$}};
\end{tikzpicture}
\,,
d^{-k}
\begin{tikzpicture}[baseline = -.1cm]
 \draw[thick, rounded corners = 5pt] (-.2,-.3) rectangle (1.4,.3);
 \draw (.8,.3) arc (-180:0:.2cm);
 \draw (.8,-.3) arc (180:0:.2cm);
 \draw (0,-.3) -- (0,.3);
 \node at (0,-.5) {\scriptsize{$n{+}2i{+}2j$}};
 \node at (0,.5) {\scriptsize{$n{+}2i{+}2j$}};
 \node at (.8,.5) {\scriptsize{$k$}};
 \node at (1.2,.5) {\scriptsize{$k$}};
 \node at (.8,-.5) {\scriptsize{$k$}};
 \node at (1.2,-.5) {\scriptsize{$k$}};
\end{tikzpicture}
\,,
1_{n+2i+2j+2k}
\right).
$$
In Proposition \ref{prop:InjectiveAlgebraMap} below, we verify the map $\pi$ is an injective unital algebra map satisfying $\pi(x^\dag) = \pi(x)^*$, and is thus an isomorphism onto its image.
Thus $\im(\pi)$ is a unital $*$-subalgebra of the finite dimensional von Neumann algebra $p\operatorname{Mat}_4(M_{n+2i+2j+2k})p$, which means $\im(\pi)$ is a von Neumann algebra by the finite dimensional bicommutant theorem \cite[Thm.~3.2.1]{JonesVNA}.
By looking at the $2\times 2$ and $3\times 3$ corners of the linking algebra $\cL$ and \eqref{eq:4x4Map},
we immediately see: 
\begin{itemize}
\item
$\dag$ is a dagger structure on $\cM_0$,
\item
for every $f\in \cM_0([n]\to [n+2j])$, there is a $g\in \cM_0([n] \to [n])$ and an $h\in \cM_0([n+2j]\to [n+2j])$ such that 
$f^\dag \circ f = g^\dag \circ g$ and $f\circ f^\dag = h^\dag \circ h$, and
\item
there are (pullback) norms on the (finite dimensional) $\Hom$-spaces $\cM_0([n]\to [n+2i])$ and $\cM_0([n+2j] \to [n])$ which are sub-multiplicative with respect to composition and which satisfy the $\Cstar$ axiom $\|f^\dag \circ f \| = \|f^2\|$.\footnote{
Notice that in a $\Cstar$ category, these norms can be recovered from spectral radii together with the positivity and $\Cstar$ axioms.
Thus these norms are \emph{not} part of the data of the $\Cstar$ category.
}
\end{itemize}
Hence $\cM_0$ is $\Cstar$, and thus so is its unitary Karoubi completion $\cM$.
\end{defn}

\begin{prop}
\label{prop:InjectiveAlgebraMap}
The map $\pi: \cL \to p\operatorname{Mat}_4(M_{n+2i+2j+2k})p$ from Definition \ref{def:MarkovProjections} is an injective unital algebra map such that $\pi(x^\dag) = \pi(x)^*$ for all $x\in \cL$.
\end{prop}

We begin with the following lemma.

 \begin{lem}
 \label{lem:LeftKink}
For all $x \in M_{n+2i+j}$,
$
d^i
E^{n+j+2i}_{n+j+i}
\left(
x
\cdot
\begin{tikzpicture}[baseline = .3cm]
 \draw[thick, rounded corners = 5pt] (0,0) rectangle (1.5,.8);
 \draw (.2,0) -- (.2,.8);
 \draw (.5,.8) arc (-180:0:.2cm);
 \draw (.5,0) .. controls ++(90:.3cm) and ++(270:.3cm) .. (1.2,.8);
 \draw (.8,0) arc (180:0:.2cm);
 \node at (.2,1) {\scriptsize{$n$}};
 \node at (.5,1) {\scriptsize{$i$}};
 \node at (1.2,1) {\scriptsize{$j$}};
 \node at (.8,-.2) {\scriptsize{$i$}};
\end{tikzpicture}
\right)
\cdot
\begin{tikzpicture}[baseline=.3cm]
 \draw[thick, rounded corners = 5pt] (0,0) rectangle (1.6,.8);
 \draw (.2,0) -- (.2,.8);
 \draw (0.4,.8) arc (-180:0:.5);
 \draw (0.6,.8) arc (-180:0:.3);
 \draw (0.4,0) arc (180:0:.2);
 \draw (1,0) arc (180:0:.2);
 \node at (.2,-.2) {\scriptsize{$n$}};
 \node at (.4,1) {\scriptsize{$j$}};
 \node at (.6,1) {\scriptsize{$i$}};
 \node at (.4,-.2) {\scriptsize{$i$}};
 \node at (1,-.2) {\scriptsize{$j$}};
\end{tikzpicture}
=
x\cdot
\begin{tikzpicture}[baseline=.3cm]
 \draw[thick, rounded corners = 5pt] (0,0) rectangle (1.6,.8);
 \draw (.2,0) -- (.2,.8);
 \draw (0.4,0) arc (180:0:.2);
 \draw (1,0) arc (180:0:.2);
 \draw (0.4,.8) arc (-180:0:.2);
 \draw (1,.8) arc (-180:0:.2);
 \node at (.2,-.2) {\scriptsize{$n$}};
 \node at (.4,-.2) {\scriptsize{$i$}};
 \node at (1,-.2) {\scriptsize{$j$}};
 \node at (.4,1) {\scriptsize{$i$}};
 \node at (1,1) {\scriptsize{$j$}};
\end{tikzpicture}
$\,.
\end{lem}

\begin{proof}

We calculate
\begin{align*}
x\cdot
\begin{tikzpicture}[baseline=.3cm]
 \draw[thick, rounded corners = 5pt] (0,0) rectangle (1.6,.8);
 \draw (.2,0) -- (.2,.8);
 \draw (0.4,0) arc (180:0:.2);
 \draw (1,0) arc (180:0:.2);
 \draw (0.4,.8) arc (-180:0:.2);
 \draw (1,.8) arc (-180:0:.2);
 \node at (.2,-.2) {\scriptsize{$n$}};
 \node at (.4,-.2) {\scriptsize{$i$}};
 \node at (1,-.2) {\scriptsize{$j$}};
 \node at (.4,1) {\scriptsize{$i$}};
 \node at (1,1) {\scriptsize{$j$}};
\end{tikzpicture}
&=
d^{i}
E^{n+j+3i}_{n+j+2i}\left(
\begin{tikzpicture}[baseline = .3cm]
 \draw[thick, rounded corners = 5pt] (0,0) rectangle (1.4,.8);
 \draw (.2,0) -- (.2,.8);
 \draw (.4,0) -- (.4,.8);
 \draw (.6,0) -- (.6,.8);
 \draw (.8,.8) arc (-180:0:.2cm);
 \draw (.8,0) arc (180:0:.2cm);
 \node at (.8,1) {\scriptsize{$i$}};
 \node at (.8,-.2) {\scriptsize{$i$}};
 \node at (.2,1) {\scriptsize{$n$}};
 \node at (.4,1) {\scriptsize{$i$}};
 \node at (.6,1) {\scriptsize{$j$}};
\end{tikzpicture}
\right)
\cdot
x\cdot
d^{-i}\,
\begin{tikzpicture}[baseline = .3cm]
 \draw[thick, rounded corners = 5pt] (0,0) rectangle (1.6,.8);
 \draw (.2,0) -- (.2,.8);
 \draw (1.4,0) -- (1.4,.8);
 \draw (.5,.8) arc (-180:0:.2cm);
 \draw (.5,0) .. controls ++(90:.3cm) and ++(270:.3cm) .. (1.2,.8);
 \draw (.8,0) arc (180:0:.2cm);
 \node at (.2,1) {\scriptsize{$n$}};
 \node at (.5,1) {\scriptsize{$i$}};
 \node at (1.2,1) {\scriptsize{$j$}};
 \node at (.8,-.2) {\scriptsize{$i$}};
 \node at (1.4,1) {\scriptsize{$j$}};
\end{tikzpicture}
\cdot
\begin{tikzpicture}[baseline=.3cm]
 \draw[thick, rounded corners = 5pt] (0,0) rectangle (1.6,.8);
 \draw (.2,0) -- (.2,.8);
 \draw (0.4,.8) arc (-180:0:.5);
 \draw (0.6,.8) arc (-180:0:.3);
 \draw (0.4,0) arc (180:0:.2);
 \draw (1,0) arc (180:0:.2);
 \node at (.2,-.2) {\scriptsize{$n$}};
 \node at (.4,1) {\scriptsize{$j$}};
 \node at (.6,1) {\scriptsize{$i$}};
 \node at (.4,-.2) {\scriptsize{$i$}};
 \node at (1,-.2) {\scriptsize{$j$}};
\end{tikzpicture}
\displaybreak[1]\\&=
E^{n+j+3i}_{n+j+2i}\left(
\begin{tikzpicture}[baseline = .3cm]
 \draw[thick, rounded corners = 5pt] (0,0) rectangle (1.4,.8);
 \draw (.2,0) -- (.2,.8);
 \draw (.4,0) -- (.4,.8);
 \draw (.6,0) -- (.6,.8);
 \draw (.8,.8) arc (-180:0:.2cm);
 \draw (.8,0) arc (180:0:.2cm);
 \node at (.8,1) {\scriptsize{$i$}};
 \node at (.8,-.2) {\scriptsize{$i$}};
 \node at (.2,1) {\scriptsize{$n$}};
 \node at (.4,1) {\scriptsize{$i$}};
 \node at (.6,1) {\scriptsize{$j$}};
\end{tikzpicture}
\cdot
x\cdot
\begin{tikzpicture}[baseline = .3cm]
 \draw[thick, rounded corners = 5pt] (0,0) rectangle (1.6,.8);
 \draw (.2,0) -- (.2,.8);
 \draw (1.4,0) -- (1.4,.8);
 \draw (.5,.8) arc (-180:0:.2cm);
 \draw (.5,0) .. controls ++(90:.3cm) and ++(270:.3cm) .. (1.2,.8);
 \draw (.8,0) arc (180:0:.2cm);
 \node at (.2,1) {\scriptsize{$n$}};
 \node at (.5,1) {\scriptsize{$i$}};
 \node at (1.2,1) {\scriptsize{$j$}};
 \node at (.8,-.2) {\scriptsize{$i$}};
 \node at (1.4,1) {\scriptsize{$i$}};
\end{tikzpicture}
\right)
\cdot
\begin{tikzpicture}[baseline=.3cm]
 \draw[thick, rounded corners = 5pt] (0,0) rectangle (1.6,.8);
 \draw (.2,0) -- (.2,.8);
 \draw (0.4,.8) arc (-180:0:.5);
 \draw (0.6,.8) arc (-180:0:.3);
 \draw (0.4,0) arc (180:0:.2);
 \draw (1,0) arc (180:0:.2);
 \node at (.2,-.2) {\scriptsize{$n$}};
 \node at (.4,1) {\scriptsize{$j$}};
 \node at (.6,1) {\scriptsize{$i$}};
 \node at (.4,-.2) {\scriptsize{$i$}};
 \node at (1,-.2) {\scriptsize{$j$}};
\end{tikzpicture}
\displaybreak[1]\\&=
E^{n+j+3i}_{n+j+2i}\left(
\begin{tikzpicture}[baseline = .3cm]
 \draw[thick, rounded corners = 5pt] (0,0) rectangle (1.4,.8);
 \draw (.2,0) -- (.2,.8);
 \draw (.4,0) -- (.4,.8);
 \draw (.6,0) -- (.6,.8);
 \draw (.8,.8) arc (-180:0:.2cm);
 \draw (.8,0) arc (180:0:.2cm);
 \node at (.8,1) {\scriptsize{$i$}};
 \node at (.8,-.2) {\scriptsize{$i$}};
 \node at (.2,1) {\scriptsize{$n$}};
 \node at (.4,1) {\scriptsize{$i$}};
 \node at (.6,1) {\scriptsize{$j$}};
\end{tikzpicture}
\cdot
x\cdot
\begin{tikzpicture}[baseline = -.5cm]
 \draw[thick, rounded corners = 5pt] (0,.8) rectangle (1.8,-1.6);
 \draw[dashed] (0,0) -- (1.8,0);
 \draw[dashed] (0,-.8) -- (1.8,-.8); 
 \draw (.2,-1.6) -- (.2,.8);
 \draw (1.6,.8) -- (1.6,0) arc (0:-180:.2cm) arc (0:180:.2cm) -- (.8,-.8) arc (-180:0:.2cm) arc (180:0:.2cm) -- (1.6,-1.6);
 \draw (.5,.8) arc (-180:0:.2cm);
 \draw (.5,-1.6) -- (.5,0) .. controls ++(90:.3cm) and ++(270:.3cm) .. (1.2,.8);
 \draw (.8,-1.6) arc (180:0:.2cm);
 \node at (.2,1) {\scriptsize{$n$}};
 \node at (.5,1) {\scriptsize{$i$}};
 \node at (1.2,1) {\scriptsize{$j$}};
 \node at (.8,-1.8) {\scriptsize{$i$}};
 \node at (1.6,1) {\scriptsize{$i$}};
\end{tikzpicture}
\right)
\cdot
\begin{tikzpicture}[baseline=.3cm]
 \draw[thick, rounded corners = 5pt] (0,0) rectangle (1.6,.8);
 \draw (.2,0) -- (.2,.8);
 \draw (0.4,.8) arc (-180:0:.5);
 \draw (0.6,.8) arc (-180:0:.3);
 \draw (0.4,0) arc (180:0:.2);
 \draw (1,0) arc (180:0:.2);
 \node at (.2,-.2) {\scriptsize{$n$}};
 \node at (.4,1) {\scriptsize{$j$}};
 \node at (.6,1) {\scriptsize{$i$}};
 \node at (.4,-.2) {\scriptsize{$i$}};
 \node at (1,-.2) {\scriptsize{$j$}};
\end{tikzpicture}
\displaybreak[1]\\&=
E^{n+j+3i}_{n+j+2i}\left(
d^i\,
E^{n+j+2i}_{n+j+i}\left(
x\cdot
\begin{tikzpicture}[baseline = .3cm]
 \draw[thick, rounded corners = 5pt] (0,0) rectangle (1.5,.8);
 \draw (.2,0) -- (.2,.8);
 \draw (.5,.8) arc (-180:0:.2cm);
 \draw (.5,0) .. controls ++(90:.3cm) and ++(270:.3cm) .. (1.2,.8);
 \draw (.8,0) arc (180:0:.2cm);
 \node at (.2,1) {\scriptsize{$n$}};
 \node at (.5,1) {\scriptsize{$i$}};
 \node at (1.2,1) {\scriptsize{$j$}};
 \node at (.8,-.2) {\scriptsize{$i$}};
\end{tikzpicture}
\right)
\cdot
\begin{tikzpicture}[baseline = -.1cm]
 \draw[thick, rounded corners = 5pt] (0,.8) rectangle (1.8,-.8);
 \draw[dashed] (0,0) -- (1.8,0);
 \draw (.2,-.8) -- (.2,.8);
 \draw (1.6,.8) -- (1.6,0) arc (0:-180:.2cm) arc (0:180:.2cm) -- (.8,-.8);
 \draw (.8,.8) arc (-180:0:.2cm);
 \draw (.5,-.8) -- (.5,.8);
 \draw (1.2,-.8) arc (180:0:.2cm);
 \node at (.2,1) {\scriptsize{$n$}};
 \node at (.5,1) {\scriptsize{$j$}};
 \node at (1.2,1) {\scriptsize{$i$}};
 \node at (1.6,1) {\scriptsize{$i$}};
 \node at (1.6,-1) {\scriptsize{$i$}};
\end{tikzpicture}
\right)
\cdot
\begin{tikzpicture}[baseline=.3cm]
 \draw[thick, rounded corners = 5pt] (0,0) rectangle (1.6,.8);
 \draw (.2,0) -- (.2,.8);
 \draw (0.4,.8) arc (-180:0:.5);
 \draw (0.6,.8) arc (-180:0:.3);
 \draw (0.4,0) arc (180:0:.2);
 \draw (1,0) arc (180:0:.2);
 \node at (.2,-.2) {\scriptsize{$n$}};
 \node at (.4,1) {\scriptsize{$j$}};
 \node at (.6,1) {\scriptsize{$i$}};
 \node at (.4,-.2) {\scriptsize{$i$}};
 \node at (1,-.2) {\scriptsize{$j$}};
\end{tikzpicture}
\displaybreak[1]\\&=
d^i
E^{n+j+2i}_{n+j+i}\left(
x\cdot
\begin{tikzpicture}[baseline = .3cm]
 \draw[thick, rounded corners = 5pt] (0,0) rectangle (1.5,.8);
 \draw (.2,0) -- (.2,.8);
 \draw (.5,.8) arc (-180:0:.2cm);
 \draw (.5,0) .. controls ++(90:.3cm) and ++(270:.3cm) .. (1.2,.8);
 \draw (.8,0) arc (180:0:.2cm);
 \node at (.2,1) {\scriptsize{$n$}};
 \node at (.5,1) {\scriptsize{$i$}};
 \node at (1.2,1) {\scriptsize{$j$}};
 \node at (.8,-.2) {\scriptsize{$i$}};
\end{tikzpicture}
\right)
\cdot
E^{n+j+3i}_{n+j+2i}\left(
\begin{tikzpicture}[baseline = -.1cm]
 \draw[thick, rounded corners = 5pt] (0,.8) rectangle (1.8,-.8);
 \draw[dashed] (0,0) -- (1.8,0);
 \draw (.2,-.8) -- (.2,.8);
 \draw (1.6,.8) -- (1.6,0) arc (0:-180:.2cm) arc (0:180:.2cm) -- (.8,-.8);
 \draw (.8,.8) arc (-180:0:.2cm);
 \draw (.5,-.8) -- (.5,.8);
 \draw (1.2,-.8) arc (180:0:.2cm);
 \node at (.2,1) {\scriptsize{$n$}};
 \node at (.5,1) {\scriptsize{$j$}};
 \node at (1.2,1) {\scriptsize{$i$}};
 \node at (1.6,1) {\scriptsize{$i$}};
 \node at (1.6,-1) {\scriptsize{$i$}};
\end{tikzpicture}
\right)
\cdot
\begin{tikzpicture}[baseline=.3cm]
 \draw[thick, rounded corners = 5pt] (0,0) rectangle (1.6,.8);
 \draw (.2,0) -- (.2,.8);
 \draw (0.4,.8) arc (-180:0:.5);
 \draw (0.6,.8) arc (-180:0:.3);
 \draw (0.4,0) arc (180:0:.2);
 \draw (1,0) arc (180:0:.2);
 \node at (.2,-.2) {\scriptsize{$n$}};
 \node at (.4,1) {\scriptsize{$j$}};
 \node at (.6,1) {\scriptsize{$i$}};
 \node at (.4,-.2) {\scriptsize{$i$}};
 \node at (1,-.2) {\scriptsize{$j$}};
\end{tikzpicture}
\displaybreak[1]\\&=
d^i
E^{n+j+2i}_{n+j+i}\left(
x\cdot
\begin{tikzpicture}[baseline = .3cm]
 \draw[thick, rounded corners = 5pt] (0,0) rectangle (1.5,.8);
 \draw (.2,0) -- (.2,.8);
 \draw (.5,.8) arc (-180:0:.2cm);
 \draw (.5,0) .. controls ++(90:.3cm) and ++(270:.3cm) .. (1.2,.8);
 \draw (.8,0) arc (180:0:.2cm);
 \node at (.2,1) {\scriptsize{$n$}};
 \node at (.5,1) {\scriptsize{$i$}};
 \node at (1.2,1) {\scriptsize{$j$}};
 \node at (.8,-.2) {\scriptsize{$i$}};
\end{tikzpicture}
\right)
\cdot
d^{-i}\,
\begin{tikzpicture}[baseline = .3cm]
 \draw[thick, rounded corners = 5pt] (0,.8) rectangle (1.8,0);
 \draw (.2,0) -- (.2,.8);
 \draw (1.5,.8) -- (1.5,0);
 \draw (.8,.8) arc (-180:0:.2cm);
 \draw (.5,0) -- (.5,.8);
 \draw (.8,0) arc (180:0:.2cm);
 \node at (.2,1) {\scriptsize{$n$}};
 \node at (.5,1) {\scriptsize{$j$}};
 \node at (1.2,1) {\scriptsize{$i$}};
 \node at (1.5,1) {\scriptsize{$j$}};
 \node at (.8,-.2) {\scriptsize{$i$}};
\end{tikzpicture}
\cdot
\begin{tikzpicture}[baseline=.3cm]
 \draw[thick, rounded corners = 5pt] (0,0) rectangle (1.6,.8);
 \draw (.2,0) -- (.2,.8);
 \draw (0.4,.8) arc (-180:0:.5);
 \draw (0.6,.8) arc (-180:0:.3);
 \draw (0.4,0) arc (180:0:.2);
 \draw (1,0) arc (180:0:.2);
 \node at (.2,-.2) {\scriptsize{$n$}};
 \node at (.4,1) {\scriptsize{$j$}};
 \node at (.6,1) {\scriptsize{$i$}};
 \node at (.4,-.2) {\scriptsize{$i$}};
 \node at (1,-.2) {\scriptsize{$j$}};
\end{tikzpicture}
\displaybreak[1]\\&=
d^i
E^{n+j+2i}_{n+j+i}
\left(
x
\cdot
\begin{tikzpicture}[baseline = .3cm]
 \draw[thick, rounded corners = 5pt] (0,0) rectangle (1.5,.8);
 \draw (.2,0) -- (.2,.8);
 \draw (.5,.8) arc (-180:0:.2cm);
 \draw (.5,0) .. controls ++(90:.3cm) and ++(270:.3cm) .. (1.2,.8);
 \draw (.8,0) arc (180:0:.2cm);
 \node at (.2,1) {\scriptsize{$n$}};
 \node at (.5,1) {\scriptsize{$i$}};
 \node at (1.2,1) {\scriptsize{$j$}};
 \node at (.8,-.2) {\scriptsize{$i$}};
\end{tikzpicture}
\right)
\cdot
\begin{tikzpicture}[baseline=.3cm]
 \draw[thick, rounded corners = 5pt] (0,0) rectangle (1.6,.8);
 \draw (.2,0) -- (.2,.8);
 \draw (0.4,.8) arc (-180:0:.5);
 \draw (0.6,.8) arc (-180:0:.3);
 \draw (0.4,0) arc (180:0:.2);
 \draw (1,0) arc (180:0:.2);
 \node at (.2,-.2) {\scriptsize{$n$}};
 \node at (.4,1) {\scriptsize{$j$}};
 \node at (.6,1) {\scriptsize{$i$}};
 \node at (.4,-.2) {\scriptsize{$i$}};
 \node at (1,-.2) {\scriptsize{$j$}};
\end{tikzpicture}\,.
\qedhere
\end{align*}
\end{proof}

\begin{proof}[Proof of Proposition \ref{prop:InjectiveAlgebraMap}]
By inspection of the definition of $\pi$ from \eqref{eq:4x4Map}, it is clear that $\pi$ is injective, unital, $\bbC$-linear, and respects the $\dag$-structure. 
The difficulty is in seeing that $\pi$ is an algebra homomorphism. 
In the following, we suppress the rightmost $2k$ strings of entries of \eqref{eq:4x4Map}, as well as the factor $d^{-k}$, since they are essentially inert when only three objects are considered. 
Because $\pi$ respects $\dag$, as in Definition \ref{def:MarkovProjections}, we only need to consider 3 cases of composition.
\item[\underline{Case 1:}]
Let $x:[n]\rightarrow [n+2i]$ and $y:[n+2i]\rightarrow [n+2i+2j]$. 
Then 
$$
\pi(y\circ x)
=
d^{-j}E_{n+i+j}^{n+2i+j}
\left(yx \cdot 
 \begin{tikzpicture}[baseline=.3cm]
  \draw (.2,0) -- (.2,.8);
  \node at (.2,1) {\scriptsize{$n$}};
  \draw (.4,.8) arc (-180:0:.2cm);
  \node at (.4,1) {\scriptsize{$i$}};
  \draw (.4,0) .. controls ++(90:.3cm) and ++(270:.3cm) .. (1,.8);
  \node at (1,1) {\scriptsize{$j$}};
  \draw (.6,0) arc (180:0:.2cm);
  \node at (.6,-.2) {\scriptsize{$i$}};
  \draw[thick, rounded corners = 5pt] (0,0) rectangle (1.2,.8);
 \end{tikzpicture}
 \right) 
\cdot 
 \begin{tikzpicture}[baseline=.3cm]
  \draw (.2,0) -- (.2,.8);
  \node at (.2,1) {\scriptsize{$n$}};
  \draw (.4,0) arc (180:0:.2cm);
  \node at (.4,-.2) {\scriptsize{$i$}};
  \draw (1,0) arc (180:0:.2cm);
  \node at (1,-.2) {\scriptsize{$j$}};
  \draw (.4,.8) arc (-180:0:.5cm);
  \node at (.4,1) {\scriptsize{$i$}};
  \draw (.6,.8) arc (-180:0:.3cm);
  \node at (.6,1) {\scriptsize{$j$}};
  \draw[thick, rounded corners = 5pt] (0,0) rectangle (1.6,.8);
 \end{tikzpicture}
\underset{\text{(Lem.~\ref{lem:LeftKink})}}{=} 
d^{-i-j} yx \cdot 
 \begin{tikzpicture}[baseline=.3cm]
  \draw (.2,0) -- (.2,.8);
  \node at (.2,1) {\scriptsize{$n$}};
  \draw (.4,0) arc (180:0:.2cm);
  \node at (.4,-.2) {\scriptsize{$i$}};
  \draw (.4,.8) arc (-180:0:.2cm);
  \node at (.4,1) {\scriptsize{$i$}};
  \draw (1,0) arc (180:0:.2cm);
  \node at (1,-.2) {\scriptsize{$j$}};
  \draw (1,.8) arc (-180:0:.2cm);
  \node at (1,1) {\scriptsize{$j$}};
  \draw[thick, rounded corners = 5pt] (0,0) rectangle (1.6,.8);
 \end{tikzpicture}
 = \pi(y) \cdot \pi(x).
$$

\item[\underline{Case 2:}]
Let $x:[n]\rightarrow [n+2i+2j]$ and $y:[n+2i+2j]\rightarrow [n+2i]$. 
Then 
\begin{align*}
\pi(y\circ x)
&=
d^{-j}
E_{n+i}^{n+2i+j}\left(yx\cdot 
\begin{tikzpicture}[baseline = .3cm]
 \draw[thick, rounded corners = 5pt] (0,0) rectangle (1.5,.8);
 \draw (.2,0) -- (.2,.8);
 \draw (.8,.8) arc (-180:0:.2cm);
 \draw (1.2,0) .. controls ++(90:.3cm) and ++(270:.3cm) .. (.5,.8);
 \draw (.5,0) arc (180:0:.2cm);
 \node at (.2,1) {\scriptsize{$n$}};
 \node at (.5,1) {\scriptsize{$j$}};
 \node at (.8,1) {\scriptsize{$i$}};
 \node at (.5,-.2) {\scriptsize{$i$}};
\end{tikzpicture}
\right)\cdot
 \begin{tikzpicture}[baseline=.3cm]
  \draw (.2,0) -- (.2,.8);
  \node at (.2,1) {\scriptsize{$n$}};
  \draw (.4,0) arc (180:0:.2cm);
  \node at (.4,-.2) {\scriptsize{$i$}};
  \draw (.4,.8) arc (-180:0:.2cm);
  \node at (.4,1) {\scriptsize{$i$}};
  \draw (1,0) arc (180:0:.2cm);
  \node at (1,-.2) {\scriptsize{$j$}};
  \draw (1,.8) arc (-180:0:.2cm);
  \node at (1,1) {\scriptsize{$j$}};
  \draw[thick, rounded corners = 5pt] (0,0) rectangle (1.6,.8);
 \end{tikzpicture}
\\&
=
d^{-i-j}E_{n+i}^{n+2i}\left(
\underbrace{
E_{n+2i}^{n+2i+j}\left(yx\cdot
\begin{tikzpicture}[baseline = .3cm]
 \draw[thick, rounded corners = 5pt] (0,0) rectangle (1.5,.8);
 \draw (.2,0) -- (.2,.8);
 \draw (.8,.8) arc (-180:0:.2cm);
 \draw (1.2,0) .. controls ++(90:.3cm) and ++(270:.3cm) .. (.5,.8);
 \draw (.5,0) arc (180:0:.2cm);
 \node at (.2,1) {\scriptsize{$n$}};
 \node at (.5,1) {\scriptsize{$j$}};
 \node at (.8,1) {\scriptsize{$i$}};
 \node at (.5,-.2) {\scriptsize{$i$}};
\end{tikzpicture}
\right)
}_{z}\cdot
 \begin{tikzpicture}[baseline=.3cm]
  \draw (.2,0) -- (.2,.8);
  \node at (.2,1) {\scriptsize{$n$}};
  \draw (.4,0) arc (180:0:.2cm);
  \node at (.4,-.2) {\scriptsize{$i$}};
  \draw (.4,.8) arc (-180:0:.2cm);
  \node at (.4,1) {\scriptsize{$i$}};
  \draw[thick, rounded corners = 5pt] (0,0) rectangle (1,.8);
 \end{tikzpicture}
\right)\cdot\,
 \begin{tikzpicture}[baseline=.3cm]
  \draw (.2,0) -- (.2,.8);
  \node at (.2,1) {\scriptsize{$n$}};
  \draw (.4,0) arc (180:0:.2cm);
  \node at (.4,-.2) {\scriptsize{$i$}};
  \draw (.4,.8) arc (-180:0:.2cm);
  \node at (.4,1) {\scriptsize{$i$}};
  \draw (1,0) arc (180:0:.2cm);
  \node at (1,-.2) {\scriptsize{$j$}};
  \draw (1,.8) arc (-180:0:.2cm);
  \node at (1,1) {\scriptsize{$j$}};
  \draw[thick, rounded corners = 5pt] (0,0) rectangle (1.6,.8);
 \end{tikzpicture}
\\&
\underset{
\substack{
\text{(Lem.~\ref{lem:LeftKink} for}
\\
\text{ 
$z$ with $j=0$)}
}
}
{=}
d^{-2i-j}
E_{n+2i}^{n+2i+j}\left(yx\cdot
\begin{tikzpicture}[baseline = .3cm]
 \draw[thick, rounded corners = 5pt] (0,0) rectangle (1.5,.8);
 \draw (.2,0) -- (.2,.8);
 \draw (.8,.8) arc (-180:0:.2cm);
 \draw (1.2,0) .. controls ++(90:.3cm) and ++(270:.3cm) .. (.5,.8);
 \draw (.5,0) arc (180:0:.2cm);
 \node at (.2,1) {\scriptsize{$n$}};
 \node at (.5,1) {\scriptsize{$j$}};
 \node at (.8,1) {\scriptsize{$i$}};
 \node at (.5,-.2) {\scriptsize{$i$}};
\end{tikzpicture}
\right)\cdot
 \begin{tikzpicture}[baseline=.3cm]
  \draw (.2,0) -- (.2,.8);
  \node at (.2,1) {\scriptsize{$n$}};
  \draw (.4,0) arc (180:0:.2cm);
  \node at (.4,-.2) {\scriptsize{$i$}};
  \draw (.4,.8) arc (-180:0:.2cm);
  \node at (.4,1) {\scriptsize{$i$}};
  \draw (1,0) arc (180:0:.2cm);
  \node at (1,-.2) {\scriptsize{$j$}};
  \draw (1,.8) arc (-180:0:.2cm);
  \node at (1,1) {\scriptsize{$j$}};
  \draw[thick, rounded corners = 5pt] (0,0) rectangle (1.6,.8);
 \end{tikzpicture}
 \\&
 \underset{\text{(Prop.~\ref{prop:MultistepJonesProjections})}}{=}
 d^{-2i-2j}\,
 \begin{tikzpicture}[baseline=.3cm]
  \draw (.2,0) -- (.2,.8);
  \node at (.2,1) {\scriptsize{$n$}};
  \draw (.4,0) -- (.4,.8);
  \node at (.4,1) {\scriptsize{$i$}};
  \draw (.6,0) -- (.6,.8);
  \node at (.6,1) {\scriptsize{$i$}};
  \draw (.8,0) arc (180:0:.2cm);
  \node at (.8,-.2) {\scriptsize{$j$}};
  \draw (.8,.8) arc (-180:0:.2cm);
  \node at (.8,1) {\scriptsize{$j$}};
  \draw[thick, rounded corners = 5pt] (0,0) rectangle (1.4,.8);
 \end{tikzpicture}
\cdot yx \cdot 
\begin{tikzpicture}[baseline = .3cm]
 \draw[thick, rounded corners = 5pt] (0,0) rectangle (1.7,.8);
 \draw (.2,0) -- (.2,.8);
 \draw (.8,.8) arc (-180:0:.2cm);
 \draw (1.2,0) .. controls ++(90:.3cm) and ++(270:.3cm) .. (.5,.8);
 \draw (.5,0) arc (180:0:.2cm);
 \draw (1.5,0) -- (1.5,.8);
 \node at (.2,1) {\scriptsize{$n$}};
 \node at (.5,1) {\scriptsize{$j$}};
 \node at (.8,1) {\scriptsize{$i$}};
 \node at (.5,-.2) {\scriptsize{$i$}};
 \node at (1.5,1) {\scriptsize{$j$}};
\end{tikzpicture}
\cdot
 \begin{tikzpicture}[baseline=.3cm]
  \draw (.2,0) -- (.2,.8);
  \node at (.2,1) {\scriptsize{$n$}};
  \draw (.4,0) arc (180:0:.2cm);
  \node at (.4,-.2) {\scriptsize{$i$}};
  \draw (.4,.8) arc (-180:0:.2cm);
  \node at (.4,1) {\scriptsize{$i$}};
  \draw (1,0) arc (180:0:.2cm);
  \node at (1,-.2) {\scriptsize{$j$}};
  \draw (1,.8) arc (-180:0:.2cm);
  \node at (1,1) {\scriptsize{$j$}};
  \draw[thick, rounded corners = 5pt] (0,0) rectangle (1.6,.8);
 \end{tikzpicture} 
\\&=
 d^{-i-2j}\,
 \begin{tikzpicture}[baseline=.3cm]
  \draw (.2,0) -- (.2,.8);
  \node at (.2,1) {\scriptsize{$n$}};
  \draw (.4,0) -- (.4,.8);
  \node at (.4,1) {\scriptsize{$i$}};
  \draw (.6,0) -- (.6,.8);
  \node at (.6,1) {\scriptsize{$i$}};
  \draw (.8,0) arc (180:0:.2cm);
  \node at (.8,-.2) {\scriptsize{$j$}};
  \draw (.8,.8) arc (-180:0:.2cm);
  \node at (.8,1) {\scriptsize{$j$}};
  \draw[thick, rounded corners = 5pt] (0,0) rectangle (1.4,.8);
 \end{tikzpicture}
\cdot yx \cdot 
 \begin{tikzpicture}[baseline=.3cm]
  \draw (.2,0) -- (.2,.8);
  \node at (.2,1) {\scriptsize{$n$}};
  \draw (.4,0) arc (180:0:.2cm);
  \node at (.4,-.2) {\scriptsize{$i$}};
  \draw (1,0) arc (180:0:.2cm);
  \node at (1,-.2) {\scriptsize{$j$}};
  \draw (.4,.8) arc (-180:0:.5cm);
  \node at (.4,1) {\scriptsize{$i$}};
  \draw (.6,.8) arc (-180:0:.3cm);
  \node at (.6,1) {\scriptsize{$j$}};
  \draw[thick, rounded corners = 5pt] (0,0) rectangle (1.6,.8);
 \end{tikzpicture}
\\&=
\pi(y)\cdot \pi(x).
\end{align*}

\item[\underline{Case 3:}]
Let $x:[n+2i]\rightarrow [n]$ and $y:[n]\to [n+2i+2j]$.
Then
\begin{align*}
 \pi(y\circ x) 
 &= 
 d^{-j}\left(
 y\cdot\,
d^{-i}\,
\begin{tikzpicture}[baseline = .3cm]
 \draw[thick, rounded corners = 5pt] (0,0) rectangle (1.5,.8);
 \draw (.2,0) -- (.2,.8);
 \draw (.8,.8) arc (-180:0:.2cm);
 \draw (1.2,0) .. controls ++(90:.3cm) and ++(270:.3cm) .. (.5,.8);
 \draw (.5,0) arc (180:0:.2cm);
 \node at (.2,1) {\scriptsize{$n$}};
 \node at (.5,1) {\scriptsize{$j$}};
 \node at (.8,1) {\scriptsize{$i$}};
 \node at (.5,-.2) {\scriptsize{$i$}};
\end{tikzpicture}
 \,\cdot x 
 \right)
 \cdot\,
 \begin{tikzpicture}[baseline=.3cm]
  \draw (.2,0) -- (.2,.8);
  \node at (.2,1) {\scriptsize{$n$}};
  \draw (.4,0) -- (.4,.8);
  \node at (.4,1) {\scriptsize{$i$}};
  \draw (.6,0) -- (.6,.8);
  \node at (.6,1) {\scriptsize{$i$}};
  \draw (.8,0) arc (180:0:.2cm);
  \node at (.8,-.2) {\scriptsize{$j$}};
  \draw (.8,.8) arc (-180:0:.2cm);
  \node at (.8,1) {\scriptsize{$j$}};
  \draw[thick, rounded corners = 5pt] (0,0) rectangle (1.4,.8);
 \end{tikzpicture}
 =
 d^{-i-j}
 y\cdot
\begin{tikzpicture}[baseline = .3cm]
 \draw[thick, rounded corners = 5pt] (0,0) rectangle (1.7,.8);
 \draw (.2,0) -- (.2,.8);
 \draw (.8,.8) arc (-180:0:.2cm);
 \draw (1.2,0) .. controls ++(90:.3cm) and ++(270:.3cm) .. (.5,.8);
 \draw (.5,0) arc (180:0:.2cm);
 \draw (1.5,0) -- (1.5,.8);
 \node at (.2,1) {\scriptsize{$n$}};
 \node at (.5,1) {\scriptsize{$j$}};
 \node at (.8,1) {\scriptsize{$i$}};
 \node at (.5,-.2) {\scriptsize{$i$}};
 \node at (1.5,1) {\scriptsize{$j$}};
\end{tikzpicture}
\,\cdot\,
 \begin{tikzpicture}[baseline=.3cm]
  \draw (.2,0) -- (.2,.8);
  \node at (.2,1) {\scriptsize{$n$}};
  \draw (.4,0) -- (.4,.8);
  \node at (.4,1) {\scriptsize{$i$}};
  \draw (.6,0) -- (.6,.8);
  \node at (.6,1) {\scriptsize{$i$}};
  \draw (.8,0) arc (180:0:.2cm);
  \node at (.8,-.2) {\scriptsize{$j$}};
  \draw (.8,.8) arc (-180:0:.2cm);
  \node at (.8,1) {\scriptsize{$j$}};
  \draw[thick, rounded corners = 5pt] (0,0) rectangle (1.4,.8);
 \end{tikzpicture}
\cdot x 
\\&= 
d^{-i-j}
y\cdot\, 
 \begin{tikzpicture}[baseline=.3cm]
  \draw (.2,0) -- (.2,.8);
  \node at (.2,1) {\scriptsize{$n$}};
  \draw (.4,0) arc (180:0:.2cm);
  \node at (.4,-.2) {\scriptsize{$i$}};
  \draw (1,0) arc (180:0:.2cm);
  \node at (1,-.2) {\scriptsize{$j$}};
  \draw (.4,.8) arc (-180:0:.5cm);
  \node at (.4,1) {\scriptsize{$i$}};
  \draw (.6,.8) arc (-180:0:.3cm);
  \node at (.6,1) {\scriptsize{$j$}};
  \draw[thick, rounded corners = 5pt] (0,0) rectangle (1.6,.8);
 \end{tikzpicture}
\,\cdot x
=
\pi(y)\cdot \pi(x).
\qedhere
\end{align*}
\end{proof}

\begin{cor}
\label{cor:SemisimpleProjectionCategory}
Let $M_\bullet$ be a Markov sequence and let $\cM$ be its unitary Karoubi completion. 
Then $\cM$ is semisimple, and the isomorphism classes of simple objects are in canonical bijection with the vertices of the principal graph.
\end{cor}
\begin{proof}
All endomorphism algebras of $\cM_0$ are finite dimensional $\Cstar$ algebras which are semisimple, and thus $\cM$ is semisimple. 
By \ref{EP:BasicContruction}, every minimal projection in $X_{n+2}$ for $n\geq 0$ is equivalent to a minimal projection in $M_{n}$ via a partial isometry in $\cM_0([n] \to [n+2])$. 
Explicitly, $p\in M_n$ is equivalent to $p e_{n+1} \in M_{n+2}$ via the morphism $p\in M_{n+1} = \cM_0([n] \to [n+2])$, which is a partial isometry using the definition of composition \ref{compose:updown} and \ref{compose:downup} in $\cM_0$, and this exhausts all equivalence classes of minimal projections in $X_{n+2}$.
By recursion, we see that the equivalence classes of minimal projections in $M_\bullet$ are in canonical bijective correspondence with the minimal projections in the $(Y_n)_{n\geq 0}$, which are exactly the vertices of the principal graph.
\end{proof}

\begin{rem}
\label{rem:HowMCategoryChangesUnderMarkovTowerOperations}
By Remark \ref{Rem:ShiftTowerEffects}, the category of projections of a Markov tower is invariant (up to equivalence) under applying shifts. 
By Remark \ref{Rem:CompressionTowerEffects}, the category of projections of the compression $pMp_\bullet $ is the subcategory of the category of projections of $M_\bullet$ generated by minimal projections under $p$. 
In particular, if $p\neq 0$ and the Bratteli diagram of $M_\bullet$ is connected, then the two categories of projections are again equivalent. 
Finally, in the case of the multistep tower, we see by Remark \ref{Rem:MultistepTowerEffects} that the category of projections of $M_{k\bullet}$ is equivalent to the category of projections for $M_\bullet$ when $k$ is odd, and to the subcategory generated by the even part when $k$ is even.
\end{rem}

\subsection{Temperley-Lieb-Jones module categories} 
\label{sec:TLJmodules}

We now show that the category of projections of a connected Markov tower of modulus $d$ can be canonically endowed with the structure of a cyclic pivotal right Temperley-Lieb-Jones ($\cT\cL\cJ(d)$) module $\Cstar$ category.
Moreover, all cyclic pivotal right $\cT\cL\cJ(d)$-module $\Cstar$ categories arise in this way.
Combined with the classification of connected Markov towers of modulus $d$ from Proposition \ref{prop:ClassificationOfMarkovTowers}, we get the following result, which should be compared with \cite{MR3420332} in the non-pivotal setting.

\begin{cor*}[Corollary \ref{cor:TLJPivotalModuleClassification}]
Cyclic pivotal right module $\Cstar$ categories for $\cT\cL\cJ(d)$ are classified by triples $(\Gamma, \dim, v_0)$ where $\Gamma=(V_+, V_- , E)$ is a bipartite graph, $v_0$ is a distinguished vertex, and $\dim : V_+\amalg V_- \to \bbR_{>0}$ is a function satisfying $\dim(v_0) = 1$ and
$$
\sum_{w\sim v} \dim(w) = d \dim(v),
$$
where we write $w\sim v$ to mean $w$ is connected to $v$, and the above sum is taken with multiplicity.
\end{cor*}

\begin{defn}
\label{def:ModuleFromMarkovTower}
 If $M_\bullet$ is a Markov tower with modulus $d$, the corresponding $\cT\cL\cJ(d)$-module is just the category $\cM$ of projections of $M_\bullet$, as in Definition \ref{def:MarkovProjections}.  
The $\TLJ$ action on $\cM$ comes from the fact that $TLJ_\bullet$ is the initial Markov tower for a given parameter. 
By construction, to define an action of $\TLJ$ on $\cM$, it suffices to define the action of the objects $[m,\pm]_{\cT\cL\cJ}$ on $1_{[n]_\cM}$ for every $n$ and $m$, and then define the action functorially on morphisms. 
 
We set $[m]_\cM\vartriangleleft [n]_{\TLJ}:=[m+n]_\cM$. 
For morphisms, we first consider the case where one morphism is the identity. 
 If $g:[a]_{\TLJ}\to[b]_{\TLJ}$, we first add $n$ strings to the left to obtain $1_n \otimes_{\TLJ} g \in \TLJ([a+n]_{\TLJ} \to [b+n]_{\TLJ})$, and we define $1_{M_n}\vartriangleleft g\in \cM([a+n]_\cM \to [b+n]_\cM)$ to be the image of the element $1_n \otimes_{\TLJ} g$ in $M_{n+(a+b)/2}$. 
 The case of $f\vartriangleleft 1$ is more complicated. 
 If $f\in \cM([n]\to [n+2k])=M_{n+k}$, then we define 
\[
f\vartriangleleft 1_{j}
:=
\begin{cases}
f\cdot
 \begin{tikzpicture}[baseline=.3cm]
  \draw (.2,0) -- (.2,.8);
  \node at (.2,1) {\scriptsize{$n$}};
  \draw (.6,0) arc (180:0:.2cm);
  \node at (.6,-.2) {\scriptsize{$k$}};
  \draw (.4,0) .. controls ++(90:.3cm) and ++(270:.3cm) .. (1,.8);
  \node at (1.3,.5) {\scriptsize{$j-k$}};
  \draw (.4,.8) arc (-180:0:.2cm);
  \node at (.4,1) {\scriptsize{$k$}};
  \draw[thick, rounded corners = 5pt] (0,0) rectangle (1.7,.8);
 \end{tikzpicture}
&
\text{if }k\geq j
\\
f\cdot
 \begin{tikzpicture}[baseline=.3cm]
  \draw (.2,0) -- (.2,.8);
  \node at (.2,1) {\scriptsize{$n$}};
  \draw (.4,0) arc (180:0:.2cm);
  \node at (.4,-.2) {\scriptsize{$k$}};
  \draw (1,0) .. controls ++(90:.3cm) and ++(270:.3cm) .. (.4,.8);
  \node at (1.3,.3) {\scriptsize{$k-j$}};
  \draw (.6,.8) arc (-180:0:.2cm);
  \node at (.6,1) {\scriptsize{$k$}};
  \draw[thick, rounded corners = 5pt] (0,0) rectangle (1.7,.8);
 \end{tikzpicture}
&
\text{if }j\leq k.
\end{cases}
\]
Notice these formulas agree when $j=k$.
For $f\in \cM([n]\to[m])$ where $m<n$, we define $f\vartriangleleft 1 := (f^\dag \vartriangleleft 1)^\dag$.
 
 Next, we check that $(f\vartriangleleft1)\circ(1\vartriangleleft g)=(1\vartriangleleft g)\circ(f\vartriangleleft1)$. 
 We illustrate the case $f:[n]_{\cM}\to[n+2k]_{\cM}$ and $g:[m]_{\TLJ}\to [m+2\ell]_{\TLJ}$, where both $m\ge k$ and $\ell\ge k$. 
 In other words, $g\in[k+j]_{\TLJ}\to[k+j+2(k+i)]_{\TLJ}$ for some non-negative $i$ and $j$. The other cases are similar.
 
 Let $\iota$ denote the inclusion in $M_\bullet$. In the following, since $g\in\TLJ$, we represent $g$ by a ticket within Temperley-Lieb tangles, which we may freely move via isotopy.
 By definition, we have 
 \begingroup\allowdisplaybreaks 
 \begin{align*}
  (f\vartriangleleft 1_{2i+j+3k})\circ(1_{n}\vartriangleleft g) 
  &
  \underset{\ref{compose:upup}}{=}
  d^{i+k} 
  E^{n+j+4k+2i}_{n+j+3k+i}\left(\iota(f)\cdot
  \begin{tikzpicture}[baseline=.3cm]
   \begin{scope}[shift={(0,1)}]
    \begin{scope}[shift={(.1,0)}]
     \draw (.1,0) -- (.1,.8);
     \node at (.1,1) {\scriptsize{$n$}};
     \draw (.3,.8) arc (-180:0:.2cm);
     \node at (.3,1) {\scriptsize{$k$}};
     \node at (.7,1) {\scriptsize{$k$}};
     \begin{scope}[shift={(.2,0)}] 
      \draw (.3,0)  .. controls ++(90:.3cm) and ++(270:.3cm) .. (.9,.8);
      \node at (.9,1) {\scriptsize{$j$}};
      \draw (.5,0)  .. controls ++(90:.3cm) and ++(270:.3cm) .. (1.1,.8);
      \node at (1.1,1) {\scriptsize{$k$}};
      \draw (.7,0)  .. controls ++(90:.3cm) and ++(270:.3cm) .. (1.3,.8);
      \node at (1.3,1) {\scriptsize{$i$}};
      \draw (.9,0)  .. controls ++(90:.3cm) and ++(270:.3cm) .. (1.5,.8);
      \node at (1.5,1) {\scriptsize{$k$}};
     \end{scope}
     \begin{scope}[shift={(.4,0)}] 
      \draw (1.1,0)  .. controls ++(90:.3cm) and ++(270:.3cm) .. (1.7,.8);
      \node at (1.7,1) {\scriptsize{$i$}};
      \draw (1.3,0) arc (180:0:.2cm);
      \node at (1.8,.3) {\scriptsize{$k$}};
     \end{scope}
    \end{scope}
    \draw[dashed] (0,0) -- (2.6,0);
   \end{scope}
   \draw (.2,0) -- (.2,1); 
   \draw (.6,0) -- (.6,1); 
   \draw (.8,0) -- (.8,1);
   \draw (1,0) -- (1,1);
   \draw (1.2,0) -- (1.2,1);
   \draw[thick, rounded corners = 5pt, fill=white] (.4,.2) rectangle (1.5,.8);
   \node at (.95,.5) {$g$};
   \draw (1.6,0) -- (1.6,1); 
   \draw (1.8,0) -- (1.8,1);
   \draw (2.2,0) -- (2.2,1);
   \draw[dashed] (0,0) -- (2.6,0);
   \begin{scope}[shift={(0,-1)}]
    \begin{scope}[shift={(.1,0)}]
     \draw (.1,0) -- (.1,1);
     \node at (.1,-0.2) {\scriptsize{$n$}};
     \begin{scope}[shift={(.2,0)}] 
      \draw (.3,0) -- (.3,1);
      \node at (.3,-0.2) {\scriptsize{$j$}};
      \draw (.5,0) -- (.5,1);
      \node at (.5,-0.2) {\scriptsize{$k$}};
      \draw (.7,1) arc (-180:0:.4cm); 
      \draw (.9,1) arc (-180:0:.2cm); 
      \draw (1,0)  .. controls ++(90:.3cm) and ++(270:.3cm) .. (1.9,1);
      \node at (1,-0.2) {\scriptsize{$k$}};
      \draw (1.3,0) arc (180:0:.4cm);
      \node at (1.3,-0.2) {\scriptsize{$i$}};
      \node at (2.1,-0.2) {\scriptsize{$i$}};
      \draw (1.5,0) arc (180:0:.2cm);
      \node at (1.5,-0.2) {\scriptsize{$k$}};
      \node at (1.9,-0.2) {\scriptsize{$k$}};
     \end{scope}
    \end{scope}
   \end{scope}
   \draw[thick, rounded corners = 5pt] (0,-1) rectangle (2.6,1.8);
  \end{tikzpicture}
  \right)
  \\&= 
  d^{i+k}  
  E^{n+j+4k+i}_{n+j+3k+i}\left(\iota(f)\cdot E^{n+j+4k+2i}_{n+j+4k+i}\left(
  \begin{tikzpicture}[baseline=.3cm]
   \begin{scope}[shift={(0,1)}]
    \begin{scope}[shift={(.1,0)}]
     \draw (.1,0) -- (.1,.8);
     \node at (.1,1) {\scriptsize{$n$}};
     \draw (.3,.8) arc (-180:0:.2cm);
     \node at (.3,1) {\scriptsize{$k$}};
     \begin{scope}[shift={(.2,0)}] 
      \draw (.3,0)  .. controls ++(90:.3cm) and ++(270:.3cm) .. (.9,.8);
      \node at (.9,1) {\scriptsize{$j$}};
      \draw (.5,0)  .. controls ++(90:.3cm) and ++(270:.3cm) .. (1.1,.8);
      \node at (1.1,1) {\scriptsize{$k$}};
      \draw (.7,0)  .. controls ++(90:.3cm) and ++(270:.3cm) .. (1.3,.8);
      \node at (1.3,1) {\scriptsize{$i$}};
      \draw (.9,0)  .. controls ++(90:.3cm) and ++(270:.3cm) .. (1.5,.8);
      \node at (1.5,1) {\scriptsize{$k$}};
     \end{scope}
     \begin{scope}[shift={(.4,0)}] 
      \draw (1.1,0)  .. controls ++(90:.3cm) and ++(270:.3cm) .. (1.7,.8);
      \node at (1.7,1) {\scriptsize{$i$}};
      \draw (1.3,0) arc (180:0:.2cm);
      \node at (1.8,.3) {\scriptsize{$k$}};
     \end{scope}
    \end{scope}
    \draw[dashed] (0,0) -- (2.6,0);
   \end{scope}
   \draw (.2,0) -- (.2,1); 
   \draw (.6,0) -- (.6,1); 
   \draw (.8,0) -- (.8,1);
   \draw (1,0) -- (1,1);
   \draw (1.2,0) -- (1.2,1);
   \draw[thick, rounded corners = 5pt, fill=white] (.4,.2) rectangle (1.5,.8);
   \node at (.95,.5) {$g$};
   \draw (1.6,0) -- (1.6,1); 
   \draw (1.8,0) -- (1.8,1);
   \draw (2.2,0) -- (2.2,1);
   \draw[dashed] (0,0) -- (2.6,0);
   \begin{scope}[shift={(0,-1)}]
    \begin{scope}[shift={(.1,0)}]
     \draw (.1,0) -- (.1,1);
     \node at (.1,-0.2) {\scriptsize{$n$}};
     \begin{scope}[shift={(.2,0)}] 
      \draw (.3,0) -- (.3,1);
      \node at (.3,-0.2) {\scriptsize{$j$}};
      \draw (.5,0) -- (.5,1);
      \node at (.5,-0.2) {\scriptsize{$k$}};
      \draw (.7,1) arc (-180:0:.4cm); 
      \draw (.9,1) arc (-180:0:.2cm); 
      \draw (1,0)  .. controls ++(90:.3cm) and ++(270:.3cm) .. (1.9,1);
      \node at (1,-0.2) {\scriptsize{$k$}};
      \draw (1.3,0) arc (180:0:.4cm);
      \node at (1.3,-0.2) {\scriptsize{$i$}};
      \node at (2.1,-0.2) {\scriptsize{$i$}};
      \draw (1.5,0) arc (180:0:.2cm);
      \node at (1.5,-0.2) {\scriptsize{$k$}};
      \node at (1.9,-0.2) {\scriptsize{$k$}};
     \end{scope}
    \end{scope}
   \end{scope}
   \draw[thick, rounded corners = 5pt] (0,-1) rectangle (2.6,1.8);
  \end{tikzpicture}
 \right)\right)
  \\&= 
  d^{i+k}  
  E^{n+j+4k+i}_{n+j+3k+i}\left(\iota(f)\cdot E^{n+j+4k+2i}_{n+j+4k+i}\left(
  \begin{tikzpicture}[baseline=.3cm]
   \begin{scope}[shift={(0,0)}] 
    \begin{scope}[shift={(.1,0)}]
     \draw (.1,0) -- (.1,.8);
     \draw (.3,.8) arc (-180:0:.2cm);
     \begin{scope}[shift={(.2,0)}] 
      \draw (.3,0)  .. controls ++(90:.3cm) and ++(270:.3cm) .. (.9,.8);
      \draw (.5,0)  .. controls ++(90:.3cm) and ++(270:.3cm) .. (1.1,.8);
      \draw (.7,0)  .. controls ++(90:.3cm) and ++(270:.3cm) .. (1.3,.8);
      \draw (.9,0)  .. controls ++(90:.3cm) and ++(270:.3cm) .. (1.5,.8);
     \end{scope}
     \begin{scope}[shift={(.4,0)}] 
      \draw (1.1,0)  .. controls ++(90:.3cm) and ++(270:.3cm) .. (1.7,.8);
      \draw (1.3,0) arc (180:0:.2cm);
     \end{scope}
    \end{scope}
    \draw[dashed] (0,0) -- (2.6,0);
   \end{scope}
   \begin{scope}[shift={(0,.8)}] 
    \draw (.2,0) -- (.2,1);
    \node at (.2,1.2) {\scriptsize{$n$}};
    \draw (.4,0) -- (.4,1); 
    \node at (.4,1.2) {\scriptsize{$k$}};
    \draw (.8,0) -- (.8,1); 
    \node at (.8,1.2) {\scriptsize{$k$}};
    \draw (1.2,0) -- (1.2,1);
    \node at (1.2,1.2) {\scriptsize{$j$}};
    \draw (1.4,0) -- (1.4,1);
    \node at (1.4,1.2) {\scriptsize{$k$}};
    \draw (1.6,0) -- (1.6,1);
    \node at (1.6,1.2) {\scriptsize{$i$}};
    \draw (1.8,0) -- (1.8,1);
    \node at (1.8,1.2) {\scriptsize{$k$}};
    \draw[thick, rounded corners = 5pt, fill=white] (1,.2) rectangle (2,.8);
    \node at (1.5,.5) {$g$};
    \draw (2.2,0) -- (2.2,1);
    \node at (2.2,1.2) {\scriptsize{$i$}};
    \draw[dashed] (0,0) -- (2.6,0);
   \end{scope}
   \begin{scope}[shift={(0,-1)}]
    \begin{scope}[shift={(.1,0)}]
     \draw (.1,0) -- (.1,1);
     \node at (.1,-0.2) {\scriptsize{$n$}};
     \begin{scope}[shift={(.2,0)}] 
      \draw (.3,0) -- (.3,1);
      \node at (.3,-0.2) {\scriptsize{$j$}};
      \draw (.5,0) -- (.5,1);
      \node at (.5,-0.2) {\scriptsize{$k$}};
      \draw (.7,1) arc (-180:0:.4cm); 
      \draw (.9,1) arc (-180:0:.2cm); 
      \draw (1,0)  .. controls ++(90:.3cm) and ++(270:.3cm) .. (1.9,1);
      \node at (1,-0.2) {\scriptsize{$k$}};
      \draw (1.3,0) arc (180:0:.4cm);
      \node at (1.3,-0.2) {\scriptsize{$i$}};
      \node at (2.1,-0.2) {\scriptsize{$i$}};
      \draw (1.5,0) arc (180:0:.2cm);
      \node at (1.5,-0.2) {\scriptsize{$k$}};
      \node at (1.9,-0.2) {\scriptsize{$k$}};
     \end{scope}
    \end{scope}
   \end{scope}
   \draw[thick, rounded corners = 5pt] (0,-1) rectangle (2.6,1.8);
  \end{tikzpicture}
 \right)\right)
 \\&=d^{i+k}  
E^{n+j+4k+i}_{n+j+3k+i}\left(
 \begin{tikzpicture}[baseline=.3cm]
  \draw (.2,0) -- (.2,1);
  \node at (.2,1.2) {\scriptsize{$n$}};
  \draw (.4,0) -- (.4,1);
  \node at (.4,1.2) {\scriptsize{$k$}};
  \draw (.6,0) -- (.6,1);
  \node at (.6,1.2) {\scriptsize{$k$}};
  \draw (1,0) -- (1,1);
  \node at (1,1.2) {\scriptsize{$j$}};
  \draw (1.2,0) -- (1.2,1);
  \node at (1.2,1.2) {\scriptsize{$k$}};
  \draw (1.4,0) -- (1.4,1);
  \node at (1.4,1.2) {\scriptsize{$i$}};
  \draw (1.6,0) -- (1.6,1);
  \node at (1.6,1.2) {\scriptsize{$k$}};
  \draw[thick, rounded corners = 5pt, fill=white] (.8,.2) rectangle (1.8,.8);
  \node at (1.3,.5) {$g$};
  \draw[thick,rounded corners = 5pt] (0,0) rectangle (2,1);
 \end{tikzpicture}
 \cdot\iota(f)\cdot E^{n+j+4k+2i}_{n+j+4k+i}\left(
  \begin{tikzpicture}[baseline=-0.2cm]
   \begin{scope}[shift={(.1,0)}]
    \draw (.1,0) -- (.1,.8);
    \draw (.3,.8) arc (-180:0:.2cm);
    \draw (.3,0)  .. controls ++(90:.3cm) and ++(270:.3cm) .. (.9,.8);
    \draw (.5,0)  .. controls ++(90:.3cm) and ++(270:.3cm) .. (1.1,.8);
    \draw (.7,0)  .. controls ++(90:.3cm) and ++(270:.3cm) .. (1.3,.8);
    \draw (.9,0)  .. controls ++(90:.3cm) and ++(270:.3cm) .. (1.5,.8);
    \node at (.1,1) {\scriptsize{$n$}};
    \node at (.3,1) {\scriptsize{$k$}}; 
    \node at (.7,1) {\scriptsize{$k$}}; 
    \node at (.9,1) {\scriptsize{$j$}};
    \node at (1.1,1) {\scriptsize{$k$}};
    \node at (1.3,1) {\scriptsize{$k$}};
    \node at (1.5,1) {\scriptsize{$i$}};
    \draw (1.3,0)  .. controls ++(90:.3cm) and ++(270:.3cm) .. (1.9,.8);
    \node at (1.9,1) {\scriptsize{$i$}};
    \draw (1.5,0) arc (180:0:.2cm);
    \node at (1.9,.3) {\scriptsize{$k$}};
   \end{scope}
   \draw[dashed] (0,0) -- (2.4,0);
   \begin{scope}[shift={(.1,-1)}]
    \draw (.1,0) -- (.1,1);
    \node at (.1,-.2) {\scriptsize{$n$}};
    \draw (.3,0) -- (.3,1);
    \node at (.3,-.2) {\scriptsize{$j$}};
    \draw (.5,0) -- (.5,1);
    \node at (.5,-.2) {\scriptsize{$k$}};
    \draw (.7,1) arc (-180:0:.4cm); 
    \draw (.9,1) arc (-180:0:.2cm); 
    \draw (1.1,0)  .. controls ++(90:.3cm) and ++(270:.3cm) .. (1.9,1);
    \node at (1.1,-0.2) {\scriptsize{$k$}};
    \draw (1.3,0) arc (180:0:.4cm);
    \node at (1.3,-0.2) {\scriptsize{$i$}};
    \node at (2.1,-0.2) {\scriptsize{$i$}};
    \draw (1.5,0) arc (180:0:.2cm);
    \node at (1.5,-0.2) {\scriptsize{$k$}};
    \node at (1.9,-0.2) {\scriptsize{$k$}};
   \end{scope}
   \draw[thick, rounded corners = 5pt] (0,-1) rectangle (2.4,0.8);
  \end{tikzpicture}
  \right)\right)
  \\&= 
d^{i+k}  
E^{n+j+4k+i}_{n+j+3k+i}\left(
  \begin{tikzpicture}[baseline=.3cm]
   \draw (.2,0) -- (.2,1);
   \node at (.2,1.2) {\scriptsize{$n$}};
   \draw (.4,0) -- (.4,1);
   \node at (.4,1.2) {\scriptsize{$k$}};
   \draw (.6,0) -- (.6,1);
   \node at (.6,1.2) {\scriptsize{$k$}};
   \draw (1,0) -- (1,1);
   \node at (1,1.2) {\scriptsize{$j$}};
   \draw (1.2,0) -- (1.2,1);
   \node at (1.2,1.2) {\scriptsize{$k$}};
   \draw (1.4,0) -- (1.4,1);
   \node at (1.4,1.2) {\scriptsize{$i$}};
   \draw (1.6,0) -- (1.6,1);
   \node at (1.6,1.2) {\scriptsize{$k$}};
   \draw[thick, rounded corners = 5pt, fill=white] (.8,.2) rectangle (1.8,.8);
   \node at (1.3,.5) {$g$};
   \draw[thick,rounded corners = 5pt] (0,0) rectangle (2,1);
  \end{tikzpicture}
  \cdot\iota(f)\cdot
d^{-i}
  \begin{tikzpicture}[baseline=-0.1cm]
   \begin{scope}[shift={(.1,0)}]
    \draw (.1,0) -- (.1,.8);
    \node at (.1,1) {\scriptsize{$n$}};
    \draw (.3,.8) arc (-180:0:.2cm);
    \node at (.3,1) {\scriptsize{$k$}};
    \node at (.7,1) {\scriptsize{$k$}};
    \draw (.3,0) .. controls ++(90:.3cm) and ++(270:.3cm) .. (.9,.8);
    \node at (.9,1) {\scriptsize{$j$}};
    \draw (.5,0) arc (180:0:.2cm); 
    \draw (1.3,0) -- (1.3,.8);
    \node at (1.3,1) {\scriptsize{$k$}};
    \draw (1.5,0) -- (1.5,.8);
    \node at (1.5,1) {\scriptsize{$i$}};
    \draw (1.7,0) -- (1.7,.8);
    \node at (1.7,1) {\scriptsize{$k$}};
   \end{scope}
   \draw[dashed] (0,0) -- (2,0);
   \begin{scope}[shift={(.1,-0.8)}]
    \draw (.1,0) -- (.1,.8);
    \node at (.1,-0.2) {\scriptsize{$n$}};
    \draw (.3,0) -- (.3,.8);
    \node at (.3,-0.2) {\scriptsize{$j$}};
    \draw (.5,0) -- (.5,.8);
    \node at (.5,-0.2) {\scriptsize{$k$}};
    \draw (.9,.8) arc (-180:0:.2cm); 
    \draw (.9,0) .. controls ++(90:.3cm) and ++(270:.3cm) .. (1.5,.8);
    \node at (0.9,-0.2) {\scriptsize{$i$}};
    \draw (1.1,0) .. controls ++(90:.3cm) and ++(270:.3cm) .. (1.7,.8);
    \node at (1.1,-0.2) {\scriptsize{$k$}};
    \draw (1.3,0) arc (180:0:.2cm);
    \node at (1.3,-.2) {\scriptsize{$k$}};
    \node at (1.7,-.2) {\scriptsize{$k$}};
   \end{scope}
   \draw[thick, rounded corners = 5pt] (0,-0.8) rectangle (2,0.8);
  \end{tikzpicture}
\right)\\
&\underset{\ref{compose:upup}}{=}
(1\vartriangleleft g)\circ(f\vartriangleleft 1).
 \end{align*}
 \endgroup
 To show that $\vartriangleleft$ is a well-defined bifunctor, it remains to show that for morphisms $f$ and $g$ in $\cM$ and $h$ and $k$ in $\TLJ$, we have $(g\vartriangleleft 1)\circ(f\vartriangleleft 1)=(g\circ f)\vartriangleleft 1$ and $(1\vartriangleleft k)\circ(1\vartriangleleft h)=1\vartriangleleft(k\circ h)$. 
 Functoriality in the right variable comes directly from the definition, but functoriality in the left variable is more involved and done in cases.
 We illustrate a representative case. 
 Suppose $f:[n]_{\cM}\to[n+2i+2j]_{\cM}$ and $g:[n+2i+2j]_{\cM}\to[n+2i]_{\cM}$, and set $m:=k+i+j$. 
 Then
 \begin{align*}
  (g\vartriangleleft 1_m)\circ (f\vartriangleleft 1_m) 
  &= 
  d^{i+j}
  E^{n+3i+k+2j}_{n+2i+k+j}\left(
  \begin{tikzpicture}[baseline=.3cm]
   \draw (.2,0) -- (.2,.8);
   \node at (.2,1) {\scriptsize{$n$}};
   \draw (.4,0) -- (.4,.8);
   \node at (.4,1) {\scriptsize{$i$}};
   \draw (.6,0) -- (.6,.8);
   \node at (.6,1) {\scriptsize{$i$}};
   \draw (.8,0) arc (180:0:.2cm);
   \node at (.8,-.2) {\scriptsize{$j$}};
   \node at (1.2,-.2) {\scriptsize{$j$}};
   \draw (.8,.8) .. controls ++(-90:.3cm) and ++(-270:.3cm) .. (1.4,0);
   \node at (.8,1) {\scriptsize{$k$}};
   \draw (1,.8) .. controls ++(-90:.3cm) and ++(-270:.3cm) .. (1.6,0);
   \node at (1,1) {\scriptsize{$i$}};
   \draw (1.2,.8) arc (-180:0:.2cm);
   \node at (1.2,1) {\scriptsize{$j$}};
   \node at (1.6,1) {\scriptsize{$j$}};
   \draw[thick, rounded corners = 5pt] (0,0) rectangle (1.8,.8);
  \end{tikzpicture}
  \cdot\iota(g)\cdot\iota(f)\cdot
  \begin{tikzpicture}[baseline=.3cm]
   \begin{scope}[shift={(0,.4)}]
    \draw (.2,0) -- (.2,1);
    \node at (.2,1.2) {\scriptsize{$n$}};
    \draw (.4,1) arc (-180:0:.4cm);
    \node at (.4,1.2) {\scriptsize{$j$}};
    \node at (1.2,1.2) {\scriptsize{$j$}};
    \draw (.6,1) arc (-180:0:.2cm);
    \node at (.6,1.2) {\scriptsize{$i$}};
    \node at (1,1.2) {\scriptsize{$i$}};
    \draw (.4,0) .. controls ++(90:.3cm) and ++(270:.3cm) .. (1.5,1);
    \node at (1.5,1.2) {\scriptsize{$k$}};
    \draw (.8,0) arc (180:0:.4cm); 
    \draw (1,0) arc (180:0:.2cm); 
    \draw (2,0) -- (2,1);
    \node at (2,1.2) {\scriptsize{$i$}};
   \end{scope}
   \draw[dashed] (0,.4) -- (2.2,.4);
   \begin{scope}[shift={(0,-.4)}]
    \draw (.2,0) -- (.2,.8);
    \node at (.2,-.2) {\scriptsize{$n$}};
    \draw (.4,0) -- (.4,.8);
    \node at (.4,-.2) {\scriptsize{$k$}};
    \begin{scope}[shift={(.2,0)}]
     \draw (.6,0) -- (.6,.8);
     \node at (.6,-.2) {\scriptsize{$i$}};
     \draw (.8,0) -- (.8,.8);
     \node at (.8,-.2) {\scriptsize{$j$}};
     \draw (1,0) arc (180:0:.2cm);
     \node at (1,-.2) {\scriptsize{$i$}};
     \node at (1.4,-.2) {\scriptsize{$i$}};
     \draw (1.2,.8) .. controls ++(-90:.3cm) and ++(-270:.3cm) .. (1.6,0);
     \node at (1.6,-.2) {\scriptsize{$i$}};
     \draw (1.4,.8) arc (-180:0:.2cm); 
    \end{scope}
   \end{scope}
   \draw[thick, rounded corners = 5pt] (0,-.4) rectangle (2.2,1.4);
  \end{tikzpicture}
  \right)
  \\&= 
  d^{-i}
  E^{n+3i+k+2j}_{n+2i+k+j}\left(
  \begin{tikzpicture}[baseline=.3cm]
   \begin{scope}[shift={(0,.4)}]
    \draw (.2,0) -- (.2,.8);
    \node at (.2,1) {\scriptsize{$n$}};
    \draw (.4,0) -- (.4,.8);
    \node at (.4,1) {\scriptsize{$i$}};
    \draw (.6,0) -- (.6,.8);
    \node at (.6,1) {\scriptsize{$i$}};
    \draw (.8,0) arc (180:0:.2cm);
    \node at (.9,.35) {\scriptsize{$j$}}; 
    \draw (1,.8) .. controls ++(-90:.3cm) and ++(-270:.3cm) .. (1.4,0);
    \node at (1,1) {\scriptsize{$k$}};
    \draw (1.2,.8) .. controls ++(-90:.3cm) and ++(-270:.3cm) .. (1.6,0);
    \node at (1.2,1) {\scriptsize{$i$}};
    \draw (1.4,.8) arc (-180:0:.2cm);
    \node at (1.4,1) {\scriptsize{$j$}};
    \node at (1.8,1) {\scriptsize{$j$}};
   \draw[dashed] (0,0) -- (2,0);
   \end{scope}
   \begin{scope}[shift={(0,-.4)}]
    \draw (.2,0) -- (.2,.8);
    \node at (.2,-.2) {\scriptsize{$n$}};
    \draw (.4,0) -- (.4,.8);
    \node at (.4,-.2) {\scriptsize{$i$}};
    \draw (.6,0) -- (.6,.8);
    \node at (.6,-.2) {\scriptsize{$i$}};
    \draw (.8,.8) arc (-180:0:.2cm); 
    \draw (.8,0) arc (180:0:.2cm);
    \node at (.8,-.2) {\scriptsize{$j$}};
    \node at (1.2,-.2) {\scriptsize{$j$}};
    \draw (1.4,0) -- (1.4,.8);
    \node at (1.4,-.2) {\scriptsize{$k$}};
    \draw (1.6,0) -- (1.6,.8);
    \node at (1.6,-.2) {\scriptsize{$i$}};
   \end{scope}
   \draw[thick, rounded corners = 5pt] (0,-.4) rectangle (2,1.2);
  \end{tikzpicture}
  \cdot\iota(g)\cdot\iota(f)\cdot
  \begin{tikzpicture}[baseline=.3cm]
   \begin{scope}[shift={(0,.8)}]
    \draw (.2,0) -- (.2,.8);
    \node at (.2,1) {\scriptsize{$n$}};
    \draw (.6,.8) .. controls ++(-90:.3cm) and ++ (-270:.3cm) .. (1,0);
    \node at (.6,1) {\scriptsize{$j$}};
    \draw (.8,.8) arc (-180:0:.2cm);
    \node at (.8,1) {\scriptsize{$i$}};
    \node at (1.2,1) {\scriptsize{$i$}};
    \draw (.4,0) arc (180:0:.2cm); 
    \draw (1.4,0) -- (1.4,.8);
    \node at (1.4,1) {\scriptsize{$j$}};
    \draw (1.6,0) -- (1.6,.8);
    \node at (1.6,1) {\scriptsize{$k$}};
    \draw (1.8,0) -- (1.8,.8);
    \node at (1.8,1) {\scriptsize{$i$}};
    \draw[dashed] (0,0) -- (2,0);
   \end{scope}
   \begin{scope}[shift={(0,0)}]
    \draw (.2,0) -- (.2,.8); 
    \draw (.4,0) -- (.4,.8); 
    \draw (.8,0) -- (.8,.8); 
    \node at (.6,.4) {\scriptsize{$i$}}; 
    \draw (1,.8) arc (-180:0:.2cm); 
    \draw (1,0) arc(180:0:.2cm);
    \node at (1.2,.4) {\scriptsize{$j$}}; 
    \draw (1.6,0) -- (1.6,.8); 
    \draw (1.8,0) -- (1.8,.8); 
    \draw[dashed] (0,0) -- (2,0);
   \end{scope}
   \begin{scope}[shift={(0,-.8)}]
    \draw (.2,0) -- (.2,.8);
    \node at (.2,-.2) {\scriptsize{$n$}};
    \draw (.4,.8) arc (-180:0:.2cm); 
    \draw (1,.8) arc (-180:0:.2cm); 
    \draw (1.6,.8) .. controls ++(-90:.3cm) and ++(-270:.3cm) .. (.9,0);
    \node at (.9,-.2) {\scriptsize{$k$}};
    \draw (1.8,.8) .. controls ++(-90:.3cm) and ++(-270:.3cm) .. (1.1,0);
    \node at (1.1,-.2) {\scriptsize{$i$}};
    \draw (1.3,0) arc (180:0:.2cm);
    \node at (1.3,-.2) {\scriptsize{$j$}};
    \node at (1.7,-.2) {\scriptsize{$j$}};
   \end{scope}
   \draw[thick, rounded corners=5pt] (0,-.8) rectangle (2,1.6);
  \end{tikzpicture}
  \right)
  \\&= 
  d^{-i+j}
  E^{n+3i+k+2j}_{n+2i+k+j}\left(
  \begin{tikzpicture}[baseline=.3cm]
   \draw (.2,0) -- (.2,.8);
   \node at (.2,1) {\scriptsize{$n$}};
   \draw (.4,0) -- (.4,.8);
   \node at (.4,1) {\scriptsize{$i$}};
   \draw (.6,0) -- (.6,.8);
   \node at (.6,1) {\scriptsize{$i$}};
   \begin{scope}[shift={(.4,0)}] 
   \draw (.4,0) arc (180:0:.2cm);
    \node at (.4,-.2) {\scriptsize{$j$}};
    \node at (.8,-.2) {\scriptsize{$j$}};
    \draw (.6,.8) .. controls ++(-90:.3cm) and ++(-270:.3cm) .. (1,0);
    \node at (.6,1) {\scriptsize{$k$}};
    \draw (.8,.8) .. controls ++(-90:.3cm) and ++(-270:.3cm) .. (1.2,0);
    \node at (.8,1) {\scriptsize{$i$}};
    \draw (1,.8) arc (-180:0:.2cm);
    \node at (1,1) {\scriptsize{$j$}};
    \node at (1.4,1) {\scriptsize{$j$}};
   \end{scope}
   \draw[thick, rounded corners=5pt] (0,0) rectangle (2,.8);
  \end{tikzpicture}
  \cdot\iota\left(E^{n+2i+j}_{n+2i}\left(g\cdot\iota(f)\cdot 
  \begin{tikzpicture}[baseline=.3cm]
   \draw (.2,0) -- (.2,.8);
   \node at (.2,1) {\scriptsize{$n$}};
   \draw (.6,.8) .. controls ++(-90:.3cm) and ++ (-270:.3cm) .. (1,0);
   \node at (.6,1) {\scriptsize{$j$}};
   \draw (.8,.8) arc (-180:0:.2cm);
   \node at (.8,1) {\scriptsize{$i$}};
   \node at (1.2,1) {\scriptsize{$i$}};
   \draw (.4,0) arc (180:0:.2cm);
   \node at (.4,-.2) {\scriptsize{$i$}};
   \node at (.8,-.2) {\scriptsize{$i$}};
   \draw[thick, rounded corners=5pt] (0,0) rectangle (1.4,.8);
  \end{tikzpicture}
  \right)\right)
  \cdot
  \begin{tikzpicture}[baseline=.3cm]
   \begin{scope}[shift={(0,.4)}]
    \draw (.2,0) -- (.2,.8); 
    \draw (.4,0) -- (.4,.8); 
    \draw (.8,0) -- (.8,.8); 
    \node at (.6,.4) {\scriptsize{$i$}}; 
    \draw (1,.8) arc (-180:0:.2cm); 
    \draw (1,0) arc(180:0:.2cm);
    \node at (1.2,.4) {\scriptsize{$j$}}; 
    \draw (1.6,0) -- (1.6,.8); 
    \draw (1.8,0) -- (1.8,.8); 
    \draw[dashed] (0,0) -- (2.3,0);
   \end{scope}
   \begin{scope}[shift={(0,-.4)}]
    \draw (.2,0) -- (.2,1);
    \node at (.2,-.2) {\scriptsize{$n$}};
    \draw (.4,.8) arc (-180:0:.2cm); 
    \draw (1,.8) arc (-180:0:.2cm); 
    \draw (1.6,.8) .. controls ++(-90:.3cm) and ++(-270:.3cm) .. (.8,0);
    \node at (.8,-.2) {\scriptsize{$k$}};
    \draw (1.8,.8) .. controls ++(-90:.3cm) and ++(-270:.3cm) .. (1,0);
    \node at (1,-.2) {\scriptsize{$i$}};
    \draw (1.3,0) arc (180:0:.4cm);
    \node at (1.3,-.2) {\scriptsize{$j$}};
    \node at (2.1,-.2) {\scriptsize{$j$}};
    \draw (1.5,0) arc (180:0:.2cm);
    \node at (1.5,-.2) {\scriptsize{$i$}};
    \node at (1.9,-.2) {\scriptsize{$i$}};
   \end{scope}
   \draw[thick, rounded corners=5pt] (0,-.4) rectangle (2.3,1.2);
  \end{tikzpicture}
  \right)
  \\&=
   d^{-i+j}
  \iota\left(E^{n+2i+j}_{n+2i}\left(g\cdot\iota(f)\cdot 
  \begin{tikzpicture}[baseline=.3cm]
   \draw (.2,0) -- (.2,.8);
   \node at (.2,1) {\scriptsize{$n$}};
   \draw (.6,.8) .. controls ++(-90:.3cm) and ++ (-270:.3cm) .. (1,0);
   \node at (.6,1) {\scriptsize{$j$}};
   \draw (.8,.8) arc (-180:0:.2cm);
   \node at (.8,1) {\scriptsize{$i$}};
   \node at (1.2,1) {\scriptsize{$i$}};
   \draw (.4,0) arc (180:0:.2cm);
   \node at (.4,-.2) {\scriptsize{$i$}};
   \node at (.8,-.2) {\scriptsize{$i$}};
   \draw[thick, rounded corners=5pt] (0,0) rectangle (1.4,.8);
  \end{tikzpicture}
  \right)\right)\cdot E^{n+3i+k+2j}_{n+2i+k+j}\left(
  \begin{tikzpicture}[baseline=.3cm]
   \begin{scope}[shift={(0,.8)}]
    \draw (.2,0) -- (.2,.8);
    \node at (.2,1) {\scriptsize{$n$}};
    \draw (.4,0) -- (.4,.8);
    \node at (.4,1) {\scriptsize{$i$}};
    \draw (.8,0) -- (.8,.8);
    \node at (.8,1) {\scriptsize{$i$}};
    \begin{scope}[shift={(.6,0)}] 
     \draw (.4,0) arc (180:0:.2cm); 
     \draw (.6,.8) .. controls ++(-90:.3cm) and ++(-270:.3cm) .. (1,0);
     \node at (.6,1) {\scriptsize{$k$}};
     \draw (.8,.8) .. controls ++(-90:.3cm) and ++(-270:.3cm) .. (1.2,0);
     \node at (.8,1) {\scriptsize{$i$}};
     \draw (1,.8) arc (-180:0:.2cm);
     \node at (1,1) {\scriptsize{$j$}};
     \node at (1.4,1) {\scriptsize{$j$}};
    \end{scope}
    \draw[dashed] (0,0) -- (2.4,0);
   \end{scope}
   \begin{scope}[shift={(0,0)}]
    \draw (.2,0) -- (.2,.8); 
    \draw (.4,0) -- (.4,.8); 
    \draw (.8,0) -- (.8,.8); 
    \node at (.6,.4) {\scriptsize{$i$}}; 
    \draw (1,.8) arc (-180:0:.2cm); 
    \draw (1,0) arc(180:0:.2cm);
    \node at (1.2,.4) {\scriptsize{$j$}}; 
    \draw (1.6,0) -- (1.6,.8); 
    \draw (1.8,0) -- (1.8,.8); 
    \draw[dashed] (0,0) -- (2.4,0);
   \end{scope}
   \begin{scope}[shift={(0,-.8)}]
    \draw (.2,0) -- (.2,.8);
    \node at (.2,-.2) {\scriptsize{$n$}};
    \draw (.4,.8) arc (-180:0:.2cm);
    \node at (.6,.5) {\scriptsize{$i$}}; 
    \draw (1,.8) arc (-180:0:.2cm); 
    \draw (1.6,.8) .. controls ++(-90:.3cm) and ++(-270:.3cm) .. (.9,0);
    \node at (.9,-.2) {\scriptsize{$k$}};
    \draw (1.8,.8) .. controls ++(-90:.3cm) and ++(-270:.3cm) .. (1.1,0);
    \node at (1.1,-.2) {\scriptsize{$i$}};
    \begin{scope}[shift={(.3,0)}] 
     \draw (1.1,0) arc (180:0:.4cm);
     \node at (1.1,-.2) {\scriptsize{$j$}};
     \node at (1.9,-.2) {\scriptsize{$j$}};
     \draw (1.3,0) arc (180:0:.2cm);
     \node at (1.3,-.2) {\scriptsize{$i$}};
     \node at (1.7,-.2) {\scriptsize{$i$}};
    \end{scope}
   \end{scope}
   \draw[thick, rounded corners=5pt] (0,-.8) rectangle (2.4,1.6);
  \end{tikzpicture}
  \right)
  \\&= 
   d^{-2i}
  \iota\left(E^{n+2i+j}_{n+2i}\left(g\cdot\iota(f)\cdot 
  \begin{tikzpicture}[baseline=.3cm]
   \draw (.2,0) -- (.2,.8);
   \node at (.2,1) {\scriptsize{$n$}};
   \draw (.6,.8) .. controls ++(-90:.3cm) and ++ (-270:.3cm) .. (1,0);
   \node at (.6,1) {\scriptsize{$j$}};
   \draw (.8,.8) arc (-180:0:.2cm);
   \node at (.8,1) {\scriptsize{$i$}};
   \node at (1.2,1) {\scriptsize{$i$}};
   \draw (.4,0) arc (180:0:.2cm);
   \node at (.4,-.2) {\scriptsize{$i$}};
   \node at (.8,-.2) {\scriptsize{$i$}};
   \draw[thick, rounded corners=5pt] (0,0) rectangle (1.4,.8);
  \end{tikzpicture}
  \right)\right)\cdot
  \begin{tikzpicture}[baseline=.3cm]
   \draw (.2,0) -- (.2,.8);
   \node at (.2,1) {\scriptsize{$n$}};
   \draw (.4,.8) arc (-180:0:.2cm);
   \node at (.4,1) {\scriptsize{$i$}};
   \node at (.8,1) {\scriptsize{$i$}};
   \draw (.6,0) .. controls ++(90:.3cm) and ++(270:.3cm) .. (1,.8);
   \node at (1,1) {\scriptsize{$k$}};
   \draw (.8,0) .. controls ++(90:.3cm) and ++(270:.3cm) .. (1.2,.8);
   \node at (1.2,1) {\scriptsize{$j$}};
   \draw (1,0) arc (180:0:.2cm);
   \node at (1,-.2) {\scriptsize{$i$}};
   \node at (1.4,-.2) {\scriptsize{$i$}};
   \draw[thick, rounded corners=5pt] (0,0) rectangle (1.6,.8);
  \end{tikzpicture}
  \\&= 
  d^{-3i}
  \iota\left(
  \underbrace{E^{n+2i+j}_{n+2i}\left(g\cdot\iota(f)\cdot 
  \begin{tikzpicture}[baseline=.3cm]
   \draw (.2,0) -- (.2,.8);
   \node at (.2,1) {\scriptsize{$n$}};
   \draw (.6,.8) .. controls ++(-90:.3cm) and ++ (-270:.3cm) .. (1,0);
   \node at (.6,1) {\scriptsize{$j$}};
   \draw (.8,.8) arc (-180:0:.2cm);
   \node at (.8,1) {\scriptsize{$i$}};
   \node at (1.2,1) {\scriptsize{$i$}};
   \draw (.4,0) arc (180:0:.2cm);
   \node at (.4,-.2) {\scriptsize{$i$}};
   \node at (.8,-.2) {\scriptsize{$i$}};
   \draw[thick, rounded corners=5pt] (0,0) rectangle (1.4,.8);
  \end{tikzpicture}
  \right)}_{z}
  \right)\cdot
  \begin{tikzpicture}[baseline=.3cm]
   \begin{scope}[shift={(0,.4)}]
    \draw (.2,0) -- (.2,.8);
    \node at (.2,1) {\scriptsize{$n$}};
    \draw (.4,0) arc (180:0:.2cm); 
    \draw (.4,.8) arc (-180:0:.2cm);
    \node at (.4,1) {\scriptsize{$i$}};
    \node at (.8,1) {\scriptsize{$i$}};
    \draw (1,0) -- (1,.8);
    \node at (1,1) {\scriptsize{$k$}};
    \draw (1.2,0) -- (1.2,.8);
    \node at (1.2,1) {\scriptsize{$j$}};
    \node at (.6,.35) {\scriptsize{$i$}};
    \draw[dashed] (0,0) -- (1.6,0);
   \end{scope}
   \begin{scope}[shift={(0,-.4)}]
    \draw (.2,0) -- (.2,.8);
    \node at (.2,-.2) {\scriptsize{$n$}};
    \draw (.4,.8) arc (-180:0:.2cm); 
    \draw (.6,0) .. controls ++(90:.3cm) and ++(270:.3cm) .. (1,.8);
    \node at (.6,-.2) {\scriptsize{$k$}};
    \draw (.8,0) .. controls ++(90:.3cm) and ++(270:.3cm) .. (1.2,.8);
    \node at (.8,-.2) {\scriptsize{$j$}};
    \draw (1,0) arc (180:0:.2cm);
    \node at (1,-.2) {\scriptsize{$i$}};
    \node at (1.4,-.2) {\scriptsize{$i$}};
   \end{scope}
   \draw[thick, rounded corners=5pt] (0,-.4) rectangle (1.6,1.2);
  \end{tikzpicture}
  \\& 
  \underset{
  \substack{
  \text{(Lem.~\ref{lem:LeftKink} for}
  \\
  \text{ 
  $z$ with $j=0$)}
  }
  }
  {=}
  d^{-2i}
  \iota\left(E^{n+2i}_{n+i}\left(E^{n+2i+j}_{n+2i}\left(g\cdot\iota(f)\cdot
  \begin{tikzpicture}[baseline=.3cm]
   \draw (.2,0) -- (.2,.8);
   \node at (.2,1) {\scriptsize{$n$}};
   \draw (.6,.8) .. controls ++(-90:.3cm) and ++ (-270:.3cm) .. (1,0);
   \node at (.6,1) {\scriptsize{$j$}};
   \draw (.8,.8) arc (-180:0:.2cm);
   \node at (.8,1) {\scriptsize{$i$}};
   \node at (1.2,1) {\scriptsize{$i$}};
   \draw (.4,0) arc (180:0:.2cm);
   \node at (.4,-.2) {\scriptsize{$i$}};
   \node at (.8,-.2) {\scriptsize{$i$}};
   \draw[thick, rounded corners=5pt] (0,0) rectangle (1.4,.8);
  \end{tikzpicture}
  \right)\cdot
  \begin{tikzpicture}[baseline=.3cm]
   \draw (.2,0) -- (.2,.8);
   \node at (.2,1) {\scriptsize{$n$}};
   \draw (.4,0) arc (180:0:.2cm);
   \node at (.4,-.2) {\scriptsize{$i$}};
   \node at (.8,-.2) {\scriptsize{$i$}};
   \draw (.4,.8) arc (-180:0:.2cm);
   \node at (.4,1) {\scriptsize{$i$}};
   \node at (.8,1) {\scriptsize{$i$}};
   \draw[thick, rounded corners=5pt] (0,0) rectangle (1,.8);
  \end{tikzpicture}
  \right)\right)\cdot
  \begin{tikzpicture}[baseline=.3cm]
   \begin{scope}[shift={(0,.4)}]
    \draw (.2,0) -- (.2,.8);
    \node at (.2,1) {\scriptsize{$n$}};
    \draw (.4,0) arc (180:0:.2cm); 
    \draw (.4,.8) arc (-180:0:.2cm);
    \node at (.4,1) {\scriptsize{$i$}};
    \node at (.8,1) {\scriptsize{$i$}};
    \draw (1,0) -- (1,.8);
    \node at (1,1) {\scriptsize{$k$}};
    \draw (1.2,0) -- (1.2,.8);
    \node at (1.2,1) {\scriptsize{$j$}};
    \node at (.6,.35) {\scriptsize{$i$}};
    \draw[dashed] (0,0) -- (1.6,0);
   \end{scope}
   \begin{scope}[shift={(0,-.4)}]
    \draw (.2,0) -- (.2,.8);
    \node at (.2,-.2) {\scriptsize{$n$}};
    \draw (.4,.8) arc (-180:0:.2cm); 
    \draw (.6,0) .. controls ++(90:.3cm) and ++(270:.3cm) .. (1,.8);
    \node at (.6,-.2) {\scriptsize{$k$}};
    \draw (.8,0) .. controls ++(90:.3cm) and ++(270:.3cm) .. (1.2,.8);
    \node at (.8,-.2) {\scriptsize{$j$}};
    \draw (1,0) arc (180:0:.2cm);
    \node at (1,-.2) {\scriptsize{$i$}};
    \node at (1.4,-.2) {\scriptsize{$i$}};
   \end{scope}
   \draw[thick, rounded corners=5pt] (0,-.4) rectangle (1.6,1.2);
  \end{tikzpicture}
  \\&= 
  d^{-i}
  \iota\left(E^{n+2i+j}_{n+i}\left(g\cdot\iota(f)\cdot
  \begin{tikzpicture}[baseline=.3cm]
   \draw (.2,0) -- (.2,.8);
   \node at (.2,1) {\scriptsize{$n$}};
   \draw (.6,.8) .. controls ++(-90:.3cm) and ++ (-270:.3cm) .. (1,0);
   \node at (.6,1) {\scriptsize{$j$}};
   \draw (.8,.8) arc (-180:0:.2cm);
   \node at (.8,1) {\scriptsize{$i$}};
   \node at (1.2,1) {\scriptsize{$i$}};
   \draw (.4,0) arc (180:0:.2cm);
   \node at (.4,-.2) {\scriptsize{$i$}};
   \node at (.8,-.2) {\scriptsize{$i$}};
   \draw[thick, rounded corners=5pt] (0,0) rectangle (1.4,.8);
  \end{tikzpicture}
  \right)\right)\cdot
  \begin{tikzpicture}[baseline=.3cm]
   \draw (.2,0) -- (.2,.8);
   \node at (.2,1) {\scriptsize{$n$}};
   \draw (.4,.8) arc (-180:0:.2cm);
   \node at (.4,1) {\scriptsize{$i$}};
   \node at (.8,1) {\scriptsize{$i$}};
   \draw (.6,0) .. controls ++(90:.3cm) and ++(270:.3cm) .. (1,.8);
   \node at (1,1) {\scriptsize{$k$}};
   \draw (.8,0) .. controls ++(90:.3cm) and ++(270:.3cm) .. (1.2,.8);
   \node at (1.2,1) {\scriptsize{$j$}};
   \draw (1,0) arc (180:0:.2cm);
   \node at (1,-.2) {\scriptsize{$i$}};
   \node at (1.4,-.2) {\scriptsize{$i$}};
   \draw[thick, rounded corners=5pt] (0,0) rectangle (1.6,.8);
  \end{tikzpicture}
  \\&= (g\circ f)\vartriangleleft1_m.
 \end{align*}

 It is easier to check that the action is associative.
Since we have already shown that $\vartriangleleft$ is a bifunctor, it suffices to check associativity of triples of morphisms when two are the identity. 
 That $1\vartriangleleft(1\otimes h)=(1\vartriangleleft1)\vartriangleleft h$ and that $1\vartriangleleft(g\otimes 1)=(1\vartriangleleft g)\vartriangleleft 1$ for all $g,h\in\TLJ$ follow directly from the definitions of $\vartriangleleft$ and $\otimes$. 
 Finally, in case $f\in\cM([n]\to[n+2i])$, we have 
  \[f\vartriangleleft(1_{[j+i]}\otimes1_{[k+i]}) = f\vartriangleleft 1_{[j+k+2i]}
  = \iota(f)\cdot
  \begin{tikzpicture}[baseline=.3cm]
   \draw (.2,0) -- (.2,.8);
   \node at (.2,1) {\scriptsize{$n$}};
   \draw (.4,.8) arc (-180:0:.2cm);
   \node at (.4,1) {\scriptsize{$i$}};
   \node at (.8,1) {\scriptsize{$i$}};
   \draw (.6,0) .. controls ++(90:.3cm) and ++(270:.3cm) .. (1,.8);
   \node at (1,1) {\scriptsize{$j$}};
   \draw (.8,0) .. controls ++(90:.3cm) and ++(270:.3cm) .. (1.2,.8);
   \node at (1.2,1) {\scriptsize{$i$}};
   \draw (1,0) .. controls ++(90:.3cm) and ++(270:.3cm) .. (1.4,.8);
   \node at (1.4,1) {\scriptsize{$k$}};
   \draw (1.2,0) arc (180:0:.2cm);
   \node at (1,-.2) {\scriptsize{$i$}};
   \node at (1.6,-.2) {\scriptsize{$i$}};
   \draw[thick, rounded corners=5pt] (0,0) rectangle (1.8,.8);
  \end{tikzpicture}
  = \iota(f)\cdot
  \begin{tikzpicture}[baseline=.3cm]
   \begin{scope}[shift={(0,.4)}]
    \draw (.2,0) -- (.2,.8);
    \node at (.2,1) {\scriptsize{$n$}};
    \draw (.4,.8) arc (-180:0:.2cm);
    \node at (.4,1) {\scriptsize{$i$}};
    \node at (.8,1) {\scriptsize{$i$}};
    \draw (.6,0) .. controls ++(90:.3cm) and ++(270:.3cm) .. (1,.8);
    \node at (1,1) {\scriptsize{$j$}};
    \draw (.8,0) arc (180:0:.2cm); 
    \draw (1.6,0) -- (1.6,.8);
    \node at (1.6,1) {\scriptsize{$i$}};
    \draw (1.8,0) -- (1.8,.8);
    \node at (1.8,1) {\scriptsize{$k$}};
    \draw[dashed] (0,0) -- (2.2,0);
   \end{scope}
   \begin{scope}[shift={(0,-.4)}]
    \draw (.2,0) -- (.2,.8);
    \node at (.2,-.2) {\scriptsize{$n$}};
    \draw (.6,0) -- (.6,.8);
    \node at (.6,-.2) {\scriptsize{$j$}};
    \draw (.8,0) -- (.8,.8);
    \node at (.8,-.2) {\scriptsize{$i$}};
    \draw (1.2,.8) arc (-180:0:.2cm); 
    \draw (1.4,0) .. controls ++(90:.3cm) and ++(270:.3cm) .. (1.8,.8);
    \node at (1.4,-.2) {\scriptsize{$k$}};
    \draw (1.6,0) arc (180:0:.2cm);
    \node at (1.6,-.2) {\scriptsize{$i$}};
    \node at (2,-.2) {\scriptsize{$i$}};
   \end{scope}
   \draw[thick, rounded corners=5pt] (0,-.4) rectangle (2.2,1.2);
  \end{tikzpicture}
  = (f\vartriangleleft1_{[j+i]})\vartriangleleft1_{[k+i]}\]
 The other cases are similar.
\end{defn}

\begin{rem}
 Notice that the principal graph of $M_\bullet$ is exactly the fusion graph for the associated $\cT\cL\cJ$-module category $\cM$ with respect to the unshaded-shaded strand $X\in \cT\cL\cJ$ with basepoint the simple projection $1_{M_0}\in \cM$. By Remark \ref{rem:HowMCategoryChangesUnderMarkovTowerOperations}, the operation of shifting the Markov tower does not change $\cM$ (up to equivalence), but corresponds to replacing the basepoint $1_{M_0}$ with $[0]\vartriangleleft X^{\text{alt}\otimes 2n}$. 
 Similarly, compressing by a minimal projection $p\in M_n$ corresponds to moving the basepoint to $p$.
 In contrast, the multistep basic construction, which may affect the principal graph, is analagous to replacing $X\in\TLJ$ with $X^{\text{alt}\otimes 2n}$, without changing the basepoint of $\cM$. The $\TLJ$-module structure on the category of projections $\cM^{(n)}$ of $M_{n\bullet}$ comes from combining the action of the subcategory $\TLJ^{(n)}$ of $\TLJ$ generated by subobjects of $[kn]_{\TLJ}$
 and the pivotal $\TLJ$-right module structure of $\TLJ^{(n)}$.
\end{rem}

\begin{rem}
 \label{rem:TLAsCatOfProjections}
Observe that we may identify the \emph{tensor} category $\TLJ$ as the category of projections of the Markov tower ${TLJ}_\bullet$, where the tensor structure is given by the $\TLJ$-module structure from Definition \ref{def:ModuleFromMarkovTower}. 
 One should think of the definitions for composition and tensor product in the category of projections as being obtained by isotoping the much simpler definitions for a planar algebra into a form that can be written down in terms of the data of a Markov tower. 
 Notice that under this identification, $\ev_{[n]}\in\cT\cL\cJ([2n]\to[0])$ and $\coev_{[n]}\in\cT\cL\cJ([0]\to[2n])$ are both identified with $1_n\in TLJ_n$ under Definition \ref{def:MarkovProjections}; one then checks they satisfy the zig-zag axioms using the definitions of $\vartriangleleft$ from Definition \ref{def:ModuleFromMarkovTower} and composition from Definition \ref{def:MarkovProjections}.
\end{rem}

\begin{defn}[Pivotal module structure] 
\label{def:PivotalModuleFromMarkovTower}
The category $\cM$ obtains a unitary trace $\Tr^\cM$ from the Markov tower $(M_\bullet,\tr_\bullet)$, by renormalizing so that isomorphic projections have the same trace.
For $[n] \in \cM$, we define $\Tr^\cM_{[n]} : \cM([n] \to [n]) = M_n \to \bbC$ by $d^n \tr_n$.
It is clear that $\Tr^\cM_{[n]}$ satisfies \ref{PivotalModule:positive} for all $n\geq 0$.
Now if $x \in \cM([n] \to [n+ 2k])= M_{n+k}$ and $y\in \cM ([n+ 2k] \to [n]) = M_{n+ k}$, observe that 
\begin{align*}
\Tr^\cM_{[n]}(y\circ x)
&\underset{\ref{compose:updown}}{:=}
d^n\tr_n(E^{n+k}_n(yx))
=
d^n \tr_{n+k}(yx)
=
d^{n+k}
\cdot
\tr_{n+k}
\left(
x \cdot
d^{k}\,
\begin{tikzpicture}[baseline = .3cm]
 \draw[thick, rounded corners = 5pt] (0,0) rectangle (1.2,.8);
 \draw (.2,0) -- (.2,.8);
 \draw (.5,.8) arc (-180:0:.2cm);
 \draw (.5,0) arc (180:0:.2cm);
 \node at (.2,1) {\scriptsize{$n$}};
 \node at (.5,1) {\scriptsize{$k$}};
 \node at (.5,-.2) {\scriptsize{$k$}};
\end{tikzpicture}\,
\cdot y
\right)
\\&\underset{\ref{compose:downup}}{=:}
\Tr^\cM_{[n+2k]}(x\circ y),
\end{align*}
and thus $\Tr^\cM$ satisfies \ref{PivotalModule:tracial}.
To see that $\Tr^\cM$ satisfies \ref{PivotalModule:compatible}, we must show that for all $n$, $k$, and $f\in\cM([n]\to[n])$,
 \[\Tr^\cM_{[n]\vartriangleleft[k]}(f)=\Tr^\cM_{[n]}\left((1\vartriangleleft\coev_{[k]}^\dag)\circ(f\vartriangleleft1_{[k]})\circ(1\vartriangleleft\coev_{[k]})\right).
 \]
 Now by Remark \ref{rem:TLAsCatOfProjections}, it is straightforward to show the left and right sides of the above equation are respectively equal to the left and right sides of the following equation, which trivially holds:
  \[
  \Tr(f)
  =
  \Tr\left(f\cdot E^{n+2k}_n\left(d^{k}\,
  \begin{tikzpicture}[baseline=.3cm]
   \draw (.2,0) -- (.2,.8);
   \node at (.2,1) {\scriptsize{$n$}};
   \draw (.4,.8) arc (-180:0:.2cm);
   \node at (.4,1) {\scriptsize{$k$}};
   \draw (.4,0) arc (180:0:.2cm);
   \node at (.4,-.2) {\scriptsize{$k$}};
   \draw[thick, rounded corners = 5pt] (0,0) rectangle (1,.8);
  \end{tikzpicture}
 \right)\right)
 =
 \Tr\left(E^{n+2k}_n\left(\iota(f)\cdot d^{k}\,
  \begin{tikzpicture}[baseline=.3cm]
   \draw (.2,0) -- (.2,.8);
   \node at (.2,1) {\scriptsize{$n$}};
   \draw (.4,.8) arc (-180:0:.2cm);
   \node at (.4,1) {\scriptsize{$k$}};
   \draw (.4,0) arc (180:0:.2cm);
   \node at (.4,-.2) {\scriptsize{$k$}};
   \draw[thick, rounded corners = 5pt] (0,0) rectangle (1,.8);
  \end{tikzpicture}
 \right)\right)
 \text{.}
 \]
\end{defn}

\begin{proof}[Proof of Corollary \ref{cor:TLJPivotalModuleClassification}]
We saw in Definitions \ref{def:ModuleFromMarkovTower} and \ref{def:PivotalModuleFromMarkovTower} how a connected Markov tower of modulus $d$ gives us a cyclic pivotal right $\cT\cL\cJ(d)$-module $\Cstar$ category.
We saw in Example \ref{ex:MarkovTowerFromRightModule} that given a cyclic pivotal right $\cT\cL\cJ(d)$-module $\Cstar$ category $(\cM, m,\Tr^\cM)$, 
defining $M_n := \End(m \vartriangleleft X^{\text{alt}\otimes n})$ and $\tr_n := \Tr^\cM_{m \vartriangleleft X^{\text{alt}\otimes n}}$ defines a connected Markov tower.
One now shows these two processes are mutually inverse up to dagger equivalence.
\end{proof}

\begin{rem}
\label{rem:LambdaLatticeModule}
While the process of defining a tensor structure on the category of projections $\cP$ of a Markov tower $P_\bullet$ obtained from a planar algebra $\cP_\bullet$ is fairly straightforward and similar to Remark \ref{rem:TLAsCatOfProjections}, 
it is far less obvious for a Markov tower coming from a standard $\lambda$-lattice as in \cite{MR1334479}.
Given a standard $\lambda$-lattice $A_{\bullet\bullet}$ with $\lambda = d^{-2}$, we expect that a Markov tower $M_\bullet$ of modulus $d$ such that $M_n \supset A_{0n}$ for all $n\geq 0$ satisfying certain compatibility conditions is the equivalent notion of a right module for $A_{\bullet\bullet}$ in the spirit of Theorems \ref{thm:ModuleEquivalence} and \ref{thm:SubfactorTensorCategoryEquivalence}.
$$
\begin{tikzpicture}
 \node at (0,1) {$M_{0}$};
 \node at (.5,1) {$\subset$};
 \node at (1,1) {$M_{1}$};
 \node at (1.5,1) {$\subset$};
 \node at (2,1) {$M_{2}$};
 \node at (2.5,1) {$\subset$};
 \node at (3,1) {$M_{3}$};
 \node at (3.5,1) {$\subset$};
 \node at (4,1) {$\cdots$};
 \node at (0,.5) {$\cup$};
 \node at (1,.5) {$\cup$};
 \node at (2,.5) {$\cup$};
 \node at (3,.5) {$\cup$};
 \node at (0,0) {$A_{00}$};
 \node at (.5,0) {$\subset$};
 \node at (1,0) {$A_{01}$};
 \node at (1.5,0) {$\subset$};
 \node at (2,0) {$A_{02}$};
 \node at (2.5,0) {$\subset$};
 \node at (3,0) {$A_{03}$};
 \node at (3.5,0) {$\subset$};
 \node at (4,0) {$\cdots$};
 \node at (1,-.5) {$\cup$};
 \node at (2,-.5) {$\cup$};
 \node at (3,-.5) {$\cup$};
 \node at (1,-1) {$A_{11}$};
 \node at (1.5,-1) {$\subset$};
 \node at (2,-1) {$A_{12}$};
 \node at (2.5,-1) {$\subset$};
 \node at (3,-1) {$A_{13}$};
 \node at (3.5,-1) {$\subset$};
 \node at (4,-1) {$\cdots$};
 \node at (2,-1.5) {$\cup$};
 \node at (3,-1.5) {$\cup$};
 \node at (2,-2) {$A_{22}$};
 \node at (2.5,-2) {$\subset$};
 \node at (3,-2) {$A_{23}$};
 \node at (3.5,-2) {$\subset$};
 \node at (4,-2) {$\cdots$};
 \node at (3,-2.5) {$\cup$};
 \node at (3,-3) {$A_{33}$};
 \node at (3.5,-3) {$\subset$};
 \node at (4,-3) {$\cdots$};
 \node at (4,-3.5) {$\ddots$};
\end{tikzpicture}
$$
We leave this exploration to a future joint article, as it would take us too far afield.
\end{rem}

\section{The canonical planar algebra from a strongly Markov inclusion} 

We begin this section by recalling the construction of the canonical planar $\dag$-algebra from a strongly Markov inclusion.
The reader is advised to review the definition of a strongly Markov inclusion from Definition \ref{def:StronglyMarkovInclusion} before proceeding.
We then discuss various operations on the inclusion, and how such operations affect (or do not affect!) the planar algebra.

\subsection{The canonical relative commutant planar algebra}
\label{sec:StronglyMarkovPA}

There is a canonical planar algebra structure on the towers of relative commutants, called the \textit{canonical planar $\dag$-algebra of a strongly Markov inclusion} corresponding to $A_0\subseteq (A_1,\tr_1)$. 
Denote by $(A_n,\tr_n)_{n\geq 0}$ the Jones tower for inclusion $A_0\subset (A_1,\tr_1)$.
The box spaces are defined by the relative commutants
$$
\cP_{n,+}:=A_0'\cap A_{n} 
\qquad\qquad
\cP_{n,-}:=A_1'\cap A_{n+1},
$$
which are finite dimensional by \cite[Prop.~2.7.3]{MR996807}.
We refer the reader to \cite[\S2.3]{MR2812459} for the action of tangles.
The $\dag$-structure is given by $*$ in the relative commutants. 
We remark that one important feature of this construction is that it depends on the existence of a Pimsner-Popa basis for $A_1$ over $A_0$, but not on a choice of basis. 

The following theorem uniquely characterizes the canonical relative commutant planar $\dag$-algebra.

\begin{thm}[{\cite[Thm.~2.50]{MR2812459}}]
Given a strongly Markov inclusion $A_0 \subset (A_1, \tr_1)$, there is a unique planar $\dag$-algebra $\cP_\bullet$ of modulus $d = [A_1: A_0]^{1/2}$ whose box spaces are given by
$$
\cP_{n,+}:=A_0'\cap A_{n} 
\qquad\qquad
\cP_{n,-}:=A_1'\cap A_{n+1},
$$
such that
\begin{enumerate}[label={\rm(PA\arabic*)}]
\item
The $\dag$-structure of $\cP_{n,\pm}$ is given by $x^\dag = x^*$ in the relative commutant, and stacking corresponds to multiplication in the relative commutant:
$$
\begin{tikzpicture}[baseline=.35cm]
 \draw (0,-.5) -- (0,1.4);
 \node at (-.2,1.35) {\scriptsize{$n$}};
 \node at (-.2,.45) {\scriptsize{$n$}};
 \node at (-.2,-.45) {\scriptsize{$n$}};
 \roundNbox{unshaded}{(0,0)}{.25}{0}{0}{$y$}
 \roundNbox{unshaded}{(0,.9)}{.25}{0}{0}{$x$}
\end{tikzpicture}
=
xy \in \cP_{n,\pm}.
$$
\item
The Jones projection $e_{n} \in A_0'\cap A_n$ for the strongly Markov inclusion $A_{n-1}\subset A_n$ is given by the following tangles depending on parity for $k\geq 0$:
$$
e_{2k+1,+}
:=
\begin{tikzpicture}[baseline]
 \draw (0,-.5) -- (0,.5);
 \filldraw[shaded] (.3,.5) arc (-180:0:.3cm);
 \filldraw[shaded] (.3,-.5) arc (180:0:.3cm);
 \node at (.2,0){\tiny{$2k$}};
\end{tikzpicture} 
\qquad
e_{2k+2,+}
:=
\begin{tikzpicture}[baseline]
 \fill[shaded] (0,-.5) rectangle (1.2,.5);
 \draw (0,-.5) -- (0,.5);
 \filldraw[unshaded] (.3,.5) arc (-180:0:.3cm);
 \filldraw[unshaded] (.3,-.5) arc (180:0:.3cm);
 \node at (.4,0){\tiny{$2k+1$}};
\end{tikzpicture} \,.
$$

\item
For $x\in \cP_{n,+} = A_0'\cap A_n$ and $\{b\}$ a Pimsner-Popa basis for $A_1$ over $A_0$,
\begin{itemize}
\item
$
\begin{tikzpicture}[baseline=-.1cm]
 \draw (.4,-.5) -- (.4,.5);
 \draw (0,-.5) -- (0,.5);
 \node at (-.2,.45) {\scriptsize{$n$}};
 \node at (-.2,-.45) {\scriptsize{$n$}};
 \roundNbox{unshaded}{(0,0)}{.25}{0}{0}{$x$}
\end{tikzpicture}
=
x\in \cP_{n+1,+}=A_0'\cap A_{n+1}
$,

\item
$
\begin{tikzpicture}[baseline=-.1cm]
 \draw (.15,-.25) arc (-180:0:.15cm) -- (.45,.25) arc (0:180:.15cm);
 \draw (0,-.5) -- (0,.5);
 \node at (-.4,.45) {\scriptsize{$n-1$}};
 \node at (-.4,-.45) {\scriptsize{$n-1$}};
 \roundNbox{unshaded}{(0,0)}{.25}{0}{0}{$x$}
\end{tikzpicture}
=
dE^{A_n}_{A_{n-1}}(x) 
\in 
\cP_{n-1,+}=A_0'\cap A_{n-1}
$, and

\item
$
\begin{tikzpicture}[baseline=-.1cm, xscale=-1]
 \fill[shaded] (.6,-.5) -- (.6,.5) -- (0,.5) -- (0,-.5);
 \filldraw[unshaded] (.15,-.25) arc (-180:0:.15cm) -- (.45,.25) arc (0:180:.15cm);
 \draw (0,-.5) -- (0,.5);
 \node at (-.4,.45) {\scriptsize{$n-1$}};
 \node at (-.4,-.45) {\scriptsize{$n-1$}};
 \roundNbox{unshaded}{(0,0)}{.25}{0}{0}{$x$}
\end{tikzpicture}
=
d E^{A_0'}_{A_1'}(x)
=
d^{-1}\sum_b bxb^*
\in 
\cP_{n,-}
=
A_1'\cap A_n
$.
\end{itemize}

\item
For $x\in \cP_{n,-} = A_1'\cap A_{n+1}$, 
$
\begin{tikzpicture}[baseline=-.1cm, xscale=-1]
 \fill[shaded] (.4,-.5) -- (.4,.5) -- (0,.5) -- (0,-.5);
 \draw (.4,-.5) -- (.4,.5);
 \draw (0,-.5) -- (0,.5);
 \node at (-.2,.45) {\scriptsize{$n$}};
 \node at (-.2,-.45) {\scriptsize{$n$}};
 \roundNbox{unshaded}{(0,0)}{.25}{0}{0}{$x$}
\end{tikzpicture}
=
x\in \cP_{n+1,+}= A_0'\cap A_{n+1}
$.
\end{enumerate}
\end{thm}

We will proceed with the same general technique for defining the embedding of planar algebras as in \cite{MR2812459}, in that we will also initially embed into the canonical $\dag$-planar algebra, and then make use of the following theorem:

\begin{thm}[{\cite[Theorem 3.28]{MR2812459}}] \label{planaralgebraisomorphism}
The canonical planar $\dag$-algebra associated to the strongly Markov inclusion of finite dimensional von Neumann algebras $A_0\subseteq (A_1,\tr_1)$ is isomorphic to the bipartite graph planar $\dag$-algebra of the Bratteli diagram for the inclusion $A_0\subseteq A_1$.
\end{thm}

The following two subsections  describe two isomorphisms between canonical $\dag$-planar algebras of related strongly Markov inclusions. 
Both are well known to experts, and will motivate constructions detailed in the later sections of this article.

\subsection{The shift isomorphism} 
\label{ssec:shift}

In the rest of this article, we will make extensive use of the following lemma adapted from \cite[Lem.~2.49]{MR2812459}, which provides sufficient conditions for a collection of maps to be a morphism of shaded planar $\dag$-algebras.

\begin{lem}[{\cite[Variation of Lem.~2.49]{MR2812459}}]
\label{lem:SufficientConditionsForPlanarMap}
Suppose $\Phi_{n,\pm} : \cP_{n,\pm} \to \cQ_{n,\pm}$ is a collection of unital $\dag$-algebra maps such that
\begin{enumerate}[label={\rm(\arabic*)}]
\item
$\Phi$ maps Jones projections in $\cP_{n,+}$ to Jones projections in $\cQ_{n,+}$,
\item
$\Phi$ commutes with the action of the following tangles:
$$
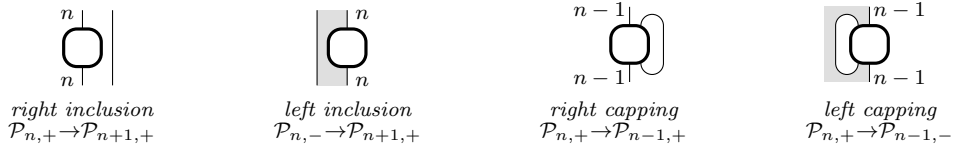
\underset{
\substack{
\text{right inclusion}
\\ 
\cP_{n,+} \to \cP_{n+1,+}
}}
{
\begin{tikzpicture}
 \draw (.4,-.5) -- (.4,.5);
 \draw (0,-.5) -- (0,.5);
 \node at (-.2,.45) {\scriptsize{$n$}};
 \node at (-.2,-.45) {\scriptsize{$n$}};
 \roundNbox{unshaded}{(0,0)}{.25}{0}{0}{}
\end{tikzpicture}
}
\qquad\qquad
\underset{
\substack{
\text{left inclusion}
\\ 
\cP_{n,-} \to \cP_{n+1,+}
}}
{
\begin{tikzpicture}[xscale=-1]
 \fill[shaded] (.4,-.5) -- (.4,.5) -- (0,.5) -- (0,-.5);
 \draw (.4,-.5) -- (.4,.5);
 \draw (0,-.5) -- (0,.5);
 \node at (-.2,.45) {\scriptsize{$n$}};
 \node at (-.2,-.45) {\scriptsize{$n$}};
 \roundNbox{unshaded}{(0,0)}{.25}{0}{0}{}
\end{tikzpicture}
}
\qquad\qquad
\underset{
\substack{
\text{right capping}
\\ 
\cP_{n,+} \to \cP_{n-1,+}
}}
{
\begin{tikzpicture}
 \draw (.15,-.25) arc (-180:0:.15cm) -- (.45,.25) arc (0:180:.15cm);
 \draw (0,-.5) -- (0,.5);
 \node at (-.4,.45) {\scriptsize{$n-1$}};
 \node at (-.4,-.45) {\scriptsize{$n-1$}};
 \roundNbox{unshaded}{(0,0)}{.25}{0}{0}{}
\end{tikzpicture}
}
\qquad\qquad
\underset{
\substack{
\text{left capping}
\\ 
\cP_{n,+} \to \cP_{n-1,-}
}}
{
\begin{tikzpicture}[xscale=-1]
 \fill[shaded] (.6,-.5) -- (.6,.5) -- (0,.5) -- (0,-.5);
 \filldraw[unshaded] (.15,-.25) arc (-180:0:.15cm) -- (.45,.25) arc (0:180:.15cm);
 \draw (0,-.5) -- (0,.5);
 \node at (-.4,.45) {\scriptsize{$n-1$}};
 \node at (-.4,-.45) {\scriptsize{$n-1$}};
 \roundNbox{unshaded}{(0,0)}{.25}{0}{0}{}
\end{tikzpicture}
}.
$$
\end{enumerate}
Then $\Phi$ is a morphism of shaded planar $\dag$-algebras.
\end{lem}

Suppose that $A_0\subset A_1 \subset \left(A_2,\tr_2,e_1\right)$ is a strongly Markov inclusion of von Neumann algebras and $\left(A_n,\tr_n,e_n\right)_{n\geq 0}$ is the tower obtained by iterating the basic construction. 
We know from \cite[Cor.~2.18]{MR2812459} that for any $0\leq k \leq n$, the inclusion $A_k \subset A_n$ is strongly Markov. 
Thus, we can find a Pimsner-Popa basis $B$ for $A_n$ over $A_k$.

By \cite[Prop.~2.20]{MR2812459}, for $0\leq j\leq 2n$, we can represent $A_j$ on $L^2(A_n, \tr_n)$ via the multistep basic construction, and $J_n A_{2n-j} J_n = A_j'\cap B(L^2(A_n, \tr_n))$ where $J_n$ is the modular conjugation.
We get a canonical trace on $A_j'$ by $\tr_j'(x):=\tr_{2n-j}(J_n x^*J_n)$ as discussed in Remark \cite[Rem.~2.21]{MR2812459}.

\begin{prop}[{\cite[Prop.~2.24]{MR2812459}}]
\label{prop:LeftCapping}
Let $0\leq j \leq k \leq n $. 
Let $\{b\}$ be a Pimsner-Popa basis for $A_k$ over $A_j$.
The conditional expectation $E_{A'_k}^{A'_j}: (A'_j\cap B(L^2(A_n,\tr_n)),\tr'_j)\rightarrow (A'_k \cap B(L^2(A_n,\tr_n)),\tr'_k)$ is given by:
\[
E^{A'_j}_{A'_k}(x)
=
d^{-2(k-j)} \sum_{b } bxb^{\ast}
\]
and is independent of the choice of basis.
\end{prop}
From the definition of the canonical $\dag$-planar algebra $\cP_\bullet$ of $A_\bullet$, one can work out that (in the language of Lemma \ref{lem:SufficientConditionsForPlanarMap}), right capping is simply the conditional expectation in the Markov tower $A_\bullet$, while left capping is the expectation on relative commutants described in the previous proposition. 
Thus, left capping is how the Pimsner-Popa basis shows itself graphically. 
The full details can be found in \cite[Prop.~2.47]{MR2812459}. 

\begin{cor}
\label{cor:OntoCor}
Adding $k$ strings to the left of $x$ gives a unital $\ast$-algebra isomorphisms
$\cP_{n,\pm} \to \cP_{n+k,\pm'}$
where $\pm' = \pm$ if $k$ is even and $\mp$ if $k$ is odd.
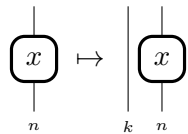
\[
\begin{tikzpicture}[baseline=-.1cm]
\draw (0,-.7) -- (0,.7);
\roundNbox{unshaded}{(0,0)}{.3}{0}{0}{$x$}
\node at (0,-.9) {\tiny{$n$}};
\end{tikzpicture}
\,\,\mapsto
\begin{tikzpicture}[baseline=-.1cm]
\draw (0.45,-.7) -- (0.45,.7);
\draw (0,-.7) -- (0,.7);
\roundNbox{unshaded}{(0.45,0)}{.3}{0}{0}{$x$}
\node at (0.45,-.9) {\tiny{$n$}};
\node at (0,-0.9){\tiny{$k$}};
\end{tikzpicture}
\]

\end{cor}

\begin{proof} 
We first prove the result for $\cP_{n,+} = A_0'\cap A_n$. 
The following implicitly uses a trick due to Vaughan Jones that can be found in \cite[Theorem~4.1]{MR2812459} along with the pictorial description of $f^{n}_{n-k}$ in the Multistep Basic Construction \cite[Remark~2.44]{MR2812459} in the equality marked $(!)$. 
Let $B$ be a Pimsner-Popa basis for $A_k$ over $A_0$.
For $x\in A'_k\cap A_{n+k}$, the element $y\in A_0'\cap A_n$ is uniquely determined by
\begin{equation*}
x
=
E_{A'_k}^{A_0'}(x)
=
d^{-2k} \sum_{b \in B} bxb^{\ast}
\underset{(!)}{=}
d^{-k}
\begin{tikzpicture}[baseline=-.1cm]
\draw (-.15,.3) arc(0:180:.15) -- (-.45,-.3) arc(-180:0:.15);
\draw (-.75,-.7) -- (-.75,.7);
\draw (0,-.7) -- (0,.7);
\roundNbox{unshaded}{(0,0)}{.3}{0}{0}{$x$};
\node at (-.55,0) {\tiny{$k$}};
\node at (-.85,0) {\tiny{$k$}};
\end{tikzpicture}
=
\begin{tikzpicture}[baseline=-.1cm]
\draw (-.15,.3) arc(0:180:.15) -- (-.45,-.3) arc(-180:0:.15);
\draw (-1.5,-.7) -- (-1.5,.7);
\draw (0,-.7) -- (0,.7);
\roundNbox{unshaded}{(0,0)}{.3}{0}{0}{$x$};
\draw[dashed] (-.75,-.5) -- (.5,-.5) -- (.5,.5) -- (-.75,.5) -- (-.75,-.5);
\node at (-.55,0) {\tiny{$k$}};
\node at (-1.6,0) {\tiny{$k$}};
\node at (0,-.9) {\tiny{$n$}};
\node at (-1.05,0) {\tiny{$d^{-k}$}};
\end{tikzpicture}
\quad
=
\begin{tikzpicture}[baseline=-.1cm]
\draw (0.45,-.7) -- (0.45,.7);
\draw (0,-.7) -- (0,.7);
\roundNbox{unshaded}{(0.45,0)}{.3}{0}{0}{$y$}
\node at (0.45,-.9) {\tiny{$n$}};
\node at (0,-0.9){\tiny{$k$}};
\end{tikzpicture}\,.
\end{equation*}
The proof is similar for $\cP_{n,-}$.
\end{proof}

Recall that the canonical planar $\dag$-algebra for the inclusion $A_0 \subseteq (A_1,\tr_1)$ is denoted by $\cP_{\bullet}$.
We denote the canonical planar $\dag$-algebra for $A_2 \subset (A_3,\tr_3)$ by $\cQ_{\bullet}$.

\begin{thm}[Shift Isomorphism]
\label{ShiftIso}
The map $\Phi:\cP_{\bullet} \to \cQ_{\bullet}$, obtained by adding two strings in front of elements of $\cP_{n,\pm}$
\[
\begin{tikzpicture}[baseline=-.1cm]
\draw (0,-.7) -- (0,.7);
\roundNbox{unshaded}{(0,0)}{.3}{0}{0}{$x$};
\node at (0,-.9) {\tiny{$n$}};
\end{tikzpicture}
\,\,\mapsto
\begin{tikzpicture}[baseline=-.1cm]
\draw (0.45,-.7) -- (0.45,.7);
\draw (0,-.7) -- (0,.7);
\roundNbox{unshaded}{(0.45,0)}{.3}{0}{0}{$x$}
\node at (0.45,-.9) {\tiny{$n$}};
\node at (0,-0.9){\tiny{$2$}};
\end{tikzpicture}
\qquad
\qquad
\qquad
\begin{tikzpicture}[baseline=-.1cm]
\fill[shaded] (-.5,-.7) rectangle (0,.7);
\draw (0,-.7) -- (0,.7);
\roundNbox{unshaded}{(0,0)}{.3}{0}{0}{$x$};
\node at (0,-.9) {\tiny{$n$}};
\end{tikzpicture}
\,\,\mapsto
\begin{tikzpicture}[baseline=-.1cm]
\fill[shaded] (-.3,-.7) rectangle (0.45,.7);
\draw (0.45,-.7) -- (0.45,.7);
\draw (0,-.7) -- (0,.7);
\roundNbox{unshaded}{(0.45,0)}{.3}{0}{0}{$x$}
\node at (0.45,-.9) {\tiny{$n$}};
\node at (0,-0.9){\tiny{$2$}};
\end{tikzpicture}
\]
defines a planar $\dag$-algebra isomorphism between $\cP_{\bullet}$ and $\cQ_{\bullet}$.
\end{thm}
\begin{proof}
The map is an isomorphism between box spaces due to Corollary \ref{cor:OntoCor}. 
In order to show that this map commutes with the action of tangles, we just have to show that it satisfies the requirements of Lemma \ref{lem:SufficientConditionsForPlanarMap}. 
We draw the string diagrams of $\cQ_{\bullet}$ in blue in order to increase clarity.
\begin{enumerate}[label={\rm(\arabic*)}]
\item (Right Inclusion)
$
\begin{tikzpicture}[baseline=-.1cm]
\draw[blue] (-.2,-0.7) -- (-.2,0.7);
\draw[blue] (0.45,-0.7) -- (0.45,0.7);
\roundNbox{unshaded, draw=blue}{(-.2,0)}{.3}{.15}{.15}{\textcolor{blue}{{$\Phi(x)$}}};
\node at (-.2,-0.9){\textcolor{blue}{\tiny{$n$}}};
\end{tikzpicture}
\quad
= 
\quad
\begin{tikzpicture}[baseline=-.1cm]
\draw (-0.2,-0.7) -- (-0.2,0.7);
\draw (0.2,-0.7) -- (0.2,0.7);
\roundNbox{unshaded}{(0,0)}{.3}{0.15}{0.15}{{$\Phi(x)$}};
\draw (0.6,-0.7) -- (0.6,0.7);
\node at (-0.2,-0.9){\tiny{$2$}};
\node at (0.2,-0.9){\tiny{$n$}};
\end{tikzpicture}
\quad
=
\quad
\begin{tikzpicture}[baseline=-.1cm]
\draw (0,-0.7) -- (0,0.7);
\draw (-0.6,-0.7) -- (-0.6,0.7);
\draw[dashed] (-0.45,-0.5) rectangle (0.6,0.5);
\roundNbox{unshaded}{(0,0)}{.3}{0}{0}{{$x$}};
\draw (0.45,-0.7) -- (0.45,0.7);
\node at (0,-0.9){\tiny{$n$}};
\node at (-0.6,-0.9){\tiny{$2$}};
\end{tikzpicture}
$
\item (Left Inclusion)
$
\begin{tikzpicture}[baseline=-.1cm,xscale=-1]
\fill[Bshading] (-.2,-0.7) rectangle (0.45,0.7);
\draw[blue] (-.2,-0.7) -- (-.2,0.7);
\draw[blue] (0.45,-0.7) -- (0.45,0.7);
\roundNbox{unshaded, draw=blue}{(-.3,0)}{.3}{.15}{.15}{\textcolor{blue}{{$\Phi(x)$}}};
\node at (0,-0.9){\textcolor{blue}{\tiny{$n$}}};
\end{tikzpicture}
\quad
= 
\quad
\begin{tikzpicture}[baseline=-.1cm,xscale=-1]
\fill[shaded] (-0.3,-0.7) rectangle (0.6,0.7);
\draw (-0.2,-0.7) -- (-0.2,0.7);
\draw (0.2,-0.7) -- (0.2,0.7);
\roundNbox{unshaded}{(0,0)}{.3}{0.15}{0.15}{{$\Phi(x)$}};
\draw (0.6,-0.7) -- (0.6,0.7);
\node at (-0.2,-0.9){\tiny{$n$}};
\node at (0.2,-0.9){\tiny{$2$}};
\end{tikzpicture}
\quad
=
\quad
\begin{tikzpicture}[baseline=-.1cm,xscale=-1]
\fill[shaded] (0,-0.7) rectangle (0.75,0.7);
\draw (0,-0.7) -- (0,0.7);
\draw (0.75,-0.7) -- (0.75,0.7);
\draw[dashed] (-0.45,-0.5) rectangle (0.6,0.5);
\roundNbox{unshaded}{(0,0)}{.3}{0}{0}{{$x$}};
\draw (0.45,-0.7) -- (0.45,0.7);
\node at (0,-0.9){\tiny{$n$}};
\node at (0.45,-0.9){\tiny{$2$}};
\end{tikzpicture}
$
\item (Right Capping)
$
\begin{tikzpicture}[baseline=-.1cm]
\draw[blue] (0,-0.7) -- (0,0.7);
\draw[blue] (0.2,0.3) arc(180:0:0.2) -- (0.6,-0.3) arc(0:-180:0.2);
\roundNbox{unshaded, draw=blue}{(0,0)}{.3}{.15}{.15}{\textcolor{blue}{{$\Phi(x)$}}};
\node at (0,-0.9){\textcolor{blue}{\tiny{$n$}}};
\end{tikzpicture}
\quad
=
\quad
\begin{tikzpicture}[baseline=-.1cm]
\draw (-0.2,-0.7) -- (-0.2,0.7);
\draw (0,-0.7) -- (0,0.7);
\roundNbox{unshaded}{(0,0)}{.3}{0.15}{0.15}{{$\Phi(x)$}};
\draw (0.2,0.3) arc(180:0:0.2) -- (0.6,-0.3) arc(0:-180:0.2);
\node at (-0.2,-0.9){\tiny{$2$}};
\node at (0,-0.9){\tiny{$n$}};
\end{tikzpicture}
\quad
=
\quad
\begin{tikzpicture}[baseline=-.1cm]
\draw (0,-0.7) -- (0,0.7);
\draw (-0.6,-0.7) -- (-0.6,0.7);
\draw[dashed] (-0.45,-0.55) rectangle (0.6,0.55);
\roundNbox{unshaded}{(0,0)}{.3}{0}{0}{{$x$}};
\draw (0.15,0.3) arc(180:0:0.15) -- (0.45,-0.3) arc(0:-180:0.15);
\node at (0,-0.9){\tiny{$n$}};
\node at (-0.6,-0.9){\tiny{$2$}};
\end{tikzpicture}
$
\item (Left Capping) We again apply the trick from \cite[Theorem~4.1]{MR2812459}. Let $B$ be a Pimsner-Popa basis of $A_3$ over $A_2$.
Then
\begin{equation*}
\begin{tikzpicture}[baseline=-0.1]
\fill[Bshading] (-0.8,-0.7) rectangle (0,0.7);
\draw[blue] (0,-0.7) -- (0,0.7);
\filldraw[unshaded,draw=blue] (-0.2,0.3) arc(0:180:0.2) -- (-0.6,-0.3) arc(-180:0:0.2);
\roundNbox{unshaded,draw=blue}{(0,0)}{0.3}{.15}{.15}{\textcolor{blue}{{$\Phi(x)$}}};
\node at (0,-0.9){\textcolor{blue}{\tiny{$n$}}};
\end{tikzpicture}
\quad
=
\quad
d^{-1}\sum_{b\in B} b\Phi(x)b^{\ast}
\quad
=
\begin{tikzpicture}[baseline=-0.1]
\fill[shaded] (-1.05,-0.7) rectangle (0,0.7);
\draw (-0.75,-0.7) -- (-0.75,0.7);
\draw (0,-0.7) -- (0,0.7);
\filldraw[unshaded] (-0.15,0.3) arc(0:180:0.15) -- (-0.45,-0.3) arc(-180:0:0.15);
\roundNbox{unshaded}{(0,0)}{0.3}{0}{0}{{$x$}};
\node at (0,-0.9){\tiny{$n$}};
\node at (-0.75,-0.9){\tiny{$2$}};
\draw[dashed] (-0.6,-0.55) rectangle (0.45,0.55);
\end{tikzpicture}
\,.
\qedhere
\end{equation*}
\end{enumerate}
\end{proof}

\begin{rem}
By Remark \ref{Rem:ShiftTowerEffects}, the categories of projections of $\cP_\bullet$ and $\cQ_\bullet$ as in Theorem \ref{ShiftIso} are equivalent.
\end{rem}


\subsection{The compression isomorphism}
\label{sec:CompressionIso}

The following lemma is well known to experts. 
We provide a proof for convenience and completeness.

\begin{lem}
\label{lem:CompressRelativeCommutant}
Suppose $N\subset M\subset B(H)$ is an inclusion of von Neumann algebras and $p\in P(N)$.
\begin{enumerate}[label={\rm(\arabic*)}]
\item
$p(N'\cap M) = pN' \cap pMp$.
\item
Suppose the central support of $p$ in $N$ is $z\in Z(N)$.
The map $x\mapsto px$ is an isomorphism $z(N'\cap M) \to p(N'\cap M)$.
\end{enumerate}
\end{lem}
\begin{proof}
\mbox{}
\item[\underline{Proof of (1):}]
The proof of (1) is similar to the proof of the standard fact that $(pNp)' = N'p$.

Clearly $(N'\cap M)p \subseteq (N'p) \cap pMp$.
Suppose $u$ is a unitary in $(N'p) \cap pMp$.
Let $K$ be the closure of $NpH$.
Let $q\in B(H)$ be the projection onto $K$, which is clearly in $N' \cap N = Z(N)$.
Define $u_0$ in $B(K)$ by $u_0 (np\xi) := npu\xi$.
One now verifies that $u_0$ is an isometry and thus is well-defined.
Look at the operator $u_0q \in N' \cap B(H)$, and note that $u = u_0qp \in N'p$.
We claim that $u_0q \in M$, so that $u = u_0qp \in (N' \cap M)p$.
First, for any $m \in M'$, $n \in N$, and $\xi\in H$,
we have
$mu_0np\xi  = mnup\xi = nupm\xi = u_0npm\xi = u_0mnp\xi$.
Thus $u_0 \in qMq$.
Since $q \in M$, for all $m \in M'$, we have $u_0qm\xi = u_0mq\xi = mu_0q\xi$.
Hence $u_0q$ commutes with $M'$ on $H$, and $u_0q \in M$. 

\item[\underline{Proof of (2):}]
For $x\in N'\cap M$, we have $p(zx) = px$.
Hence the map is surjective.
We now show the map is injective.
Suppose $x \in z(N'\cap M)$ such that $px = 0$.
By (1), $z(N'\cap M) = zN' \cap zMz$.
Then for all unitary $u\in U(N)$, $upz = z(up) \in zN$, so 
$0 = upxu^* = (upz)xu^* = x(upzu^*) = xupu^*$.
Taking sup over $u \in U(N)$ yields $0 = xz = x$.
\end{proof}

Similar to the discussion in \S\ref{sec:OperationsOnMarkovTowers}, 
given an inclusion of tracial von Neumann algebras $A_0\subset (A_1, \tr_1)$, we obtain another inclusion of tracial von Neumann algebras by compression by a non-zero projection $p\in P(A_0)$.
We define a faithful trace $\tr_1^p$ on $pA_1 p$ by
\begin{equation}
\label{eq:CompressedTrace1}
\tr^p_1(x) := \tr_1(p)^{-1}\tr_1(pxp).
\end{equation}
It is straightforward to verify that the unique trace-preserving conditional expectation is given by 
\begin{equation}
\label{eq:CompressedConditionalExpectation1}
E^p_1 : pA_1p \to pA_{0}p
\qquad
\qquad
E^p_1(pxp) := E_1(pxp) = pE_1(x)p
\end{equation}
Notice that since $[e_1,p] = 0$, we have for all $pxp \in pA_1 p$, we have 
\begin{equation}
\label{eq:CompressionImplementsConditionalExpectation1}
e_1p (pxp) e_1p = p e_1xe_1p = pE_1(x)e_1p = E_1^p(pxp)e_1p,
\end{equation}
so the conditional expectation is implemented by $e_1p$.

Suppose now that $A_0 \subset (A_1, \tr_1)$ is strongly Markov.
We would like to show that $pA_2p$ with trace $\tr^p_2(x):=\tr_2(p)^{-1}\tr_2(pxp)$ and Jones projection $pe_1$
is isomorphic to the basic construction of $pA_0p \subset (pA_1p, \tr_1^p)$, but we will need an extra assumption on $p$.
(This extra assumption will be automatic when $A_0 \subset A_1$ is a ${\rm II}_1$ subfactor; see also \cite[Lem.~2.4]{MR1262294}.)
Toward this goal, we recall the following recognition lemma based on \cite[Prop.~1.2]{MR965748}, \cite[Lem.~5.8]{MR1073519}, and \cite[Lem.~5.3.1]{MR1473221}.

\begin{lem}[{\cite[Lem.~2.15]{MR2812459}}]
\label{lem:BasicConstructionRecognition}
Suppose $A_0 \subset (A_1, \tr_1)$ is a strongly Markov inclusion of tracial von Neumann algebras, and $(B, \tr_B, p)$ is a tracial von Neumann algebra containing $A_1$ together with a projection $p\in P(B)$ such that
\begin{enumerate}[label={\rm(R\arabic*)}]
\item
\label{eq:RecognitionImplement}
$pxp = E_{A_0}^{A_1}(x)p$ for all $x\in A_1$,
\item
\label{eq:RecognitionMarkov}
$E^{B}_{A_1}(p) = [A_1:A_0]^{-1} 1_{A_1}$, and
\item
\label{eq:RecognitionSpan}
$B$ is algebraically spanned by $A_1$ and $p$, i.e., $B = A_1pA_1 := \spann\set{apb}{a,b\in A_1}$.
\end{enumerate}
Then the map $A_2\to B$ given by $ae_1b \mapsto apb$ is a (normal) unital $*$-isomorphism of von Neumann algebras.
\end{lem}
In this case, where $(B,\tr_B,p)$ is isomorphic to the basic construction $A_2$ of $A_0\subseteq (A_1,\tr_1)$, we call the inclusion $A_0\subseteq (A_1,\tr_1)\subseteq (B,\tr_B,p)$ \textit{standard}, after \cite{MR2812459}. 

We now prove that compression by well-behaved projections of $A_0$ preserves the strong Markov structure.
Here, `well-behaved' is the condition 
$A_0 = A_0 p A_0 := \spann\set{apb}{a,b\in A_0}$,
which implies that the central support of $p$ in $A_0$ is $1$.

\begin{prop}
\label{prop:CompressingStronglyMarkovInclusion}
Suppose $A_0 \subset (A_1, \tr_1)$ is a strongly Markov inclusion of tracial von Neumann algebras, and $p\in P(A_0)$ is a projection such that $A_0pA_0 = A_0$.
\begin{enumerate}[label={\rm(\arabic*)}]
\item
Let $\{a\}\subset A_0$ be a finite set such that $\sum_{a} apa^* = 1_{A_0}$.\footnote{
Such a finite set necessarily exists by the same trick used in \cite[Prop.~3(b)]{MR561983}.
}
Then for any Pimsner-Popa basis $\{b\}$ for $A_1$ over $A_0$, $\{pbap\}$ is a Pimsner-Popa basis for $pA_1p$ over $pA_0p$.
\item
The inclusion $pA_0p \subset (pA_1p, \tr_1^p)$ is strongly Markov with index $[pA_1p:pA_0p] = [A_1:A_0]$.
\item
The inclusion $pA_0p \subset (pA_1p, \tr_1^p) \subset (pA_2p , \tr_2^p, pe_1)$\footnote{
The definition of $\tr^p_2$ is analogous to \eqref{eq:CompressedTrace1}.
} is standard.
\end{enumerate}
\end{prop}
\begin{proof}
\mbox{}
\item[\underline{Proof of (1):}]
For all $pxp \in pA_1p$, we have
$$
\sum_{pbap} 
pbap
E^p_1(pa^*b^*p \cdot pxp)
=
\sum_b
\sum_a
pbapa^* E_1(b^*px)p
=
\sum_b pbE_1(b^*px)p
=
pxp.
$$

\item[\underline{Proof of (2):}]
Given that there exists a Pimsner-Popa basis for $pA_1p$ over $pA_0p$ by part (1), the inclusion is Markov if and only if the Watatani index \cite{MR996807} is a scalar by \cite[1.1.4(c)]{MR1278111}.
We now calculate
$$
\sum_{pbap} pbap(pbap)^*
=
\sum_b\sum_a pbapa^*b^*p
=
\sum_b pbb^*p
=
[A_1:A_0] p
=
[A_1:A_0] \id_{pA_1p}.
$$

\item[\underline{Proof of (3):}]
We show the hypotheses of Lemma \ref{lem:BasicConstructionRecognition} hold.
We already saw \ref{eq:RecognitionImplement} holds in \eqref{eq:CompressionImplementsConditionalExpectation1}.
To see \ref{eq:RecognitionMarkov} holds, note that $E^p_2(pxp) = pE_2(x)p$ for all $x\in A_2$ as in  \eqref{eq:CompressedConditionalExpectation1}.
Thus by part (2),
$$
E^{p}_{2}(pe_1) 
= 
pE_2(e_1) 
= 
[A_1:A_0]^{-1}p 
= 
[pA_1p:pA_0p]^{-1} 1_{pA_1p}.
$$
Finally, \ref{eq:RecognitionSpan} follows immediately from the existence of a Pimsner-Popa basis for $pA_1p$ over $pA_0p$ by part (1).
\end{proof}

By iterating Lemma \ref{lem:BasicConstructionRecognition} and Proposition \ref{prop:CompressingStronglyMarkovInclusion}, we immediately obtain the following.

\begin{cor}
\label{cor:CompressJonesTower}
Assume the hypotheses of Proposition \ref{prop:CompressingStronglyMarkovInclusion}, and let $A_\bullet=(A_n, \tr_n, e_{n+1})_{n\geq 0}$ be the Jones tower for $A_0 \subset (A_1, \tr_1)$.
Then $pAp_\bullet:=(pA_np, \tr_n^p, pe_{n+1})_{n\geq 0}$ is isomorphic to the Jones tower of $pA_0p \subset (pA_1p, \tr_1^p)$.
\end{cor}


Suppose $A_0\subset (A_1,\tr_1)$ is a strongly Markov inclusion of tracial von Neumann algebras.
Denote by $\cP_\bullet$ the canonical planar $\dag$-algebra whose box spaces are given by
$$
\cP_{n,+}
:=
A_0'\cap A_n
\qquad\qquad
\cP_{n,-}
:=
A_1'\cap A_{n+1}.
$$
Suppose $p\in P(A_0)$ is a projection such that $A_0pA_0 = A_0$.
By Corollary \ref{cor:CompressJonesTower}, the Jones tower for $pA_0p \subset (pA_1p, \tr^p_1)$ is given by $(pA_np, \tr^p_n, pe_{n+1})_{n\geq 0}$, and thus we get another canonical planar $\dag$-algebra $\cQ_\bullet$ whose box spaces are given by
$$
\cQ_{n,+}
:=
pA_0'\cap pA_np
\qquad\qquad
\cQ_{n,-}
:=
pA_1'\cap pA_{n+1}p.
$$
By Lemma \ref{lem:CompressRelativeCommutant}, the map $\Phi_{n,\pm}:x\mapsto xp$ gives an isomorphism of von Neumann algebras $\Phi_{n,\pm}:\cP_{n,\pm}\to \cQ_{n,\pm}$ for each $n\geq 0$.

\begin{thm}
\label{thm:CompresionIsomorphsim}
The maps $\Phi_{n,\pm} : \cP_{n,\pm} \to \cQ_{n,\pm}$ constitute a planar $\dag$-algebra isomorphism.
\end{thm}
\begin{proof}
We prove the unital $*$-algebra isomorphisms $\Phi_{n,\pm}$ satisfy the conditions of Lemma \ref{lem:SufficientConditionsForPlanarMap}.

First, note that $\Phi_{n,\pm}(e_n) = pe_n$, so Jones projections in $\cP_\bullet$ map to Jones projections in $\cQ_\bullet$ by Corollary \ref{cor:CompressJonesTower}.
Hence Condition (1) of Lemma \ref{lem:SufficientConditionsForPlanarMap} is satisfied.

The only interesting part in checking Condition (2) of Lemma \ref{lem:SufficientConditionsForPlanarMap} holds is verifying that left capping commutes with $\Phi_{n,\pm}$.
First, by Proposition \ref{prop:LeftCapping}, if $\{b\}$ is a Pimsner-Popa basis for $A_1$ over $A_0$, then for all $x\in \cP_{n,+}= A_0'\cap A_n$,
$$
\begin{tikzpicture}[xscale=-1, baseline = -.1cm]
 \fill[shaded] (.6,-.5) -- (.6,.5) -- (0,.5) -- (0,-.5);
 \filldraw[unshaded] (.15,-.25) arc (-180:0:.15cm) -- (.45,.25) arc (0:180:.15cm);
 \draw (0,-.5) -- (0,.5);
 \node at (-.4,.45) {\scriptsize{$n-1$}};
 \node at (-.4,-.45) {\scriptsize{$n-1$}};
 \roundNbox{unshaded}{(0,0)}{.25}{0}{0}{$x$}
\end{tikzpicture}
=
d^{-1} \sum_{\beta} b xb^*.
$$
This means that picking Pimsner-Popa bases $\{b\}$ for $A_1$ over $A_0$ and $\{a\}$ for $pA_1p$ over $pA_0p$, we must show that
\begin{equation}
\label{eq:LeftCappingForCompression}
\Phi_{n-1,-}
\left(
\begin{tikzpicture}[xscale=-1, baseline = -.1cm]
 \fill[shaded] (.6,-.5) -- (.6,.5) -- (0,.5) -- (0,-.5);
 \filldraw[unshaded] (.15,-.25) arc (-180:0:.15cm) -- (.45,.25) arc (0:180:.15cm);
 \draw (0,-.5) -- (0,.5);
 \node at (-.4,.45) {\scriptsize{$n-1$}};
 \node at (-.4,-.45) {\scriptsize{$n-1$}};
 \roundNbox{unshaded}{(0,0)}{.25}{0}{0}{$x$}
\end{tikzpicture}
\right)
=
d^{-1}p\sum_{b} b xb^* 
\overset{\text{?}}{=} 
d^{-1}\sum_{a} apxa^*
=
\begin{tikzpicture}[xscale=-1, baseline = -.1cm]
 \fill[shaded] (.85,-.5) -- (.85,.5) -- (0,.5) -- (0,-.5);
 \filldraw[unshaded] (.4,-.25) arc (-180:0:.15cm) -- (.7,.25) arc (0:180:.15cm);
 \draw (0,-.5) -- (0,.5);
 \node at (-.4,.45) {\scriptsize{$n-1$}};
 \node at (-.4,-.45) {\scriptsize{$n-1$}};
 \roundNbox{unshaded}{(0,0)}{.25}{.35}{.35}{\scriptsize{$\Phi_{n,+}(x)$}}
\end{tikzpicture}
.
\end{equation}
The trick is to carefully choose our Pimsner-Popa basis for $A_1$ over $A_0$.
We take the Pimsner-Popa basis $\{a\}$ for $pA_1p$ over $pA_0p$ and we take the disjoint union with $\{(1-p)b\}$, where $\{b\}$ was our Pimsner-Popa basis for $A_1$ over $A_0$.
We now claim $\{c\} = \{a\}\cup \{(1-p)b\}$ is a Pimsner-Popa basis for $A_1$ over $A_0$.
Indeed, since $a = pap \in pA_1p$ for all $a\in \{a\}$, we have
$$
\sum_{c} c e_1 c^*
=
\sum_{a} ape_1 a^* + \sum_{b} (1-p)be_1 b^*(1-p) 
=
p+(1-p)
= 
1_{A_2}.
$$
Thus for this special choice of Pimsner-Popa basis for $A_1$ over $A_0$, we immediately obtain
$$
p\sum_{c} c xc^* 
= 
p\left(\sum_{a} a(px)a^* + \sum_{b} (1-p)bx b^*(1-p) \right)
=
\sum_{a} apxa^*.
$$
Hence \eqref{eq:LeftCappingForCompression} holds, and the result follows.
\end{proof}

\section{The module embedding theorem via towers of algebras}

We have finally developed the tools necessary to prove Theorem \ref{thm:ModulePAEmbedding}, which turns a finite depth cyclic right pivotal module $(\cM,m)$ over the category of projections $\cC$ of a subfactor planar algebra $\cQ_\bullet$ (or equivalently, a finite depth connected right planar module over $\cQ_\bullet$) into an embedding of $\cQ_\bullet$ into the graph planar algebra of the fusion graph of $\cM$ with respect to the unshaded-shaded strand $X\in \cQ_{1,+}$. 
As a special case, we recover the embedding of a subfactor planar algebra into the graph planar algebra of its own principal graph, described in \cite{MR2812459}, along with an embedding into the graph planar algebra of the dual principal graph. 
We then verify that, up to an automorphism of the graph planar algebra, the resulting embedding does not depend on the choice of generating object for the module category.

\subsection{The Embedding Theorem}

Suppose $\cM_\bullet$ is a a connected right planar module over $\cQ_\bullet$. 
Since $\cQ_\bullet$ has finite depth, so does $\cM_\bullet$, by Lemma \ref{lem:FiniteDepthAlgebraHasFiniteDepthModule}. 
The tangle 
 
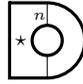
\[\begin{tikzpicture}[baseline]
  \halfcircle{}{(0,0)}{.5}{.5}{}
  \draw (.5,-.5) -- (.5,.5);
  \node at (.4,.32) {\tiny{$n$}}; 
  \ncircle{unshaded}{(.5,0)}{.2}{180}{}
 \end{tikzpicture}\]
gives a map $\cQ_\bullet\to \cM_\bullet$, which is injective since $\cM_\bullet$ is non-zero and connected. 
Let $M_\bullet$ be the Markov tower obtained from $\cM_\bullet$ from Example \ref{example:MarkovTowerFromPlanarModule}.
Choose $r\geq 0$ such that the inclusion $M_{2r}\subseteq (M_{2r+1},\tr_{2r+1})\subseteq (M_{2r+2}, \tr_{2r+2}, e_{2r+1})$ is standard.
Then setting $(A_n, \tr_n):=(M_{2r+n},\tr_{2r+n})$ for $n\geq 0$ as described in \S\ref{ssec:shift}, $A_0\subseteq (A_1, \tr_1)$ is a strongly Markov inclusion, and the tower $(A_n,\tr_{2r+n})$ is a strongly Markov tower. 
Let $\cP_\bullet$ be the canonical relative commutant planar $\dag$-algebra of the inclusion $A_0\subseteq (A_1, \tr_1)$ described in \S\ref{sec:StronglyMarkovPA}.
Therefore, the tangle 
 
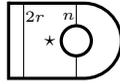
\[\begin{tikzpicture}[baseline]
  \halfcircle{}{(0,0)}{.5}{1}{}
  \draw (.2,-.5) -- (.2,.5);
  \node at (.35,.32) {\tiny{$2r$}};
  \draw (.9,-.5) -- (.9,.5);
  \node at (.8,.32) {\tiny{$n$}}; 
  \ncircle{unshaded}{(.9,0)}{.2}{180}{}
 \end{tikzpicture}\]
gives an embedding $\Phi:\cQ_n\to\cP_n$ on the level of Markov towers.
In terms of the string calculus of pivotal modules over tensor categories, $\Phi$ places $2r$ strings to the left of elements in $\cQ_{n,+}$  and $2r+1$ strings to the left of elements in $\cQ_{n,-}$:
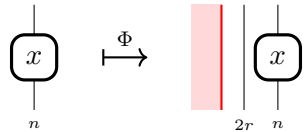
\[ \begin{tikzpicture}[baseline]
 \draw (0,-.7) -- (0,.7);
 \roundNbox{unshaded}{(0,0)}{.3}{0}{0}{$x$}
 \node at (0,-.9) {\tiny{$n$}};
\end{tikzpicture}
\quad
\begin{tikzpicture}[baseline]
 \clip (0.5,0.9) -- (-0.5,0.9) -- (-0.5,-0.9) -- (0.5,-0.9);
 \draw [|->,thick] (-0.3,0)--(0.3,0);
 \node at (0,0.25) {\scriptsize{$\Phi$}};
\end{tikzpicture}
\quad
\begin{tikzpicture}[baseline]
 \fill[ctwoshading] (-.3,-.7) -- (-0.3,.7) -- (-.7,0.7) -- (-.7,-0.7) -- (-.3,-0.7);
 \draw[thick,red] (-0.3,-.7) -- (-0.3,0.7);
 \draw (0.45,-.7) -- (0.45,.7);
 \draw (0,-.7) -- (0,.7);
 \roundNbox{unshaded}{(0.45,0)}{.3}{0}{0}{$x$}
 \node at (0.45,-.9) {\tiny{$n$}};
 \node at (0,-0.9){\tiny{$2r$}};
\end{tikzpicture} \]
Here, the unshaded-shaded strand is a generator of the category of projections $\cC$ of $\cQ_\bullet$, while the shaded-unshaded red strand is a simple generator of a cyclic pivotal right module category over $\cC$.

\begin{thm}
 The map $\Phi$ is a $\dag$-planar algebra embedding. 
\end{thm} 
\begin{proof}
 For clarity, let us denote the conditional expectations and inclusions in $\cQ_\bullet$ by $\cE$ and $\iota$, reserving the plain symbols for their counterparts in $\cP_\bullet$. 
 Similarly, let us denote the $n$-th Jones projection in $\cQ_\bullet$ by $\varepsilon_n$, reserving $e_n$ for the Jones projection in $\cP_\bullet$.

 We need only check the hypotheses of Lemma \ref{lem:SufficientConditionsForPlanarMap}: that $\Phi$ commutes with the action of several tangles.

\begin{itemize}
 \item (Jones Projections)
 $\displaystyle\Phi(\varepsilon_n)=\Phi\left(
\begin{tikzpicture}[baseline=-.1cm]
 \draw (-.1,-.7) -- (-.1,.7);
 \draw (.5,-.7) -- (.5,-.4) arc (180:0:.15cm) -- (.8,-.7);
 \draw (.5,.7) -- (.5,.4) arc (-180:0:.15cm) -- (.8,.7);
 \node at (-.1,-.9) {\tiny{$n-1$}};
\end{tikzpicture}
  \right)\,
=
\,
\begin{tikzpicture}[baseline=-.1cm]
 \fill[ctwoshading] (-.85,-.7) -- (-1.15,-.7) -- (-1.15,.7) -- (-.85,.7) -- (-.85,-.7);
 \draw[thick,red] (-.85,-.7) -- (-.85,.7);
 \draw (-.55,-.7) -- (-.55,.7);
 \draw (-.1,-.7) -- (-.1,.7);
 \draw (.5,-.7) -- (.5,-.4) arc (180:0:.15cm) -- (.8,-.7);
 \draw (.5,.7) -- (.5,.4) arc (-180:0:.15cm) -- (.8,.7);
 \node at (-.1,-.9) {\tiny{$n-1$}};
 \node at (-.55,.9) {\tiny{$2r$}};
\end{tikzpicture}
\,
=
\,
\begin{tikzpicture}[baseline=-.1cm]
 \fill[ctwoshading] (-.4,-.7) -- (-.7,-.7) -- (-.7,.7) -- (-.4,.7) -- (-.4,-.7);
 \draw[thick,red] (-.4,-.7) -- (-.4,.7);
 \draw (-.1,-.7) -- (-.1,.7);
 \draw (.5,-.7) -- (.5,-.4) arc (180:0:.15cm) -- (.8,-.7);
 \draw (.5,.7) -- (.5,.4) arc (-180:0:.15cm) -- (.8,.7);
 \node at (.55,0) {\tiny{$2r+n-1$}};
\end{tikzpicture}
=e_n
$
\item (Conditional Expectation) \label{condex} 
 $\displaystyle\Phi(\cE_n(x))=\Phi\left(
\begin{tikzpicture}[baseline=-.1cm]
 \draw (0,-.7) -- (0,.7);
 \draw (.15,.3) arc (180:0:.15cm) -- (.45,-.3) arc (0:-180:.15cm);
 \roundNbox{unshaded}{(0,0)}{.3}{0}{0}{$x$}
 \node at (0,-0.9) {\tiny{$n-1$}};
\end{tikzpicture}
  \right)\,
  =
\,
\begin{tikzpicture}[baseline=-.1cm]
 \fill[ctwoshading] (-.9,-.7) -- (-1.2,-.7) -- (-1.2,.7) -- (-.9,.7) -- (-.9,-.7);
 \draw[thick,red] (-.9,-.7) -- (-.9,.7);
 \draw (-.6,-.7) -- (-.6,.7);
 \draw (0,-.7) -- (0,.7);
 \draw (.15,.3) arc (180:0:.15cm) -- (.45,-.3) arc (0:-180:.15cm);
 \roundNbox{unshaded}{(0,0)}{.3}{0}{0}{$x$}
 \node at (0,-0.9) {\tiny{$n-1$}};
\end{tikzpicture}
\,
=
\,
\begin{tikzpicture}[baseline=-.1cm]
 \fill[ctwoshading] (-.9,-.7) -- (-1.2,-.7) -- (-1.2,.7) -- (-.9,.7) -- (-.9,-.7);
 \draw[thick,red] (-.9,-.7) -- (-.9,.7);
 \draw (-.6,-.7) -- (-.6,.7);
 \draw (0,-.7) -- (0,.7);
 \draw (.15,.3) -- (.15,.7) arc (180:0:.3cm) -- (.75,-.7) arc (0:-180:.3) -- (.15,-.3);
 \roundNbox{unshaded}{(0,0)}{.3}{0}{0}{$x$}
 \draw[dashed] (-1.4,-.65) -- (-1.4,.65) -- (.45,.65) -- (.45,-.65) -- (-1.4,-.65);
 \node at (-.2,-.9) {\tiny{$n-1$}};
 \node at (-.6,.9) {\tiny{$2r$}};
\end{tikzpicture}
=E_n(\Phi(x))
$
\item (Right Inclusion)
 $\displaystyle \iota_n(x)=i_n\left(
\begin{tikzpicture}[baseline=-.1cm]
 \draw (0,-.7) -- (0,0.7);
 \draw (0.45,-.7) -- (0.45,.7);
 \roundNbox{unshaded}{(0,0)}{.3}{0}{0}{$x$}
 \node at (0,-.9) {\tiny{$n$}};
 \node at (0.45,-.9) {\tiny{$1$}};
\end{tikzpicture}
\right)\,
=
\,
\begin{tikzpicture}[baseline=-.1cm]
 \fill[ctwoshading] (-.9,-.7) -- (-1.2,-.7) -- (-1.2,.7) -- (-.9,.7) -- (-.9,-.7);
 \draw[thick,red] (-.9,-.7) -- (-.9,.7);
 \draw (-.6,-.7) -- (-.6,.7);
 \draw (0,-.7) -- (0,0.7);
 \draw (0.45,-.7) -- (0.45,.7);
 \roundNbox{unshaded}{(0,0)}{.3}{0}{0}{$x$}
 \node at (0,-.9) {\tiny{$n$}};
 \node at (-.6,.9) {\tiny{$2r$}};
 \node at (0.45,-.9) {\tiny{$1$}};
\end{tikzpicture}
\,
=
\,
\begin{tikzpicture}[baseline=-.1cm]
 \fill[ctwoshading] (-.75,-.7) -- (-1.05,-.7) -- (-1.05,.7) -- (-.75,.7) -- (-.75,-.7);
 \draw[thick,red] (-.75,-.7) -- (-.75,.7);
 \draw (-.45,-.7) -- (-.45,.7);
 \draw (0,-.7) -- (0,0.7);
 \draw (0.6,-.7) -- (0.6,.7);
 \roundNbox{unshaded}{(0,0)}{.3}{0}{0}{$x$}
 \draw[dashed] (-1.25,-.65) -- (-1.25,.65) -- (.45,.65) -- (.45,-.65) -- (-1.25,-.65);
 \node at (0,-.9) {\tiny{$n$}};
 \node at (-.45,.9) {\tiny{$2r$}};
 \node at (0.6,-.9) {\tiny{$1$}};
\end{tikzpicture}
=i_n(\Phi(x))
$
\item (Left Capping) In \S\ref{sec:StronglyMarkovPA},
we discussed that the left-capping tangle in the canonical planar $\ast$-algebra is given for $n\geq 1$ by 
\[
E^{A_0'\cap A_{n}}_{A_1'\cap A_{n}}(x) = \frac{1}{d^2}\sum_{b}bxb^\ast
\]
where $\{b\}$ is a Pimnser Popa basis of $A_1$ over $A_0$. 
This means that, for any $x \in\cQ_{n,+}$, we have 
$$
    d^2E^{A_0'\cap A_{n}}_{A_1'\cap A_{n}}(\Phi(x))
    =
\sum_{b}
\begin{tikzpicture}[baseline=-.1cm, scale=.8]
 \fill[ctwoshading] (-1,-1.7) -- (-1,1.7) -- (-.6,1.7) -- (-.6,-1.7) -- (-1,-1.7);
 \fill[shaded] (0,1.7) -- (0,0.5) -- (0.3,0.5) -- (0.3,-0.5) -- (0,-0.5) -- (0,-1.7) -- (0.6,-1.7) -- (0.6,1.7) -- (0,1.7);
 \draw[thick,red] (-0.6,-1.7) -- (-0.6,1.7);
 \draw (0.6,-1.7) -- (0.6,1.7);
 \draw (0.3,-.5) -- (0.3,.5);
 \draw (0,-1.7) -- (0,1.7);
 \roundNbox{unshaded}{(0.45,0)}{.3}{0}{0}{$x$}
 \roundNbox{unshaded}{(0.15,0.8)}{.3}{0.6}{0}{$b$}
 \roundNbox{unshaded}{(0.15,-0.8)}{.3}{0.6}{0}{$b^\ast$}
 \node at (0.6,-2) {\tiny{$n-1$}};
 \node at (0,2) {\tiny{$2r+1$}};
 \node at (-.15,0) {\tiny{$2r$}};
\end{tikzpicture}
=
\sum_{b}
\begin{tikzpicture}[baseline=-.1cm, scale=.8]
 \fill[ctwoshading] (-1,-1.7) -- (-1,1.7) -- (-.6,1.7) -- (-.6,-1.7) -- (-1,-1.7);
 \fill[shaded] (0,1.7) -- (0,0.5) -- (0.3,0.5) -- (0.3,0.25)arc (-180:0:.15cm)-- (0.6,1.3)arc(180:0:.3cm) -- (1.2,-1.3) arc(0:-180:0.3cm) -- (0.6,-0.25) arc(0:180:0.15cm) -- (0.3,-0.5) -- (0,-.5) -- (0,-1.7) -- (1.5,-1.7) -- (1.5,1.7);
 \draw[thick,red] (-0.6,-1.7) -- (-0.6,1.7);
 \draw (0.3,.5)--(0.3,0.25) arc (-180:0:.15cm)-- (0.6,1.3) arc(180:0:.3cm) -- (1.2,.15);
 \draw (0.3,-.7)--(0.3,-0.25) arc (180:0:.15cm)-- (0.6,-1.3) arc(-180:0:.3cm) -- (1.2,-.15);
 \draw (0,-1.7) -- (0,1.7);
 \draw (1.5,-1.7) -- (1.5,1.7);
 \draw[dashed] (-1,1.3) --  (-1,-1.3) --  (0.8,-1.3) --  (0.8,1.3) -- (-1,1.3);
 \roundNbox{unshaded}{(1.35,0)}{.3}{0}{0}{$x$}
 \roundNbox{unshaded}{(0.15,0.8)}{.3}{0.6}{0}{$b$}
 \roundNbox{unshaded}{(0.15,-0.8)}{.3}{0.6}{0}{$b^\ast$}
 \node at (1.5,-2) {\tiny{$n-1$}};
 \node at (0,2) {\tiny{$2r+1$}};
 \node at (-.15,0) {\tiny{$2r$}};
\end{tikzpicture} 
=d\,
\begin{tikzpicture}[baseline=-.1cm, scale=.8]
 \fill[ctwoshading] (-1,-1.7) -- (-1,1.7) -- (-.6,1.7) -- (-.6,-1.7) -- (-1,-1.7);
 \fill[shaded] (-0.3,1.7) -- (1.2,1.7) -- (1.2,-1.7) -- (-0.3,-1.7) -- (-0.3,1.7);
 \filldraw[unshaded] (0.3,1.3) arc(180:0:.3cm) -- (0.9,-1.3)arc(0:-180:0.3cm) -- (0.3,1.3); 
 \draw[thick,red] (-0.6,-1.7) -- (-0.6,1.7);
 \draw (-0.3,-1.7) -- (-0.3,1.7);
 \draw (1.2,-1.7) -- (1.2,1.7);
 \draw[dashed] (-1,1.3) --  (-1,-1.3) --  (0.5,-1.3) --  (0.5,1.3) -- (-1,1.3);
 \roundNbox{unshaded}{(1.05,0)}{.3}{0}{0}{$x$}
 \node at (1.2,-2) {\tiny{$n-1$}};
 \node at (-0.3,2) {\tiny{$2r+1$}};
 \node at (0.2,0) {\tiny{$1$}};
\end{tikzpicture} 
=d^2\Phi\left(E^{\cQ_{n,+}}_{\cQ_{n-1,-}}(x)\right).
$$

\item (Left Inclusion) 
The left inclusion $l_{n} : \cP_{n,-} \to\cP_{n+1,+}$ is just the inclusion $A_1'\cap A_{n+1}\to A_0'\cap A_{n+1}$.
Graphically, $\ell_n:\cQ_{n,-} \to\cQ_{n+1,+}$ is equivalent to adding a string on the left. 
Thus, for $x\in \cQ_{n,-}$, we have that:
\begin{equation*}
l_n(\Phi(x))=
\begin{tikzpicture}[baseline=-.1cm]
 \fill[ctwoshading] (-.9,-.7) -- (-0.9,.7) -- (-1.3,0.7) -- (-1.3,-0.7) -- (-0.9,-0.7);
 \fill[shaded] (-0.45,-0.7) -- (-0.45,0.7) -- (0.15,0.7) -- (0.15,-0.7) -- (-0.45,-0.7);
 \draw[red, thick] (-0.9,-0.7) -- (-0.9,0.7);
 \draw (-0.45,-0.7) -- (-0.45,0.7);
 \draw (0.15,-0.7) -- (0.15,0.7);
 \roundNbox{unshaded}{(0.15,0)}{.3}{0}{0}{$x$};
 \node at (0.15,0.8){\tiny{$n$}};
 \node at (-0.45, -0.9){\tiny{$2r+1$}};
\end{tikzpicture}
\,
=
\,
\begin{tikzpicture}[baseline=-.1cm]
 \fill[ctwoshading] (-.9,-.7) -- (-0.9,.7) -- (-1.3,0.7) -- (-1.3,-0.7) -- (-0.9,-0.7);
 \fill[shaded] (0.15,-0.7) -- (0.15,0.7) -- (0.6,0.7) -- (0.6,-0.7) -- (0.15,-0.7);
 \draw[red, thick] (-0.9,-0.7) -- (-0.9,0.7);
 \draw (-0.45,-0.7) -- (-0.45,0.7);
 \draw (0.6,-0.7) -- (0.6,0.7);
 \draw (0.15,-0.7) -- (0.15,0.7);
 \roundNbox{unshaded}{(.6,0)}{0.3}{0}{0}{$x$};
 \draw[dashed] (-0.15,-1.1) -- (-0.15,1.1) -- (1.05,1.1) --  (1.05,-1.1) -- (-0.15,-1.1);
 \node at (0.6,0.8){\tiny{$n$}};
 \node at (-0.45, -0.9){\tiny{$2r$}};
 \node at (0.15,-0.9){\tiny{1}};
\end{tikzpicture}
=\Phi(\overline{\ell_{n}}(x)).
\qedhere
\end{equation*}
\end{itemize}
\end{proof}

We have checked that $\Phi:\cQ_{\bullet}\to\cP_{\bullet}$ is a planar $\dag$-algebra inclusion. 
Let $\cG_\bullet$ be the bipartite graph planar algebra of the fusion graph of $X$ acting on $\cM$. 
Then by  \cite[Th,.~3.33]{MR2812459}, 
we know that $\cP_\bullet$ is a planar $\dag$-algebra isomorphic to $\cG_\bullet$. 
Thus, we have an embedding of $\cQ_\bullet$ into $\cG_\bullet$.

\begin{cor}[The Embedding Theorem]
 \label{cor:EmbeddingTheorem}
 A finite depth subfactor planar algebra $\cQ_\bullet$ can be embedded into the bipartite graph planar algebra of the fusion graph of a connected right planar $\cQ_\bullet$-module.
\end{cor}

In particular, by considering $(\cQ_\bullet,1_{0,+})$ as a connected right planar $\cQ_\bullet$-module, we recover the main result of \cite{MR2812459}. 
By instead considering $(\cQ_\bullet,Y)$ for $Y$ a simple summand of the shaded-unshaded strand $\overline{X}\in \cQ_{1,-}$, we obtain an embedding of $\cQ_\bullet$ into the graph planar algebra for the dual principal graph.

\subsection{Invariance of the embedding}
\label{sec:InvarianceOfEmbedding}

As the observant reader may have noted, we made \emph{three} choices in defining the embedding map from $\cQ_\bullet \hookrightarrow\cG\cP\cA(\Gamma)_\bullet$.
First, we chose a simple object $m\in \cM$ to get our Markov tower $M_n := \End_\cM(m \vartriangleleft X^{\text{alt}\otimes n})$, and second, we chose $r\geq 0$ such that the inclusion $M_{2r}\subseteq (M_{2r+1}, \tr_{2r+1})\subseteq (M_{2r+2}, \tr_{2r+2}, e_{2r+1})$ is standard.
Third, we chose a basis for the strongly Markov inclusion $M_{2r}\subseteq (M_{2r+1}, \tr_{2r+1})$ to obtain a planar $\dag$-algebra isomorphism from the canonical relative commutant planar algebra $\cP_\bullet$ of the inclusion to the graph planar algebra $\cG\cP\cA(\Gamma)_\bullet$ of the fusion graph $\Gamma$.
In this section, we show that the embedding does not depend on these choices up to a $\dag$-automorphism of $\cG\cP\cA(\Gamma)_\bullet$.

\begin{defn}
Suppose $\cQ_\bullet$ is a subfactor planar algebra and $\cP_\bullet, \cP_\bullet'$ are two unitary shaded planar algebras together with planar algebra embeddings $\Phi: \cQ_\bullet \hookrightarrow \cP_\bullet$ and $\Phi': \cQ_\bullet \hookrightarrow \cP_\bullet'$.
We say the embeddings $\Phi$ and $\Phi'$ are \emph{equivalent} if there is a planar $\dag$-algebra isomorphism $\Psi: \cP_\bullet \to \cP_\bullet'$ such that the following diagram commutes:
$$
  \begin{tikzcd}
    Q_{\bullet} \arrow[hook]{r}{\Phi} \arrow[swap, hook]{dr}{\Phi'} & P_{\bullet} \arrow{d}{\Psi} 
    \\
     & P_{\bullet}'
  \end{tikzcd}
$$  
\end{defn}

We now treat our three choices for our embedding in reverse order.
First, note that choosing a different basis for the inclusion just alters the isomorphism $\cP_\bullet \cong \cG\cP\cA(\Gamma)_\bullet$ by a $\dag$-automorphism of $\cG\cP\cA(\Gamma)_\bullet$, resulting in equivalent embeddings.

Second, suppose we chose a different $r'\geq 0$ such that the inclusion $M_{2r'}\subseteq (M_{2r'+1}, \tr_{2r'+1})\subseteq (M_{2r'+2}, \tr_{2r'+2}, e_{2r'+1})$ is standard.
Without loss of generality, we may assume $r' = r+k$ for $k\in \bbN$.
Denoting the canonical relative commutant planar algebra for the strongly Markov inclusion $M_{2r+2k}\subseteq (M_{2r+2k+1}, \tr_{2r+2k+1})$ by $\cP_\bullet'$, we see that $\cP_\bullet'\cong \cP_\bullet$ by iteratively applying the shift-by-2 planar algebra isomorphism from Theorem \ref{ShiftIso}. 
Hence, replacing $r$ with $r'$ results in an equivalent embedding.

Third, suppose we chose the simple object $n\in \cM_0 \subset \cM$ instead of $m$, where $\cM_0 = \cM\vartriangleleft 1_0$.
For $i\geq 0$, define $N_i := \End_\cM(n \vartriangleleft X^{\text{alt}\otimes i})$.
Since $\cM$ is indecomposable as a right $\cC$-module category, there is a $j>0$ such that $n$ is a subobject of $m \vartriangleleft X^{\text{alt}\otimes 2j}$.
Fix an orthogonal projection $p \in M_{2j} = \End_\cM(m \vartriangleleft X^{\text{alt}\otimes 2j})$ with image isomorphic to $n$.
Notice that the compressed shifted Markov tower $(pM_{2j+k}p, \tr_{2j+k}^p, pe_{2j+k+1})_{k\geq 0}$ is $*$-isomorphic to the Markov tower $(N_{k}, \tr_k, f_{k+1})_{k\geq 0}$, where we denote by $e_i$ the Jones projections in $M_\bullet $ and by $f_i$ the Jones projections in $N_\bullet$.
Again, since $\cM$ is indecomposable, we may fix $k>0$ sufficiently large such that the following three conditions hold:
\begin{enumerate}[label={\rm(\arabic*)}]
\item
 The projection $p\vartriangleleft \id_{X^{\text{alt}\otimes 2k}}$ has central support 1 in the finite dimensional von Neumann algebra $M_{2(j+k)} = \End_\cM(m \vartriangleleft X^{\text{alt}\otimes 2(j+k)})$, which is equivalent to $M_{2(j+k)} = M_{2(j+k)}p M_{2(j+k)}$ by finite dimensionality.
\item
Setting $r:= 2(j+k)$, the inclusion $M_{2r} \subseteq (M_{2r+1}, \tr_{2r+1}) \subseteq (M_{2r+2}, \tr_{2r+2}, e_{2r+1})$ is standard.
\item
The inclusion $N_{2k} \subseteq (N_{2k+1}, \tr_{2k+1}) \subseteq (N_{2k+2}, \tr_{2k+2}, e_{2k+1})$ is standard.
\end{enumerate}
Now, by Theorem \ref{thm:CompresionIsomorphsim}, compressing $M_\bullet$ by $p\vartriangleleft \id_{X^{\text{alt}\otimes 2k}} \in M_{2r}$ gives a planar $\dag$-algebra isomorphism from the canonical relative commutant planar algebra $\cP_\bullet$ of the strongly Markov inclusion $M_{2r} \subseteq (M_{2r+1}, \tr_{2r+1})$ to the
canonical relative commutant planar algebra $\cP^p_\bullet$ of the strongly Markov inclusion $pM_{2r}p \subseteq (pM_{2r+1}p, \tr^p_{2r+1})$, where $\tr^p_{2r+1}$ is defined analogously to \eqref{eq:CompressedTrace1}, which, in turn, is $\dag$-isomorphic to the canonical relative commutant planar algebra $\cR_\bullet$ of the strongly Markov inclusion $N_{2k} \subseteq (N_{2k+1}, \tr_{2k+1})$.
Hence, replacing $m$ with $n$ results in an equivalent embedding.

We have just proved the following.

\begin{prop}
The embedding $\cQ_\bullet \hookrightarrow \cG\cP\cA(\Gamma)_\bullet$ is invariant under our choices, up to equivalence.
\end{prop}

\bibliographystyle{amsalpha}
{\footnotesize{
\bibliography{embedding_theorem}
}}
\end{document}